\documentclass{cocv}

\usepackage{amsfonts}
\usepackage{amsmath}
\usepackage{graphicx}
\usepackage{epstopdf}
\usepackage{algorithm}
\usepackage{algorithmic}
\numberwithin{equation}{section}
\usepackage{float}
\usepackage{diagbox}

\usepackage{hyperref}
\usepackage[capitalise]{cleveref}

\usepackage{subfigure}
\usepackage{dsfont}
\usepackage{enumitem}

\newcommand{\sgn}{\operatorname{sgn}}
\newcommand{\range}{\operatorname{Range}}
\newcommand{\closure}{\operatorname{cl}}
\newcommand{\var}{\operatorname{var}}
\newcommand{\dd}{\,\mathrm{d}}

\newcommand{\EE}{\mathcal{E}}

\newcommand{\LL}{\mathcal{L}}

\newcommand{\PP}{\mathcal{P}}

\newcommand{\RR}{\mathcal{R}}

\newcommand{\TT}{\mathcal{T}}
\newcommand{\UU}{\mathcal{U}}

\newcommand{\N}{\mathbb{N}}

\newcommand{\R}{\mathbb{R}}

\newcommand{\weakly}{\rightharpoonup}

\DeclareFontFamily{OT1}{pzc}{}
\DeclareFontShape{OT1}{pzc}{m}{it}{<-> s * [1.175] pzcmi7t}{}
\DeclareMathAlphabet{\mathpzc}{OT1}{pzc}{m}{it}

\newcommand\qq{\mathpzc{q}}
\newcommand\pp{\mathpzc{p}}
\newcommand\uu{\mathpzc{u}}
\newcommand\vv{\mathpzc{v}}
\newcommand\zz{\mathpzc{z}}

\newcommand{\Aone}{\hyperref[Aone]{\textup{(A1)}}}
\newcommand{\Atwo}{\hyperref[Atwo]{\textup{(A2)}}}

\newcommand{\ExA}{\hyperref[ExampleA]{\textup{1}}}
\newcommand{\ExB}{\hyperref[ExampleB]{\textup{2}}}

\newtheorem{theorem}{Theorem}[section]
\newtheorem{corollary}[theorem]{Corollary}
\newtheorem{lemma}[theorem]{Lemma}
\newtheorem{assumption}[theorem]{Assumption}

\newtheorem{proposition}[theorem]{Proposition}
\newtheorem{definition}[theorem]{Definition}
\newtheorem{remark}[theorem]{Remark}
\newtheorem{Algorithm}[theorem]{Algorithm}

\usepackage{xcolor}

\begin{document}


\title{On the identification and optimization of nonsmooth superposition operators in semilinear elliptic {PDEs}}

\runningtitle{Optimization of nonsmooth superposition operators}


\author{Constantin Christof}\address{Technische Universit\"{a}t M\"{u}nchen,
Chair of Optimal Control,
School of Computation,
Information and Technology,
Department of Mathematics,
Boltzmannstraße 3,
85748 Garching, Germany, \email{christof@cit.tum.de, julia.kowalczyk@cit.tum.de}}
\author{Julia Kowalczyk}\sameaddress{1}

%
%

\begin{abstract}
We study an infinite-dimensional optimization problem that aims 
to identify the Nemytskii operator in the nonlinear part of a
prototypical semilinear elliptic partial differential equation (PDE) which 
minimizes the distance between the PDE-solution and a given desired state.
In contrast to previous works, we 
consider this identification problem in a low-regularity regime in which the 
function inducing the Nemytskii operator is a-priori only known to be 
an element of $H^1_{loc}(\R)$. This makes the studied problem class a
suitable point of departure for the rigorous analysis of training problems for learning-informed PDEs 
in which an unknown superposition operator is approximated 
by means of a neural network with nonsmooth activation functions (ReLU, leaky-ReLU, etc.).
We establish that, despite the low regularity of the controls, 
it is possible to derive a classical 
stationarity system for local minimizers and to solve the considered problem by means of 
a gradient projection method.
The convergence of the resulting algorithm is proven in the function space setting. 
It is also shown that the established first-order necessary optimality conditions imply that 
locally optimal superposition operators share various characteristic properties 
with commonly used activation functions: They are always sigmoidal, 
continuously  differentiable away from the origin, and typically possess a distinct kink 
at zero. The paper concludes with numerical experiments which 
confirm the theoretical findings.
 
\end{abstract}


\maketitle

\section{Introduction}
\label{sec:1}
This paper is concerned with the analysis and numerical solution
of optimization problems of the following type:
\begin{equation}
\label{eq:P_g_prime}
\left.
	\begin{aligned}
		\text{Minimize} 
		\quad & \frac12 \| y - y_D\|_{L^2(\Omega)}^2 + \nu_1 \|g'\|_{L^1(\R)} + \frac{\nu_2}{2} \|g' - u_D\|_{L^2(\R)}^2 \\
        \text{w.r.t.}
        \quad &y \in H^1(\Omega), \quad  g \in H^1_{loc}(\R),\\
		\text{s.t.} \quad & -\Delta y + g(y) = f \text{ in }\Omega,\quad y = 0 \text{ on }\partial \Omega,
		\\
		 \text{and}\quad  & g' \geq 0 \text{ a.e.\  in } \R,
		 		\\
		 \text{and}\quad  & g(0) = 0. 
	\end{aligned}
	\right \}
\end{equation}

Here, $\Omega \subset \R^d$, $d \in \{1,2,3\}$, is a bounded domain
which is convex or possesses a $C^{1,1}$-boundary;
$y_D \in L^2(\Omega)$, $f \in L^2(\Omega) \setminus\{0\}$, and $u_D \in L^2(\R)$ are given functions;
$g'$ denotes the weak derivative of $g$;
$\nu_1 \geq 0$ and $\nu_2 > 0$ are regularization parameters;
$\Delta$ denotes the Laplace operator; and the objective function is to be understood as an extended 
real-valued function on $H^1(\Omega) \times H^1_{loc}(\R)$ with values in $[0,\infty]$.
For the precise assumptions on the quantities in \eqref{eq:P_g_prime} and 
the definitions of the occurring spaces etc.,
we refer the reader to \cref{subsec:2.2,subsec:2.1}. 
Note that, by introducing the variable $u := g'$, the problem \eqref{eq:P_g_prime} 
can also be formulated in the following, more convenient form:
\begin{equation}
\label{eq:P_g_prime_2}
\left.~~
	\begin{aligned}
		\text{Minimize} 
		\quad &  \frac12 \| y - y_D\|_{L^2(\Omega)}^2 + \nu_1 \|u\|_{L^1(\R)} + \frac{\nu_2}{2} \|u - u_D\|_{L^2(\R)}^2 \\
        \text{w.r.t.}
        \quad &y \in H^1(\Omega), \quad  u \in L^2(\R),\\
		\text{s.t.} \quad & -\Delta y + g_u(y) = f \text{ in }\Omega,\quad y = 0 \text{ on }\partial \Omega,
		\\
		 \text{and} \quad  & g_u(t) = \int_0^t u(s) \dd s \quad \forall t \in \R,
		 \\
		 \text{and}\quad  & u \geq 0 \text{ a.e.\  in } \R.
	\end{aligned}
\right \}
\end{equation}
The main results of this work -- \cref{th:mainBV_P,th:truncate,th:BouligandStationarityP,th:Sregulated,th:PDStationarityP,th:GradMain} -- establish improved regularity properties for
local solutions of \eqref{eq:P_g_prime} and \eqref{eq:P_g_prime_2}, 
first-order necessary optimality conditions of Bouligand- and primal-dual type, and 
convergence results for a gradient projection algorithm 
that makes it possible to solve \eqref{eq:P_g_prime} and \eqref{eq:P_g_prime_2} in 
a function space setting.
For a more detailed overview of our findings, see \cref{subsec:1.2} below. 

\subsection{Motivation, Background, and Relation to Prior Work}
\label{subsec:1.1}

Before we present in detail what we prove for the local minimizers of
\eqref{eq:P_g_prime} and \eqref{eq:P_g_prime_2}
in this paper, let us briefly discuss why it is worthwhile to consider these problems:
Our main motivation 
for the study of \eqref{eq:P_g_prime} is that this minimization problem can be interpreted as the 
continuous version of a training problem for a learning-informed 
partial differential equation (PDE)
which aims to approximate an unknown nonlinearity in a semilinear model
by means of an artificial neural network. 
Such problems have received considerable attention in recent years --
see, e.g., \cite{Aarset2023,Brunton2016,Court2022,Dong2022-1,Dong2022-2,Dong2022-3,Kaltenbacher2022,Qin2019,Rudy2017} -- and 
typically take the following form  if an all-at-once 
approach with a single right-hand side $f$ is used for the network training
and the same PDE as in \eqref{eq:P_g_prime} is studied:
\begin{equation}
\label{eq:NNP}
\left.
	\begin{aligned}
		\text{Minimize} 
		\quad & J_1(y) + J_2(\alpha)\\
        \text{w.r.t.}
        \quad &y \in H^1(\Omega), \quad  \alpha \in D,\\
		\text{s.t.} \quad & -\Delta y + \psi(\alpha, y) = f \text{ in }\Omega,\quad y = 0 \text{ on }\partial \Omega,
		\\
		 \text{and}\quad  & \alpha \in D_{ad}.
	\end{aligned}
	~\right \}
\end{equation}
Here, $J_1$ is a fidelity term that measures
the distance between the PDE-solution $y$ and some given data either in the whole of $\Omega$ or in certain 
parts of it;
$\psi\colon D \times \R \to \R$ is a neural network with parameter vector $\alpha$
and parameter space $D$ (defined, e.g., as in \cite[Section 2]{Christof2022}); $D_{ad}$ is an
appropriately chosen subset of $D$; and $J_2\colon D \to \R$ is a regularization term 
acting on $\alpha$. 
See, for example, \cite[Section 4]{Court2022},
\cite[Section 2]{Kaltenbacher2022},
and \cite[Section~3]{Qin2019} for particular instances of problems of the type \eqref{eq:NNP}. 
What our model problem \eqref{eq:P_g_prime} now represents is --
for a prototypical PDE and a special choice of the fidelity and regularization term that will be discussed in 
more detail in \cref{subsec:2.3} -- 
the continuous optimization problem that is obtained when the 
Nemytskii operator $g$ is not assumed to be given in the form of a neural 
network but allowed to vary freely among all elements of $H^{1}_{loc}(\R)$
that are nondecreasing and zero at the origin. (Note that the former of these conditions is required to ensure the 
unique solvability of the state-equation; see \cref{sec:3} below.) In other words, \eqref{eq:P_g_prime}
is the continuous optimization problem that is present \emph{before} a 
specific ansatz 
for the unknown nonlinearity $g$ in the governing PDE -- in the form of a 
neural network of a certain architecture or a
comparable approximation instrument with a finite number of parameters -- is decided on. 
This makes  \eqref{eq:P_g_prime} an appropriate point of departure 
for the rigorous analysis and comparison of different approaches for the approximation of the 
unknown Nemytskii operator $g$ in the PDE $-\Delta y + g(y) = f$:
Its solutions are precisely the limit objects that approximate solutions obtained 
from a discretization of \eqref{eq:P_g_prime} by means of 
a neural network ansatz or a comparable approach should converge to 
when the number of degrees of freedom in the used discretization scheme
goes to infinity.  In future work, we plan to use the results  established
in this paper, for instance, for the derivation of a-priori 
error estimates for finite-element-based discretizations of \eqref{eq:P_g_prime}, 
for the comparison of finite-element- and neural-network-based approaches 
for the numerical solution of \eqref{eq:P_g_prime}, and for the 
development of hybrid methods which combine machine-learning components and 
more classical instruments to achieve an accurate, fast, and reliable identification of 
unknown nonlinear terms in partial differential equations.

We would like to emphasize at this point that the idea to study
infinite-dimensional optimization problems in which an 
optimization variable enters a governing partial differential equation or variational inequality via superposition
is not new. 
In the context of the heat equation, for example, this type of problem has already been 
considered
for the calculation of unknown heat transfer laws  in 
\cite{Roesch1992,Roesch1994,Roesch1996,Roesch1996-2,Roesch1998,Roesch2002,Hao2015}
and in the context of damage evolution in quasistatic elasticity in the recent \cite{Grutzner2021}.
For seminal works on the identification of Nemytskii operators in 
partial and ordinary differential equations, see also 
\cite{Barbu1994,Barbu1996}, which analyze penalty approximation schemes for elliptic PDEs
with convex analysis techniques, 
and \cite{Goebel1998}, which is concerned with a classical boundary value problem in one dimension.
Moreover, the task of identifying nonlinear model components 
has, of course, also been studied quite extensively in the field of inverse problems, 
e.g., in relation to stability issues and iterative 
recovery algorithms;
see \cite{Kaltenbacher2020,Kaltenbacher2022,Kian2023} and the references therein. 

What distinguishes our analysis from these earlier contributions is that, 
in our model problem \eqref{eq:P_g_prime}
and its reformulation \eqref{eq:P_g_prime_2}, 
the function $g = g_u$ inducing the Nemytskii operator in the 
 governing PDE is a-priori only known to be in $H_{loc}^1(\R)$
 and, thus, merely Hölder continuous with exponent $1/2$. This nonsmoothness 
 causes the control-to-state mapping $L_+^2(\R) \ni u \mapsto y \in H_0^1(\Omega)$ to lack the
 property of directional differentiability and significantly complicates the derivation 
 of first-order optimality conditions that may serve as a starting point for the development 
 of  numerical solution algorithms -- in particular as in \eqref{eq:P_g_prime} one also has to deal with the 
  constraint $u = g' \geq 0$ a.e.\ in $\R$ which has to be imposed 
 to ensure the well-posedness of the state-equation. At least to 
the best of our knowledge, the derivation of, for instance, primal-dual
stationarity systems for local minimizers  in such a low-regularity 
 regime has not been accomplished so far in the literature.
 Compare, e.g., with \cite{Hao2015,Kian2023,Grutzner2021,Roesch1994,Roesch2002,Roesch1992} 
 in this context which all assume (at least) $C^1$-regularity of the involved superposition operators. 
In view of the applications in the field of learning-informed PDEs outlined at the beginning 
of this section and the fact that various of the most popular activation functions in machine learning
are nonsmooth (ReLU, leaky-ReLU, etc.), the requirement of continuous differentiability  is very  unsatisfactory.
 The only contributions we are aware of 
that contain results which go in a similar direction as our analysis
 are \cite{Goebel1998}, \cite{Roesch1996}, and \cite{Court2022}:
In the first of these papers, \cite{Goebel1998}, 
the author is able to establish a variant of a necessary optimality condition in terms 
of Clarke's generalized directional derivative for the identification of a nonsmooth 
Nemytskii operator in an ordinary differential equation that only requires Lipschitz continuity;
see  \cite[Theorem~3]{Goebel1998}. 
This necessary condition, however, is not amenable to any form of adjoint calculus and thus only of limited use;
see the comments in \cite[Section 3.3]{Goebel1998}.
It is moreover easy to check that the approach of \cite{Goebel1998} does not carry over to the PDE-setting.
Second, in \cite[Theorem 2]{Roesch1996}, a Fréchet differentiability result is stated for the 
function that maps a merely Lipschitz continuous Nemytskii operator in a nonlinear heat transfer law 
to the solution of the corresponding heat equation. This theorem requires rather restrictive assumptions as, 
for instance, that the right-hand side of the heat equation is zero, that the considered Nemytskii operator 
only possesses a finite number of jumps in its derivative, and that the boundary temperature of 
the solution of the heat equation is strictly decreasing in time. Lastly, in 
\cite{Court2022}, a necessary optimality condition for training problems of the type \eqref{eq:NNP} 
involving neural networks with nonsmooth activation functions has recently 
been derived by means of a regularization approach and techniques of \cite{Christof2018nonsmoothPDE}; 
see \cite[Theorem 6.5]{Court2022}.
It should be noted that, even for this special case, where the Nemytskii operator in the governing PDE is
known to have a certain structure (namely that of a neural network), the derivation of stationarity systems is
not straightforward as passages to the limit in compositions of the type $g_n'(y_n)$ are 
hard to handle unless convergences in very strong spaces are available
for the sequences $\{g_n'\}$ and $\{y_n\}$; cf.\  \cite[Remark 6.6]{Court2022}. In the present work, 
we are able to circumvent these problems by exploiting certain regularity properties of the 
solutions of \eqref{eq:P_g_prime} and Stampacchia's lemma. 

\subsection{Summary of Main Results}
\label{subsec:1.2}
The first main result of this work -- \cref{th:mainBV_P} -- is concerned with structural properties of local solutions 
of \eqref{eq:P_g_prime_2}. It establishes that, if $u_D|_{[-r,r]} \in G[-r,r]$ (resp., $u_D|_{[-r,r]} \in BV[-r,r]$) 
holds for a  representative of the desired control $u_D \in L^2(\R)$ 
and a number $r>0$
that is larger than a computable lower bound $r_P$,
then 
every local minimizer $\bar u \in L^2(\R)$ of \eqref{eq:P_g_prime_2} 
possesses a representative which satisfies 
\begin{equation}
\label{eq:BVreg1}
\bar u = \max\left ( 0, u_D - \frac{\nu_1}{\nu_2} \right ) \text{ a.e.\ in } \R \setminus (-r,r)
\quad 
\text{and}
\quad 
\bar u|_{[-r,r]} \in G[-r,r]
\text{ (resp., $\bar u|_{[-r,r]} \in BV[-r,r]$)}.
\end{equation}
Here, $G[-r,r]$ and $BV[-r, r]$  denote the spaces of regulated functions and functions 
of bounded variation on $[-r, r]$, respectively, 
and a vertical bar denotes a restriction.
The properties in \eqref{eq:BVreg1} are of fundamental importance for the 
analysis of the problem \eqref{eq:P_g_prime_2} 
for two reasons: First of all, the fact that all local minimizers of \eqref{eq:P_g_prime_2} 
are determined completely by $u_D$, $\nu_1$, and $\nu_2$ on $\R \setminus (-r, r)$ allows us to reformulate 
\eqref{eq:P_g_prime_2} as a minimization problem in $L^2(-r, r)$; see \cref{th:truncate}. 
This makes it possible to apply standard discretization techniques to solve  \eqref{eq:P_g_prime_2}  
numerically; see \cref{sec:8}. Second, the $G[-r,r]$-regularity in \eqref{eq:BVreg1} 
implies
that every local minimizer $\bar u$ of \eqref{eq:P_g_prime_2} admits a representative
which has well-defined left and right limits on $[-r,r]$.
As the function $g_{\bar u} \in H^1_{loc}(\R)$ arises from $\bar u$ 
by integration, this entails that the map $g_{\bar u}\colon \R \to \R$
is directionally differentiable on $[-r,r]$
and that the control-to-state operator $L_+^2(\R) \ni u \mapsto y \in H_0^1(\Omega)$  of \eqref{eq:P_g_prime_2}
possesses a directional derivative at all local solutions $\bar u$ of \eqref{eq:P_g_prime_2}; see \cref{cor:HadDirDifMin}.
Although the control-to-state mapping $u \mapsto y$ of \eqref{eq:P_g_prime_2}
does not possess a directional derivative everywhere, 
it thus does so at all points that are relevant for the derivation of first-order necessary optimality conditions. 
By exploiting this effect, we are able to 
state a standard Bouligand stationarity condition for \eqref{eq:P_g_prime_2}.
This condition can be found in our second main result -- \cref{th:BouligandStationarityP}. 

Due to its implicit nature, the concept of Bouligand stationarity is typically not very useful in practice. 
 To arrive at a more tangible 
stationarity system, in 
\cref{sec:7,sec:8,sec:9}, we focus on
special instances of the problem \eqref{eq:P_g_prime_2}, namely those
for which the 
level sets $\{ f = c\}$ of the given right-hand side $f \in L^2(\Omega)$ (defined up to sets
of measure zero)
have measure zero for all $c \in \R$.
Our third main result  -- \cref{th:Sregulated} -- 
establishes that, under this assumption on $f$, 
the control-to-state mapping  $u \mapsto y$ of \eqref{eq:P_g_prime_2} is 
not only directionally differentiable but even continuously Fréchet differentiable on 
the space of regulated functions.
In combination with \eqref{eq:BVreg1}, this differentiability property allows us to derive
 the following explicit primal-dual first-order necessary optimality condition 
for local minimizers $\bar u$ of \eqref{eq:P_g_prime_2}; see \cref{th:PDStationarityP}:
\begin{equation}
\label{eq:statsys2}
\begin{aligned}
 &- \Delta \bar y + g_{\bar u}(\bar y) = f \text{ in }\Omega,\quad \bar y = 0 \text{ on }\partial \Omega,
&&\qquad- \Delta \bar p_1 + \bar u(\bar y) \bar p_1 = \bar y - y_D\text{ in }\Omega,\quad \bar p_1 = 0 \text{ on }\partial \Omega,
  \\
  &\bar p_2(s) =
  \begin{cases}
  \displaystyle
  \int_{\{\bar y \geq s\}}\bar p_1 \dd x  & \text{  f.a.a.\ } s \geq 0,\vspace{0.1cm}
  \\
   \displaystyle
    - \int_{\{\bar y \leq s\}}\bar p_1 \dd x  & \text{  f.a.a.\ } s < 0,
  \end{cases}
&&\qquad\bar u = \max\left ( 0, u_D + \frac{\bar p_2 - \nu_1}{\nu_2} \right )
  ~~\text{ a.e.\ in } \R. 
\end{aligned}
\end{equation}
Here, $\bar y$ denotes the state associated with $\bar u$,
the functions
$\bar p_1 \in H_0^1(\Omega) \cap H^2(\Omega)$ and $\bar p_2 \in L^2(\R)$ 
are adjoint variables, and the composition 
$\bar u(\bar y)$ is defined in a suitable sense; see \cref{def:EEuReal}. 
 As we will see in \cref{sec:7}, for sufficiently smooth $u_D$,
the system \eqref{eq:statsys2} implies (among other things) that
 local solutions $\bar u$ of \eqref{eq:P_g_prime_2}
give rise to maps $g_{\bar u}\colon \R \to \R$ that 
share various characteristic properties with commonly used activation functions from machine learning:
The resulting $g_{\bar u}$ are always sigmoidal, 
continuously  differentiable away from the origin, and typically possess a distinct kink at zero.
We expect that this observation is of particular interest in the context of 
the learning-informed PDEs in \eqref{eq:NNP}. 
In \cref{sec:8}, we moreover show that the continuous 
Fréchet differentiability of the control-to-state map established in 
\cref{th:Sregulated} makes it possible to set up a gradient projection method for the numerical 
solution of \eqref{eq:P_g_prime_2} and
 to analyze the convergence of the resulting algorithm 
in  function space. For the main result on this topic, see \cref{th:GradMain}.\pagebreak
 
\subsection{Structure of the Remainder of the Paper}
\label{subsec:1.3}
We conclude this introduction with an overview of the content and the structure of the remainder of the paper:
\Cref{sec:2} is concerned with preliminaries. Here, we clarify the notation, state our standing assumptions, and 
briefly comment on the choice of the regularization term in the problems \eqref{eq:P_g_prime} and \eqref{eq:P_g_prime_2};
see \cref{subsec:2.3}.
In \cref{sec:3}, we prove basic results on the 
control-to-state mapping $u \mapsto y$
 of \eqref{eq:P_g_prime_2} that are needed for our analysis; see \cref{th:propS}.
\Cref{sec:4} addresses the consequences that the properties in \cref{sec:3} have
for the problem \eqref{eq:P_g_prime_2}.
This section in particular 
 shows that \eqref{eq:P_g_prime_2} indeed admits an equivalent reformulation 
 as an optimization problem in $L^2(-r, r)$ for all large enough $r$; cf.\ the comments on this topic in \cref{subsec:1.2}.
 In \cref{sec:5}, we establish the regularity result in \eqref{eq:BVreg1} 
 for the local minimizers of  \eqref{eq:P_g_prime_2}; see \cref{th:mainBV_P}.
 The proof of this key theorem relies on a relaxation of the nonnegativity constraint in \eqref{eq:P_g_prime_2}
 and a careful limit analysis.
 \Cref{sec:6} demonstrates that the control-to-state operator of \eqref{eq:P_g_prime_2}
 is directionally differentiable and -- under the additional assumption 
 on $f$ mentioned in \cref{subsec:1.2} -- even continuously Fréchet differentiable on the space of regulated functions;
 see \cref{th:Sregulated}. In \cref{sec:7}, we use the  differentiability properties
 from \cref{sec:6}
 and the regularity results from \cref{sec:5}  to 
 establish a Bouligand stationarity condition and 
 the primal-dual system \eqref{eq:statsys2} for the problem \eqref{eq:P_g_prime_2}. 
 \Cref{sec:8} is concerned with the solution of \eqref{eq:P_g_prime_2}
 by means of a gradient projection method. Here, we prove that such an algorithm 
 terminates after finitely many iterations with an $\epsilon$-stationary point. 
 In \cref{sec:9}, we conclude the paper with numerical 
 experiments which confirm that the algorithm 
 developed in \cref{sec:8} is also usable in practice. 
 
\section{Problem Setting and Notation}
\label{sec:2}
In this section, we introduce basic notation and discuss the assumptions on the quantities in 
the problems \eqref{eq:P_g_prime} and \eqref{eq:P_g_prime_2} that are supposed to hold throughout our analysis.

\subsection{Basic Notation}
\label{subsec:2.1}

In what follows, we denote norms, scalar products, and dual pairings 
by the symbols $\|\cdot\|$, $(\cdot, \cdot)$, and $\langle \cdot, \cdot \rangle$,
respectively, equipped with a subscript that clarifies the considered space. 
Given a normed space $(U, \|\cdot\|_U)$, 
we define $B_\rho^U(u)$ to be the closed ball of radius $\rho>0$ in $U$ centered at $u \in U$.
The closure of a set in $U$ is denoted by $\closure_U(\cdot)$.
If a convex closed nonempty set $D \subset U$ and a point $u \in D$ are given, 
then we define $\RR(u;D) := \{s(v - u) \mid s > 0, v \in D \}$   
to be the radial cone to $D$ at $u$
and $\TT(u;D) := \closure_U \left ( \RR(u;D) \right )$
to be the tangent cone to $D$ at $u$; cf.\
\cite[Section 2.2.4]{BonnansShapiro2000}.
For the space of linear and continuous functions $H\colon U \to Y$ between two normed 
spaces $(U, \|\cdot\|_U)$ and $(Y, \|\cdot\|_Y)$, we use the symbol $\LL(U,Y)$.
In the special case $Y = \R$, we define $U^* := \LL(U,\R)$ to be the dual space of $U$. 
If $U$ is continuously embedded into $Y$
(in the sense of \cite[Definition~4.19]{Schweizer2013}), then we write $U \hookrightarrow Y$.
To denote the strong (respectively, weak) convergence of a sequence 
in a normed space, we use the symbol   $\to$ (respectively, $\weakly$). 
Given an arbitrary nonempty set $D \subset U$
and a function $H\colon U \to Y$, we define 
$\range(H)$ to be the range of $H$,
$H|_{D}$ to be the restriction of $H$ to $D$,
and $\mathds{1}_D\colon U \to \R$ to be the function that is equal to one on $D$ and 
equal to zero on $U \setminus D$. When talking about the level sets of a map $H\colon U \to Y$, 
we use the shorthand notation $\{H = c\} := \{u \in U \mid H(u) = c\}$, $c \in Y$. 
Analogous abbreviations are also used for sublevel sets,
superlevel sets, and general preimages.
Fréchet, Gâteaux, and directional 
derivatives of a function $H\colon U \to Y$ are denoted by a prime; see \cref{def:differentiability_notions}.

Given a Lebesgue measurable set $\Omega \subset \R^d$, $d \in \mathbb{N}$, we denote the vector space 
of (equivalence classes of) Lebesgue measurable real-valued functions on $\Omega$ by $L^0(\Omega)$.
For the real Lebesgue spaces defined w.r.t.\ the $d$-dimensional 
Lebesgue measure $\lambda^d$ on $\Omega$, we use the standard notation 
$(L^q(\Omega), \|\cdot\|_{L^q(\Omega)})$, $1 \leq q \leq \infty$; see \cite{Bocharev2007}.
When working with elements of the spaces $L^q(\Omega)$,
quantities like level sets etc.\ are understood as being defined up to sets of measure zero.
For simplicity, we denote 
representatives of elements of $L^q(\Omega)$
by the same symbols as their associated equivalence classes.
We further define $L^q_+(\Omega) := \{u \in L^q(\Omega) \mid u \geq 0 \text{ a.e.\ in }\Omega \}$.
If $\Omega \subset \R^d$, $d \in \mathbb{N}$, is an open set, then 
the symbols $W^{k,q}(\Omega)$, $H^k(\Omega)$, and $H_0^1(\Omega)$,
$k \in \mathbb{N}$, $q \in [1,\infty]$,
refer to the classical Sobolev spaces on $\Omega$, endowed with their usual norms; see 
\cite[Chapter 5]{Attouch2006}. 
In this situation, we also define $L^1_{loc}(\Omega)$ and $H^1_{loc}(\Omega)$
to be the sets of functions that are locally in $L^1$ and $H^1$ on $\Omega$, respectively,
in the sense of \cite{Gilbarg2001}.
For the spaces of continuous, $k$-times continuously differentiable, 
and smooth functions
on $\Omega$ and $\overline{\Omega}$, 
we use the standard symbols 
$C(\Omega)$, $C^k(\Omega)$, $C^\infty(\Omega)$,
$C(\overline{\Omega})$, 
$\smash{C^k(\overline{\Omega})}$, and 
$C^\infty(\overline \Omega)$, respectively;
see \cite{Attouch2006,Gilbarg2001}. 
Here, $\overline{\Omega}$ is a shorthand notation for the closure of $\Omega$ w.r.t.\ the Euclidean norm $|\cdot|$, i.e., 
$\overline{\Omega} := \Omega \cup \partial \Omega$,
where $\partial(\cdot)$ denotes the boundary of a set.
When talking about smooth functions with compact support on $\Omega$, 
we write $C_c^\infty(\Omega)$.
Recall that, for open bounded $\Omega \subset \R^d$,
 $C(\overline{\Omega})$ is a Banach space when equipped with 
the supremum-norm $\smash{\|\cdot\|_{C(\overline{\Omega})} := \|\cdot\|_\infty}$.
As usual, 
we denote the dual of $(H_0^1(\Omega), \|\cdot\|_{H^1(\Omega)})$ (interpreted as a space of distributions)
by $H^{-1}(\Omega)$; see \cite[Section~5.2]{Attouch2006}. For the distributional 
Laplacian, we use the symbol $\Delta\colon H_0^1(\Omega) \to H^{-1}(\Omega)$,
and for a weak gradient, the symbol $\nabla$. 
Where appropriate, we consider $\Delta$ also as a function $\Delta \colon H^2(\Omega) \to L^2(\Omega)$. 
In the one-dimensional setting, weak 
and distributional derivatives are denoted by a prime. 

Given $a,b \in \R$ with $a < b$, we define $BV[a,b]$ and $G[a,b]$ to be the spaces 
of real-valued functions of bounded variation and real-valued regulated functions on $[a,b]$,
respectively. We emphasize that we understand $BV[a,b]$ and $G[a,b]$
as spaces of classical functions $v\colon [a,b] \to \R$ in this paper, 
analogously to \cite[Definitions~2.1.1,~4.1.1]{Monteiro2019}. In cases in which
the transition to equivalence classes in the sense of $L^\infty(a,b)$ becomes
necessary, we highlight this by working with the symbols
$
L^\infty_G(a,b) :=
\{
v \in L^\infty(a,b)
\mid
v \text{ has a representative in } G[a,b]
\}
$
and 
$
L^\infty_{BV}(a,b) :=
\{
v \in L^\infty(a,b)
\mid
v \text{ has a representative in } BV[a,b]
\}
$.
For details on the relationship between the spaces 
$BV[a,b]$, $G[a,b]$, $L^\infty_G(a,b)$, and $L^\infty_{BV}(a,b)$, see 
\cite[Section 3.2]{Ambrosio2000} and \cref{sec:6}. 
Recall that $BV[a,b]$ is a Banach space when endowed with the 
norm $\|v\|_{BV[a,b]} := |v(a)| + \var(v; [a,b])$,
where $\var(v; [a,b])$ denotes the variation of 
a function $v\colon [a,b] \to \R$ 
in the sense of \cite[Definition 2.1.1]{Monteiro2019}. Analogously, 
$G[a,b]$ is a Banach space when equipped with the supremum-norm $\|\cdot\|_\infty$;
see \cite[Theorem 4.2.1]{Monteiro2019}.
Due to \cite[Remark 4.1.2]{Monteiro2019},
we have  $C([a,b]) \subset G[a,b]$ and 
$BV[a,b] \subset G[a,b]$. From the definition of $G[a,b]$,
it follows further that every  $v \in G[a,b]$
possesses well-defined finite left limits at all points $t \in (a,b]$
and well-defined finite right limits at all points $t \in [a,b)$. We denote these limits
by $v(t-)$ and $v(t+)$, respectively, i.e., we define 
\begin{equation}
\label{eq:leftrightlimits}
v(t-) := \lim_{[a,t) \ni s \to t} v(s)\quad\forall t \in (a,b]
\qquad\text{and}\qquad
v(t+) := \lim_{(t, b] \ni s \to t} v(s)\quad\forall t \in [a,b).
\end{equation}
At the endpoints of the interval $[a,b]$, we use the conventions $v(a-) := 0$ and $v(b+) := 0$
consistent with an extension by zero; see \cref{sec:6}. 
We remark that additional symbols etc.\ are introduced in the following 
sections wherever necessary. This notation is explained on its first appearance in the text.

\subsection{The Problem Under Consideration}
\label{subsec:2.2}
Since using the weak derivative $u := g'$ 
as an optimization variable
offers various advantages (see \cref{subsec:2.3} below for details),
in the remainder of this paper, we focus on the reformulation \eqref{eq:P_g_prime_2}
of the identification problem \eqref{eq:P_g_prime}. 
The problem that we are henceforth concerned with thus reads as follows:
\begin{equation*}
\label{eq:P}
\tag{P}
\left.~~
	\begin{aligned}
		\text{Minimize} 
		\quad &  J(y,u) := \frac12 \| y - y_D\|_{L^2(\Omega)}^2 + \nu_1 \|u\|_{L^1(\R)} + \frac{\nu_2}{2} \|u - u_D\|_{L^2(\R)}^2 \\
        \text{w.r.t.}
        \quad &y \in H^1(\Omega), \quad  u \in L^2(\R),\\
		\text{s.t.} \quad & -\Delta y + g_u(y) = f \text{ in }\Omega,\quad y = 0 \text{ on }\partial \Omega,
		\\
		 \text{and} \quad  & g_u(t) = \int_0^t u(s) \dd s \quad \forall t \in \R,
		 \\
		 \text{and}\quad  & u \geq 0 \text{ a.e.\  in } \R.
	\end{aligned}
\right \}
\end{equation*}
Here, we have introduced the label \eqref{eq:P} for the ease of reference
and the letter $J$ to denote the objective function of \eqref{eq:P}, i.e., 
\begin{equation}
\label{eq:Jdef}
J\colon L^2(\Omega) \times L^2(\R) \to [0,\infty]
,\qquad 
(y,u) \mapsto \frac12 \| y - y_D\|_{L^2(\Omega)}^2 + \nu_1 \|u\|_{L^1(\R)} + \frac{\nu_2}{2} \|u - u_D\|_{L^2(\R)}^2.
\end{equation}
For the sake of clarity, 
we collect our standing assumptions on the quantities in \eqref{eq:P} in:

\begin{assumption}[Standing Assumptions]\label{ass:standing}
~
\begin{itemize}
\item $d \in \{1,2,3\}$;
\item $\Omega \subset \R^d$
is a bounded domain that is convex or of class $C^{1,1}$ (in the sense of \cite[Section~6.2]{Gilbarg2001});
\item $y_D \in L^2(\Omega)$ and $u_D \in L^2(\R)$ are given (the desired state and control);
\item $\nu_1 \geq 0$ and $\nu_2 > 0$ are given Tikhonov parameters;
\item $f \in L^2(\Omega)$ is a given right-hand side that is not identical zero.
\end{itemize}
\end{assumption}

We emphasize that the above conditions are always assumed to hold in our analysis, 
even when not explicitly mentioned. 
Regarding the shorthand notation $g_u$ in \eqref{eq:P}, we point out that 
we will use this symbol not only for controls $u \in L^2(\R)$ but also for arbitrary $u \in L^1_{loc}(\R)$, i.e., 
we set
\begin{equation}
\label{eq:defgu}
g_u\colon \R \to \R,
\qquad 
 g_u(t) := \int_0^t u(s) \dd s\qquad \forall t \in \R\qquad \forall u \in L^1_{loc}(\R).
\end{equation}
Here and in what follows, 
we define
\[
\int_{0}^{0} u(s) \dd s := 0
\qquad \text{and}
\qquad  
\int_{b}^{a} u(s) \dd s := - \int_{a}^b u(s) \dd s \qquad \forall a,b \in \R, a < b,\qquad \forall u \in L^1_{loc}(\R).
\]
Note that these conventions ensure that $g_u' = u$ holds for all $ u \in L^1_{loc}(\R)$
in the sense of weak derivatives. 
With a slight abuse of notation, we also use the symbol $g_u$ to denote 
Nemytskii operators (with varying domains of definition and image spaces) 
that are induced by the scalar functions in \eqref{eq:defgu}. The function spaces that 
these Nemytskii operators act on by superposition will always be clear from the context.
For the sake of brevity, we further  introduce the following notation 
for the differential operator in the governing PDE of \eqref{eq:P}:
\begin{equation}
\label{eq:Audef}
A_u\colon H_0^1(\Omega) \to H^{-1}(\Omega),\qquad v \mapsto  - \Delta v + g_u(v)
\qquad\quad 
\forall u \in L^q(\R), 2\leq q \leq \infty.
\end{equation}
It will be shown in \cref{lemma:operatorAu} that the operator $A_u$ in \eqref{eq:Audef}
is indeed well defined as a function from $H_0^1(\Omega)$ to $H^{-1}(\Omega)$ for all $u \in L^q(\R)$,
$2 \leq q \leq \infty$. 

\subsection{Remarks on the Problem Setting and the Choice of the Regularization Term}
\label{subsec:2.3}

Before we begin with our analysis,
we would like to emphasize that we consider \eqref{eq:P}
as a model problem in this paper. It is easy to check that the results that we derive in the following
can be extended to, e.g.,
more general elliptic partial differential operators,
more complex objective functions (with some caveats regarding $H^2(\Omega)$-regularity results etc.;
see \cite[Theorem~3.2.1.2]{Grisvard1985}, \cite[Theorem 9.15]{Gilbarg2001}, and \cref{sec:3,sec:6}),
and problems which also involve an upper control constraint. We 
do not study these generalizations here since we are mainly interested 
in the nonstandard coupling between $y$ and $u$ in \eqref{eq:P} and
the consequences that this coupling has for the derivation of optimality conditions
in the presence of nonsmooth controls and pointwise constraints.

Regarding the choice of the regularization term $ \nu_1 \|g'\|_{L^1(\R)} + \frac{\nu_2}{2} \|g' - u_D\|_{L^2(\R)}^2$
in \eqref{eq:P_g_prime}, it should be noted that a control-cost term 
that acts solely on the weak derivative of the function $g = g_u$ is a very sensible thing to consider 
when the aim is to identify an ``optimal'' Nemytskii operator in a PDE like 
$-\Delta y + g(y) = f $. Indeed, including, for instance,  an $L^2$-term $\nu_3\|g\|_{L^2(\R)}$
with  $\nu_3 > 0$
into the objective function of \eqref{eq:P_g_prime} would be  completely nonsensical since, 
in combination with the constraint $g'\geq 0$ (which is required to ensure the well-posedness of the state-equation),
such a norm causes the effective admissible set of \eqref{eq:P_g_prime} to consist only of the zero function.
A possible work-around for this problem would be to estimate the range of the
states $y$ appearing in \eqref{eq:P_g_prime} a-priori, i.e., to determine a constant $r > 0$
such that $\range(y) \subset [-r, r]$ holds for all $y$, 
to pose the problem  \eqref{eq:P_g_prime} in the space 
$H^1(-r, r)$, and to then include the regularization term 
$\nu_3\|g\|_{L^2(-r, r)}$ into the objective; cf.\ \cref{lem:PoissonEstimate}.
However, if this approach is used,  
then the solutions of the resulting minimization problem depend  on the initial guess for $r$
(as one may easily check by letting $r$ go to infinity) and the quality of 
the a-priori estimate for the range of the states has a significant impact on the obtained results. 
This effect is highly undesirable, both analytically and for the numerical realization. 

By considering a regularization term of the form 
$\smash{ \nu_1 \|g'\|_{L^1(\R)} + \frac{\nu_2}{2} \|g' - u_D\|_{L^2(\R)}^2}$
in the objective of \eqref{eq:P_g_prime}, we can avoid the above problems. 
As we will see in \cref{sec:4},
this regularizer automatically ensures that 
\eqref{eq:P_g_prime} 
can be reformulated as a problem in $L^2(-r, r)$ for all $r > 0$
that are bigger than a computable lower bound $r_P$ without introducing an 
artificial dependence on the choice of $r$.
Moreover, in the case $u_D \equiv 0$, 
the cost term in \eqref{eq:P_g_prime} naturally gives rise to 
locally optimal controls $\bar u = g_{\bar u}'$ that are ``tight'' in the sense 
that they are supported only on the range of the associated states $\bar y$
and, thus, on precisely those parts of the real line that are  relevant 
for the PDE
$-\Delta \bar y + g_{\bar u}(\bar y) = f $; cf.\ \cref{cor:tightytight}.
As an additional benefit, working with the optimization variable $u = g_u'$
also makes it possible 
to transform the monotonicity constraint in \eqref{eq:P_g_prime} into 
the far easier to handle pointwise $L^2(\R)$-control constraint in  \eqref{eq:P}.
We remark that a possible alternative to the inclusion of  the regularization term
$\smash{ \nu_1 \|g'\|_{L^1(\R)} + \frac{\nu_2}{2} \|g' - u_D\|_{L^2(\R)}^2}$
is the introduction of an additional upper control constraint on $g'$ in \eqref{eq:P_g_prime}; cf.\ 
\cite{Roesch1992,Roesch1996-2,Goebel1998}. After changing to the 
variable $u = g'$, this approach gives rise to a problem 
with classical $L^2$-box constraints that is of bang-bang type.
We leave the analysis of the problem class that is obtained along these lines and 
in particular the question of whether a regularity result 
analogous to \eqref{eq:BVreg1} can be proven for such a variant of \eqref{eq:P}
for future research. 

\section{Basic Properties of the Control-to-State Mapping}
\label{sec:3}
In this section, we prove basic results on the control-to-state mapping $u \mapsto y$
of the problem \eqref{eq:P}. We begin with the following elementary observation:
\begin{lemma}[Hölder Continuity of $g_u$]
For all $u \in L^q(\R)$, $1 \leq q \leq \infty$, it holds 
\begin{equation}
\label{eq:guLqEstimate}
\left | g_u(t_1)  - g_u(t_2)\right | \leq \|u\|_{L^q(\R)} |t_1 - t_2|^{1 - 1/q}\qquad \forall t_1, t_2 \in \R.
\end{equation}
\end{lemma}
\begin{proof}
This follows immediately from Hölder's inequality (with the usual conventions $1/\infty := 0$ and $0^0 := 1$):
\begin{equation*}
\left | g_u(t_1)  - g_u(t_2)\right |  
=
\left | \int_{t_1}^{t_2} u(s) \dd s \right | \leq \|u\|_{L^q(\R)} |t_1 - t_2|^{1 - 1/q}\qquad \forall t_1, t_2 \in \R.\\[-0.4cm]
\end{equation*}
\end{proof}

Before we study which consequences the last result has for the 
differential operator $A_u$ in \eqref{eq:Audef}, we introduce some additional notation that will be used throughout this work.
\begin{definition}[Important Constants]
\label{def:Pdefs}
Henceforth, we denote:
\begin{enumerate}[label=\roman*)]
\item by $c_P > 0$ the constant in the inequality of Poincaré-Friedrichs 
\begin{equation}
\label{eq:PoincFried}
c_P \int_\Omega v^2 \dd x \leq \int_\Omega |\nabla v|^2 \dd x \qquad \forall v \in H_0^1(\Omega);
\end{equation}
\item\label{def:Pdefs:ii}  by  $\varepsilon_P>0$ 
a constant $\varepsilon_P  \in (0, c_P)$ whose value is fixed at this point once and for all;
\item by $U_P \subset L^2(\R)$ the set $U_P :=
\left  \{u \in L^2(\R) \mid {u \geq - \varepsilon_P} 
\text{ a.e.\ in }\R \right \}$;
\item by $y_P \in H_0^1(\Omega) \cap H^2(\Omega) \setminus \{0\}$ the solution of the Poisson problem 
\begin{equation}
\label{eq:fPoisson}
 - \Delta y_P = f \text{ in }\Omega,\qquad y_P = 0 \text{ on }\partial \Omega;
\end{equation}
\item\label{def:Pdefs:v} by $r_P > 0$ the constant $r_P := 2\|y_P\|_{\infty}$.
\end{enumerate}
\end{definition}

Note that $y_P$ is indeed an element of $H_0^1(\Omega) \cap H^2(\Omega) \setminus \{0\}$ and $C(\overline{\Omega})$
by our standing assumptions on $f$, $\Omega$, and $d$, 
the regularity results in 
 \cite[Theorem 9.15, Lemma 9.17]{Gilbarg2001} 
and \cite[Theorems 3.1.3.1, 3.2.1.2]{Grisvard1985},
and the Sobolev embeddings. More precisely, we obtain from 
\cite{Grisvard1985,Gilbarg2001} that every function $v \in H_0^1(\Omega)$
satisfying $\Delta v \in L^2(\Omega)$ is an element of $H_0^1(\Omega) \cap H^2(\Omega)$
and that there exists a constant $C>0$ such that the 
following estimate is true:
\begin{equation}
\label{eq:H2_reg_res}
\|v\|_{H^2(\Omega)} \leq C \|\Delta v \|_{L^2(\Omega)}\qquad \forall v \in H_0^1(\Omega) \cap H^2(\Omega).
\end{equation}

We can now prove:
\begin{lemma}[Properties of $A_u$]
\label{lemma:operatorAu}
For every $u \in L^q(\R)$, $2 \leq q \leq \infty$,  the operator $A_u(v) := - \Delta v + g_u(v)$
in \eqref{eq:Audef}
is well defined as a function 
$A_u\colon H_0^1(\Omega) \to H^{-1}(\Omega)$. 
Moreover, the following is true: 
\begin{enumerate}[label=\roman*)]
\item\label{lemma:operatorAu:i} 
There exists a constant $C>0$ such that, for all 
$u \in L^q(\R)$, $2 \leq q \leq \infty$,  we have
\begin{equation*}
\left \|  A_u(v_1) - A_u(v_2) \right \|_{H^{-1}(\Omega)}
\leq 
C\|v_1 - v_2 \|_{H^1(\Omega)}
+
\|u\|_{L^q(\R)}
\|v_1 - v_2\|_{L^{2 - 2/q}(\Omega)}^{1 - 1/q}\quad \forall v_1, v_2 \in H_0^1(\Omega).
\end{equation*}
\item\label{lemma:operatorAu:ii}  
There exists a constant $c > 0$ such that, 
for all $u \in U_P$, it holds
\[
\left \langle A_u(v_1) - A_u(v_2), v_1 - v_2 \right \rangle_{H_0^1(\Omega)} 
\geq 
c \|v_1 - v_2 \|_{H^1(\Omega)}^2\quad \forall v_1, v_2 \in H_0^1(\Omega).
\]
\end{enumerate}
\end{lemma}

\begin{proof}
For all  $u \in L^q(\R)$, $2 \leq q \leq \infty$, 
and $v \in H_0^1(\Omega)$, we have  $-\Delta v \in H^{-1}(\Omega)$ and,
due to \eqref{eq:guLqEstimate} and $g_u(0) = 0$, 
\begin{equation}
\label{eq:randomeq273647}
\begin{aligned}
\|  g_u(v)\|_{L^2(\Omega)}^2
=
\int_\Omega  g_u(v)^2 \dd x
\leq
\int_\Omega  \left (\|u\|_{L^q(\R)} |v|^{1 - 1/q} \right )^2 \dd x
=
\|u\|_{L^q(\R)}^2
\|v \|_{L^{2 - 2/q}(\Omega)}^{2 - 2/q}.
\end{aligned}
\end{equation}
Because of the embeddings
$H_0^1(\Omega)  \hookrightarrow L^2(\Omega)\hookrightarrow L^{2 - 2/q}(\Omega)$,
$q \geq 2$,
\eqref{eq:randomeq273647} 
yields that $g_u(v) \in L^2(\Omega) \subset H^{-1}(\Omega)$ holds 
and that $A_u$ is well defined as a function from 
$H_0^1(\Omega)$ to $H^{-1}(\Omega)$ for all 
$u \in L^q(\R)$, $2 \leq q \leq \infty$. This proves 
the first assertion of the lemma. 
Next, we note that \eqref{eq:guLqEstimate} implies 
\begin{equation*}
\begin{aligned}
\left \|  A_u(v_1) - A_u(v_2) \right \|_{H^{-1}(\Omega)}
& \leq 
\left \|  -\Delta (v_1 - v_2 )\right \|_{H^{-1}(\Omega)} 
+
\left \| g_u(v_1)  - g_u(v_2)\right \|_{L^2(\Omega)}
\\
& \leq 
C\|v_1 - v_2 \|_{H^1(\Omega)}
+
\left (\int_\Omega  \left (\|u\|_{L^q(\R)} |v_1 - v_2|^{1 - 1/q} \right )^2 \dd x\right )^{1/2}
\\
& =
C\|v_1 - v_2 \|_{H^1(\Omega)}
+
\|u\|_{L^q(\R)}
\|v_1 - v_2\|_{L^{2 - 2/q}(\Omega)}^{1 - 1/q}
\end{aligned}
\end{equation*}
for all $u \in L^q(\R)$, $2 \leq q \leq \infty$, and $v_1, v_2 \in H_0^1(\Omega)$, where 
 $C > 0$ denotes the operator norm of the Laplacian $\Delta\colon H_0^1(\Omega) \to H^{-1}(\Omega)$. 
 This establishes  \ref{lemma:operatorAu:i}.
 It remains to prove \ref{lemma:operatorAu:ii}.
 To this end, 
let us assume that $u \in U_P$ is given. 
From the definition of $g_u$, we obtain that the 
function $g_{\max(0,u)}\colon \R \to \R$ is nondecreasing, 
that the function $g_{\min(0,u)}\colon \R \to \R$ is nonincreasing,
and that $g_u = g_{\max(0,u)} + g_{\min(0,u)}$ holds.
By exploiting these properties, \eqref{eq:guLqEstimate}, and the definitions 
of $U_P$, $c_P$, and $\varepsilon_P$, we obtain 
\begin{equation}
\label{eq:randomeq263gg67wgwzb}
\begin{aligned}
\left \langle A_u(v_1) - A_u(v_2), v_1 - v_2 \right \rangle_{H_0^1(\Omega)}
&=
\int_\Omega |\nabla (v_1 - v_2)|^2\dd x
+
\int_\Omega (g_u(v_1) - g_u(v_2) ) (v_1 - v_2) \dd x 
\\
&=
 \|\nabla (v_1 - v_2)  \|_{L^2(\Omega)^d}^2 
+ \int_\Omega 
(g_{\min(0,u)}(v_1) - g_{\min(0, u)}(v_2) ) (v_1 - v_2) 
 \dd x  
 \\
 &\hspace{3.375cm}
+ \int_\Omega 
(g_{\max(0,u)}(v_1) - g_{\max(0, u)}(v_2) ) (v_1 - v_2) 
 \dd x  
\\
&\geq 
 \|\nabla (v_1 - v_2)  \|_{L^2(\Omega)^d}^2 
- \int_\Omega 
\left | 
g_{\min(0,u)}(v_1) - g_{\min(0, u)}(v_2) 
\right |
|v_1 - v_2|
 \dd x   
 \\
 &\geq 
 \frac{c_P - \varepsilon_P }{2 c_P}  \|\nabla (v_1 - v_2)  \|_{L^2(\Omega)^d}^2 
+ \frac{\varepsilon_P + c_P}{2 } \|v_1 - v_2\|_{L^2(\Omega)}^2
- \varepsilon_P \int_\Omega 
(v_1 - v_2)^2
 \dd x   
 \\
&\geq
\frac{c_P - \varepsilon_P}{2 \max(1, c_P)} \|v_1 - v_2\|_{H^1(\Omega)}^2\qquad \forall v_1, v_2 \in H_0^1(\Omega).
\end{aligned}
\end{equation}
This establishes \ref{lemma:operatorAu:ii} and  completes the proof. 
\end{proof}

Using standard results and \cref{lemma:operatorAu}, we obtain:

\begin{theorem}[Properties of the Control-to-State Mapping]%
\label{th:propS}%
The partial differential equation
\begin{equation}
\label{eq:PDE}
 - \Delta y + g_u(y) = f \text{ in }\Omega,\qquad y = 0 \text{ on }\partial \Omega,
\end{equation}
possesses a unique solution $y \in H_0^1(\Omega) \cap H^2(\Omega)$ for all $u \in U_P$. 
The  solution operator $S\colon U_P \to H_0^1(\Omega) \cap H^2(\Omega)$, $u \mapsto y$,
associated with the PDE \eqref{eq:PDE} possesses the following properties:
\begin{enumerate}[label=\roman*)]
\item\label{th:propS:item:i} 
The map $S$ is completely continuous as a function from $U_P \subset L^2(\R)$ to $H^2(\Omega)$, i.e., 
\begin{equation*}
\{u_n\} \subset U_P, u_n \weakly u \text{ in }L^2(\R)
\qquad \Rightarrow\qquad S(u_n) \to S(u) \text{ in } H^2(\Omega). 
\end{equation*}

\item\label{th:propS:item:ii} 
There exists a constant $M > 0$ such that 
\begin{equation}
\label{eq:H2estimate}
S(U_P) \subset 
\left \{v \in H_0^1(\Omega) \cap H^2(\Omega)
~\big | ~\|v\|_{H^2(\Omega)}  \leq M
\right \}.
\end{equation}

\item\label{th:propS:item:iii} 
 There exists a constant $c>0$ such that 
\begin{equation*}
c \|S(u_1) - S(u_2)\|_{H^1(\Omega)}^2 
\leq 
 \int_\Omega  g_{u_2 - u_1}(S(u_1) )   ( S(u_1)  - S(u_2)  ) \dd x\quad \forall u_1, u_2 \in U_P.
\end{equation*}
\item\label{th:propS:item:iv} 
There exists a constant $C > 0$ such that, for all $2 \leq q \leq \infty$, it holds 
\begin{equation*}
\|S(u_1) - S(u_2)\|_{H^1(\Omega)}
\leq 
C \|u_1 - u_2\|_{L^q(\R)}\qquad \forall  u_1, u_2 \in U_P \cap L^q(\R).
\end{equation*}
\end{enumerate}
\end{theorem}

\begin{proof}
Let $u \in U_P$ be given. 
From \cref{lemma:operatorAu}, the continuous embedding $H_0^1(\Omega) \hookrightarrow L^2(\Omega)$,
and the identity $A_u(0) = 0$, we obtain that
$A_u\colon H_0^1(\Omega) \to H^{-1}(\Omega)$ 
is a continuous,
coercive,
and strongly monotone operator. 
In combination with \cite[Theorem~4.1]{Troeltzsch2010}, this implies that there
exists a unique $y \in H_0^1(\Omega)$ satisfying $A_u(y) = f$ in $H^{-1}(\Omega)$.
The PDE \eqref{eq:PDE} thus possesses a unique weak solution 
$y = S(u) \in H_0^1(\Omega)$.
To see that we even have $y \in H^2(\Omega)$,
we note that  $g_u(y)$ is an element of $L^2(\Omega)$ by \eqref{eq:randomeq273647}
and that 
$-\Delta y = f - g_u(y)$ holds in $H^{-1}(\Omega)$ by the definition of $A_u$.
Since $f \in L^2(\Omega)$,
this allows us to invoke \eqref{eq:H2_reg_res}
to obtain  $y \in H^2(\Omega)$ as desired.

It remains to prove the properties of the solution map $S\colon u \mapsto y$ in 
\ref{th:propS:item:i}  to \ref{th:propS:item:iv}.
To this end, we first note that
 \eqref{eq:PDE}, the identity $A_{u}(0) = 0$, and \cref{lemma:operatorAu}\ref{lemma:operatorAu:ii}
imply that there exists a constant $c > 0$ such that 
\begin{equation}
\label{eq:randomestimate2735636}
\begin{aligned}
c \|y\|_{H^1(\Omega)}^2
&\leq
\left \langle A_{u}(y), y \right \rangle_{H_0^1(\Omega)}
=
\int_\Omega f y \dd x 
\leq \|f\|_{L^2(\Omega)} \|y\|_{H^1(\Omega)}
\end{aligned}
\end{equation}
holds for all $u \in U_P$ and  $y = S(u)$.
This shows that the set
$S(U_P)$ is bounded in $H^1(\Omega)$. 
Suppose now that a sequence $\{u_n\} \subset U_P$
satisfying $u_n \weakly u$ in $L^2(\R)$ for some $u \in L^2(\R)$ is given
and define $y_n := S(u_n)$. 
Then it follows from the lemma of Mazur
that $u$ is an element of $U_P$. 
From the boundedness of the sequence $\{y_n\}$ in $H^1(\Omega)$,
the theorem of Banach-Alaoglu, and the Rellich-Kondrachov embedding theorem,
we further obtain that we can pass over to a subsequence (denoted the same) such that 
$\{y_n\}$ converges weakly in $H_0^1(\Omega)$ and strongly in $L^2(\Omega)$ to a function $y \in H_0^1(\Omega)$. 
In combination with \eqref{eq:guLqEstimate}, 
the embedding $L^2(\Omega) \hookrightarrow L^1(\Omega)$, 
the fact that $u_n \weakly u$ in $L^2(\R)$
implies $g_{u_n - u}(t) \to 0$ for all $t \in \R$,
the boundedness of $\{u_n\}$ in $L^2(\R)$, 
the majorization estimate 
\[
|  g_{u_n - u}(y) |^2 
\leq
\|u_n - u\|_{L^2(\R)}^2 |y| \qquad \text{ a.e.\ in }\Omega
\]
obtained from \eqref{eq:guLqEstimate},
and the dominated convergence theorem,
this yields
\begin{equation}
\label{eq:randomeq263636-dbgdgeut782ddgzu}
\begin{aligned}
\left \| g_u(y)  - g_{u_n}(y_n)\right \|_{L^2(\Omega)}
&=
\left \| g_u(y) - g_{u_n}(y)  + g_{u_n}(y)  - g_{u_n}(y_n)\right \|_{L^2(\Omega)}
\\
&\leq
\left \| g_{u - u_n}(y)\right \|_{L^2(\Omega)}
+
\left \| g_{u_n}(y)  - g_{u_n}(y_n)\right \|_{L^2(\Omega)}
\\
& \leq
\left \| g_{u - u_n}(y)\right \|_{L^2(\Omega)}
+
\|u_n\|_{L^2(\R)}
\|y - y_n\|_{L^1(\Omega)}^{1/2} \to 0
\end{aligned}
\end{equation}
for $n \to \infty$. 
Since the Laplacian $-\Delta\colon H_0^1(\Omega) \to  H^{-1}(\Omega)$
is linear and continuous  
and since $A_{u_n}(y_n) = f$ holds by the definition of $y_n$,
the above implies 
\[
0 = A_{u_n}(y_n) - f = - \Delta y_n + g_{u_n}(y_n) - f 
\weakly 
- \Delta y+ g_{u}(y ) - f = A_{u }(y ) - f  \text{ in } H^{-1}(\Omega). 
\]
The limit $y$ of the sequence $\{y_n\}$ thus satisfies 
$A_{u }(y ) - f = 0$ and we may conclude that $y = S(u)$ holds and that 
$S(u_n)$ converges weakly in $H_0^1(\Omega)$ to $S(u)$ for $n \to \infty$.
As the above arguments can be repeated along arbitrary 
subsequences of the initial $\{u_n\} \subset L^2(\R)$, 
it follows from a trivial contradiction argument that the convergence $S(u_n) \weakly S(u)$ in $H_0^1(\Omega)$
also holds for the whole original sequence of controls  $\{u_n\}$;
cf.\ \cite[Lemma~4.16]{Schweizer2013}. 
This proves that $S$
is weakly continuous as a function from $U_P \subset L^2(\R)$ to $H_0^1(\Omega)$. To see that we 
even have strong convergence of the states in $H^2(\Omega)$, we note that 
the estimates \eqref{eq:H2_reg_res} and \eqref{eq:randomeq263636-dbgdgeut782ddgzu}
imply that there exists a constant $C>0$ such that
\[
\|y_n - y\|_{H^2(\Omega)}
\leq
C 
\|-\Delta y_n + \Delta y\|_{L^2(\Omega)}
=
C 
\|g_u(y)  - g_{u_n}(y_n)\|_{L^2(\Omega)} \to 0
\]
holds for $n \to \infty$. We thus indeed have $S(u_n) \to S(u)$ in $H^2(\Omega)$
for all  $\{u_n\} \subset U_P$ that satisfy $u_n \weakly u$ in $L^2(\R)$ for some $u \in L^2(\R)$.
This completes the proof of \ref{th:propS:item:i}. 

To prove \ref{th:propS:item:ii}, let us assume that $u \in U_P \cap L^\infty(\R)$ is given. 
From the identity $g_u = g_{\max(0,u)} + g_{\min(0,u)}$ and \eqref{eq:PDE}, we obtain  that 
$y := S(u)$ 
is the (necessarily unique) weak solution of the semilinear PDE
\begin{equation}
\label{eq:randomeq26363}
 - \Delta y + g_{\max(0, u)} (y) = \tilde f \text{ in }\Omega,\qquad y = 0 \text{ on }\partial \Omega,
\end{equation}
whose right-hand side $\tilde f$ is defined by $\tilde f := f - g_{\min(0, u)} (y) \in L^2(\Omega)$. 
Note that the definition of $U_P$, \eqref{eq:randomeq273647},
and the estimate \eqref{eq:randomestimate2735636}
imply that 
$\|\min(0, u)\|_{L^\infty(\R)} \leq \varepsilon_P$ holds and that $\tilde f$ satisfies 
\begin{equation}
\label{eq:randomeq36354}
\|\tilde f\|_{L^2(\Omega)}
\leq
\| f\|_{L^2(\Omega)}
+
\|g_{\min(0, u)} (y)\|_{L^2(\Omega)}
\leq
\| f\|_{L^2(\Omega)} + 
\|\min(0, u)\|_{L^\infty(\R)}\|y\|_{L^2(\Omega)}
\leq
\left ( 1 + \frac{\varepsilon_P}{c}\right )\| f\|_{L^2(\Omega)}.
\end{equation}
From the $L^\infty$-regularity of $u$,
the definition of the function $g_u$,
the regularity $y \in H_0^1(\Omega) \cap H^2(\Omega) \subset C(\overline{\Omega})$,
and \cite[Theorem 2.1.11]{Ziemer1989},
it follows further that the map 
$g_{\max(0,u)}\colon \R \to \R$ is globally Lipschitz continuous, 
that $g_{\max(0,u)}(0) = 0$ holds,
and that $v := g_{\max(0, u)} (y)$ is an element of $H_0^1(\Omega)$
whose weak gradient is given by 
$\nabla v = \max(0, u(y)) \nabla y$
a.e.\ in $\Omega$. 
In combination with \eqref{eq:randomeq26363}, the last identity yields
\begin{equation*}
\|v\|_{L^2(\Omega)}^2\leq 
\int_\Omega 
 \max(0, u(y))  
 |\nabla y|^2 \dd x 
+
\|v\|_{L^2(\Omega)}^2
 =
\int_\Omega 
\nabla y \cdot \nabla v 
+
g_{\max(0, u)} (y) v \dd x
 =\int_\Omega \tilde f v \dd x
 \leq
 \|v\|_{L^2(\Omega)} \|\tilde f\|_{L^2(\Omega)}.
\end{equation*}
We thus have $\|v\|_{L^2(\Omega)} \leq \|\tilde f\|_{L^2(\Omega)}$.
By revisiting \eqref{eq:randomeq26363}, by exploiting \eqref{eq:randomeq36354},
and by again invoking the estimate \eqref{eq:H2_reg_res},
we now obtain that there exists a constant $C > 0$ independent of $u$ satisfying
\[
\|y\|_{H^2(\Omega)} 
\leq C \|\tilde f - v\|_{L^2(\Omega)}
\leq 
2C \|\tilde f\|_{L^2(\Omega)}
\leq 2C 
\left ( 1 + \frac{\varepsilon_P}{c}\right )\| f\|_{L^2(\Omega)}.
\]
This proves the bound in \ref{th:propS:item:ii} with 
$M :=2 C 
\left ( 1 + \frac{\varepsilon_P}{c}\right )\| f\|_{L^2(\Omega)}$
for all $u \in U_P \cap L^\infty(\R)$. 
Since the set $U_P \cap L^\infty(\R)$ is dense in $U_P$
w.r.t.\ convergence in $L^2(\R)$ and since $S$
is continuous as a function from $U_P \subset L^2(\R)$ to $H^2(\Omega)$ by \ref{th:propS:item:i}, 
the assertion of \ref{th:propS:item:ii} now follows immediately. 

To establish \ref{th:propS:item:iii}, we note that, for all $u_1, u_2 \in U_P$
with associated states $y_1 := S(u_1)$ and $y_2 := S(u_2)$,
we obtain from \eqref{eq:PDE} and \cref{lemma:operatorAu}\ref{lemma:operatorAu:ii}  
that 
\begin{equation}
\label{eq:randomeq27336}
\begin{aligned}
c \|y_1 - y_2 \|_{H^1(\Omega)}^2
&\leq 
\left \langle A_{u_2}(y_1) - A_{u_2}(y_2), y_1 - y_2 \right \rangle_{H_0^1(\Omega)} 
\\
& =
\left \langle A_{u_1}(y_1) - A_{u_2}(y_2), y_1 - y_2 \right \rangle_{H_0^1(\Omega)} 
+
\left \langle A_{u_2}(y_1) - A_{u_1}(y_1), y_1 - y_2 \right \rangle_{H_0^1(\Omega)} 
\\
& = 0 +
\int_\Omega \left ( g_{u_2}(y_1) - g_{u_1}(y_1)\right ) ( y_1 - y_2 ) \dd x 
=
 \int_\Omega  g_{u_2 - u_1}(y_1)   ( y_1 - y_2 ) \dd x.
\end{aligned}
\end{equation}
This proves \ref{th:propS:item:iii}.

To finally obtain \ref{th:propS:item:iv}, we note that, for all $u_1, u_2 \in U_P \cap L^q(\R)$,
$2 \leq q \leq \infty$,
we can use \eqref{eq:randomeq273647} to continue the 
estimate \eqref{eq:randomeq27336} as follows:
\begin{equation*}
c \|y_1 - y_2 \|_{H^1(\Omega)}^2 
\leq \|y_1 - y_2\|_{L^2(\Omega)}
 \| g_{u_2 - u_1}(y_1)  \|_{L^2(\Omega)}
\leq
 \|y_1 - y_2\|_{H^1(\Omega)}
 \| u_1 - u_2   \|_{L^q(\R)}
 \|y_1\|_{L^{2 - 2/q}(\Omega)}^{1 - 1/q}.
\end{equation*}
In combination with \ref{th:propS:item:ii}, this  establishes \ref{th:propS:item:iv} and completes the proof of the theorem. 
\end{proof}

Note that the boundedness of the set 
$S(U_P) \subset H^2(\Omega)$ in \cref{th:propS}\ref{th:propS:item:ii}, our assumption $d \in \{1,2,3\}$,
and the Sobolev embeddings imply that there exists a constant $r > 0$ such that 
$\|S(u)\|_{\infty} \leq r$ holds for all $u \in U_P$.
This property will later on allow us to restate \eqref{eq:P}
as a problem in $L^2(-r,r)$; see \cref{th:truncate}.  Unfortunately, the precise value of the 
bound $r$ on the maximum-norm of the states
$S(u)$, $u \in U_P$, obtained from \eqref{eq:H2estimate} is hard to calculate precisely as it depends on $c_P$
and various embedding constants. This is a problem since, 
for an efficient numerical solution of \eqref{eq:P},
a bound on the range of the appearing states is required 
that is as sharp as possible. The following lemma shows that, 
for the nonnegative controls in \eqref{eq:P},
a much more tangible $\|\cdot\|_\infty$-estimate can be obtained 
in terms of the solution $y_P = S(0)$ of the Poisson equation in \eqref{eq:fPoisson}.

\begin{lemma}[$\|\cdot\|_\infty$-Bound]
\label{lem:PoissonEstimate}
The solution operator 
$S\colon U_P \to H_0^1(\Omega) \cap H^2(\Omega)$, $u \mapsto y$,
of \eqref{eq:PDE} satisfies
\begin{equation}
\label{eq:PoissonEstimate}
\left \| S(u)\right \|_{\infty}\leq r_P
\qquad \forall\, 0 \leq u\in L^2(\R).
\end{equation}
Here, $r_P = 2\left \| y_P\right \|_{\infty}$ denotes the
constant from \cref{def:Pdefs}\ref{def:Pdefs:v}
and $y_P$ the solution of \eqref{eq:fPoisson}.
\end{lemma}

\begin{proof}
From \cref{th:propS} and the Sobolev embeddings, 
it follows that $y_P$ satisfies $y_P = S(0) \in C(\overline{\Omega})$. Consider now 
a control
$0 \leq u\in L^2(\R)$ with state $y := S(u) \in H_0^1(\Omega) \cap H^2(\Omega) \subset C(\overline{\Omega})$.
We
define $z := y - y_P$ and 
$w := z  - \min(\left \| y_P\right \|_{\infty}, \max(-\left \| y_P \right \|_{\infty}, z))$.
Due to their construction and the lemma of Stampacchia (see 
\cite[Theorem II.A.1]{KinderlehrerStampacchia1980} and 
\cite[Theorem 5.8.2]{Attouch2006}), 
the functions $z$ and $w$ satisfy $z, w \in H_0^1(\Omega)$, 
\begin{equation}
\label{eq:poisPDE}
\int_\Omega \nabla z \cdot \nabla v  + g_u(y) v \dd x = 0\qquad \forall v \in H_0^1(\Omega),
\end{equation}
and 
\[
w = 
\begin{cases}
z -  \left \| y_P\right \|_{\infty} & \text{if } z > \left \| y_P\right \|_{\infty},
\\
0 & \text{if } - \left \| y_P\right \|_{\infty}  \leq z \leq \left \| y_P\right \|_{\infty},
\\
z + \left \| y_P\right \|_{\infty} & \text{if } z < - \left \| y_P\right \|_{\infty}. 
\end{cases}
\]
By choosing $w$ as the test function in \eqref{eq:poisPDE} 
and by exploiting the formulas in \cite[Theorem 5.8.2]{Attouch2006}, we obtain
\begin{equation}
\label{eq:randomeq27367dge6}
0 = 
\int_\Omega \nabla z \cdot \nabla w \dd x + \int_\Omega g_u(y) w \dd x 
=
\int_\Omega |\nabla w |^2 \dd x + \int_\Omega g_u(y) w \dd x 
\geq 
\int_\Omega  |\nabla w |^2 \dd x.
\end{equation}
Here, the last inequality follows from the implications 
\begin{equation*}
\begin{aligned}
&w> 0 
\quad \Rightarrow\quad 
y - y_P = z > \left \| y_P\right \|_{\infty}
&&\Rightarrow\quad 
y \geq y_P + \left \| y_P\right \|_{\infty}
\quad \Rightarrow\quad 
y \geq 0,
\\
&w<  0 
\quad \Rightarrow\quad 
y - y_P = z < - \left \| y_P\right \|_{\infty}
&&\Rightarrow\quad 
y \leq y_P - \left \| y_P\right \|_{\infty}
\quad \Rightarrow\quad 
y \leq 0,
\end{aligned}
\end{equation*}
the fact that $g_u$ is nondecreasing, 
and the identity $g_u(0) = 0$. 
Due to \eqref{eq:PoincFried}, 
\eqref{eq:randomeq27367dge6} yields $w=0$ and, 
by the definition of $w$, 
$y_P - \left \| y_P\right \|_{\infty}  \leq y  \leq y_P + \left \| y_P\right \|_{\infty}$.
This establishes \eqref{eq:PoissonEstimate} and completes the proof. 
\end{proof}

\begin{remark}
Inequality \eqref{eq:PoissonEstimate} is not true anymore in general when the factor two 
in the definition of $r_P$
is dropped. Indeed, numerical evidence suggests that the estimate 
$\left \| S(u)\right \|_{\infty}\leq  2 \left \| y_P\right \|_{\infty}$
for all $u\in L_+^2(\R)$~is~optimal.
\end{remark}

\section{First Consequences for the Optimization Problem}
\label{sec:4}

Having studied the properties of the state-equation \eqref{eq:PDE},
we now turn our attention to the analysis of the optimization problem \eqref{eq:P}.
Henceforth, with the letters $S$ and $J$, we always refer to the control-to-state 
mapping $S\colon U_P \to H_0^1(\Omega) \cap H^2(\Omega)$
introduced in \cref{th:propS} and the function $J$ defined in 
\eqref{eq:Jdef}. For the definitions of the quantities $\varepsilon_P$, $r_P$,
and $U_P$ appearing in the following, see \cref{def:Pdefs}. We first state:

\begin{definition}[Local and Global Solutions]
\label{def:localMin}
A function $0 \leq \bar u \in L^2(\R)$ is called a local minimizer/local solution/locally optimal control of 
$\eqref{eq:P}$ if there exists a number $\rho >0$ (the radius of optimality) such that 
\begin{equation}
\label{eq:localMin}
J(S(u), u) \geq J(S(\bar u), \bar u) \qquad \forall \,0 \leq u \in B_{\rho}^{L^2(\R)}(\bar u).
\end{equation}
If $\rho$ can be chosen as $\infty$, then we call $\bar u$ a global minimizer/global solution/globally 
optimal control of \eqref{eq:P}.
\end{definition}

It is easy to check that a control $\bar u$ can only be a local minimizer of 
\eqref{eq:P} if $J(S(\bar u), \bar u) < \infty$ holds. 
In the case $\nu_1 > 0$,
all local 
solutions of \eqref{eq:P} thus satisfy $\bar u \in L^1(\R)$.  
For the various variants of the problem \eqref{eq:P} that appear in the following sections, 
the notions of local and global optimality are defined analogously to \eqref{eq:localMin}. 
By exploiting the complete continuity of the map $S$ in \cref{th:propS}\ref{th:propS:item:i},
we obtain:

\begin{theorem}[Solvability of \eqref{eq:P}]
\label{th:solexistence}
The problem \eqref{eq:P} possesses at least one globally 
optimal control $ \bar u \in L^2(\R)$.
\end{theorem}

\begin{proof}
This follows immediately from the direct method of the calculus of variations. 
Indeed, the admissible set of the problem \eqref{eq:P} is clearly nonempty (as it contains the zero function) 
and we can find 
nonnegative controls
 $\{u_n\} \subset L^2(\R)$ with associated states $y_n := S(u_n)$
such that the function values $J(y_n, u_n)$ are finite and such that the sequence 
$\{J(y_n, u_n)\}$ converges to the infimal value of \eqref{eq:P}. 
Due to the structure of $J$, the sequence $\{u_n\}$ is bounded in $L^2(\R)$ and we may assume 
w.l.o.g.\ that $u_n \weakly \bar u$ holds in $L^2(\R)$ for some (necessarily nonnegative) $\bar u \in L^2(\R)$.
Since the function $L^2(\R) \ni u \mapsto \nu_1 \|u\|_{L^1(\R)} + \frac{\nu_2}{2} \|u - u_D\|_{L^2(\R)}^2 \in [0, \infty]$
is convex and lower semicontinuous by the lemma of Fatou and, 
as a consequence,  weakly lower semicontinuous by the lemma of Mazur, 
and since the 
map $U_P\ni u \mapsto \frac12\|S(u) - y_D \|_{L^2(\Omega)}^2 \in \R$ is completely continuous 
by \cref{th:propS}\ref{th:propS:item:i}, we obtain that 
$\liminf_{n \to \infty} J(S(u_n), u_n) \geq J(S(\bar u), \bar u)$ holds. Due to the choice of $\{u_n\}$,
this proves the claim.
\end{proof}

As already mentioned in \cref{sec:3}, the $\|\cdot\|_\infty$-bound
in  \eqref{eq:PoissonEstimate}
makes it possible to recast \eqref{eq:P} as a minimization problem in $L^2(-r,r)$
for all large enough $r>0$. 
To formulate  this result rigorously, we introduce:

\begin{definition}[Restriction and Extension by Zero] 
\label{def:ExtZero}
Given a number $r > 0$, we define:
\begin{enumerate}[label=\roman*)]
\item $\uu_D \in L^2(-r,r)$ to be the function $\uu_D := u_D|_{(-r,r)}$;
\item\label{def:ExtZero:ii} $
E_r\colon  L^2(-r,r) \to L^2(\R)$ to be the map
that extends elements of $L^2(-r,r)$ by zero to elements of $L^2(\R)$;
\item  $U_{P,r}$ to be the set 
$ U_{P,r} := 
\left  \{\uu \in L^2(-r,r) \mid \uu  \geq - \varepsilon_P
\text{ a.e.\ in } (-r,r) \right \}
= 
\left  \{\uu \in L^2(-r,r) \mid E_r(\uu) \in U_P \right \}
$;
\item 
$S_r$  to be the operator
\mbox{$S_r\colon  U_{P,r}\to H_0^1(\Omega) \cap H^2(\Omega)$,}
$S_r(\uu) := S(E_r(\uu))$;
\item\label{def:ExtZero:v} $F_r$  to be the function 
\[
F_r\colon U_{P,r} \to \R,\qquad
F_r(\uu) := \frac12 \|S_r(\uu)- y_D\|_{L^2(\Omega)}^2 
		+ \int_{-r}^r \nu_1  \uu + \frac{\nu_2}{2}( \uu - \uu_D)^2\dd t.
		\]
\end{enumerate}
\end{definition}
Here and in what follows,  the typesetting $\uu$ is used to emphasize 
that we talk about functions on a subinterval $(-r,r)$ and not about 
controls $u$ on the whole real line as appearing in the original problem formulation \eqref{eq:P}. 
Note that, due to the structure of the PDE \eqref{eq:PDE}, we have:

\begin{lemma}[Truncation Invariance] 
\label{lemma:truncinv}
The operators $S$ and $S_r$ satisfy
\begin{equation}
\label{eq:truncid}
\begin{aligned}
S(u) = S\left (\mathds{1}_{(-r,r)} u\right ) = S_r\left (u|_{(-r,r)}\right )
\qquad 
\forall\,0 \leq u \in L^2(\R)\qquad \forall r \geq r_P.
\end{aligned}
\end{equation}
\end{lemma}

\begin{proof}
Let $0 \leq u \in L^2(\R)$ and $r \geq r_P$ be given. 
By the definition of $S$, 
the function $y = S(u) \in H_0^1(\Omega)\cap H^2(\Omega)$ 
is the weak solution of the PDE
\begin{equation*}
 - \Delta y + g_u(y) = f \text{ in }\Omega,\qquad y = 0 \text{ on }\partial \Omega,
\end{equation*}
and from \eqref{eq:PoissonEstimate}, we obtain 
that the continuous representative of $y$ satisfies $\range(y) \subset [-r_P, r_P] \subset [-r,r]$.
In combination with \eqref{eq:defgu}, this implies that $g_{u}(t) = g_{\mathds{1}_{(-r,r)} u}(t)$
holds for all $t \in [-r,r]$ and that $y$ is also the weak solution of 
\begin{equation*}
 - \Delta y +  g_{\mathds{1}_{(-r,r)} u}(y) = f \text{ in }\Omega,\qquad y = 0 \text{ on }\partial \Omega.
\end{equation*}
Thus, $S(u) = S\left (\mathds{1}_{(-r,r)} u\right )$ as claimed. The second identity 
$S\left (\mathds{1}_{(-r,r)} u\right ) = S_r\left (u|_{(-r,r)}\right )$ in \eqref{eq:truncid} 
follows  
from the definition of $S_r$. This establishes the assertion.
\end{proof}

We can now prove:

\begin{theorem}[Reformulation of \eqref{eq:P}]
\label{th:truncate}
Let $r \geq r_P$ and $\bar u \in L^2(\R)$ 
be given. 
Then $\bar u$ is a local solution of \eqref{eq:P}
with radius of optimality $\rho $
if and only if  the following two conditions are satisfied:
\begin{enumerate}[label=\roman*)]
\item\label{th:truncate:item:i} It holds 
\begin{equation}
\label{eq:projoutside}
\bar u = \max\left ( 0, u_D - \frac{\nu_1}{\nu_2} \right ) \text{ a.e.\ in } \R \setminus (-r,r).
\end{equation}
\item\label{th:truncate:item:ii} The restriction
$\bar \uu := \bar u|_{(-r,r)} \in L^2(-r,r)$ is a local solution 
with radius of optimality $\rho$
of the problem
\begin{equation*}
\label{eq:Pr}
\tag{P$_r$}
\left.
	\begin{aligned}
		\text{Minimize} 
		\quad & F_r(\uu) = \frac12 \|S_r(\uu) - y_D\|_{L^2(\Omega)}^2 
		+ \int_{-r}^r \nu_1  \uu + \frac{\nu_2}{2}( \uu - \uu_D)^2\dd t \\
        \text{w.r.t.}
        \quad &\uu \in L^2(-r,r),\\
		\text{s.t.} \quad &  \uu \geq 0 \text{ a.e.\  in } (-r,r).
	\end{aligned}~~
\right \}
\end{equation*}
\end{enumerate}
\end{theorem}

\begin{proof}
Suppose that $\bar u \in L^2(\R)$ is a local solution of \eqref{eq:P}
and that $v \in L^1(\R) \cap L^2(\R)$ is an arbitrary function
satisfying $v = 0$ a.e.\ in $(-r,r)$ and $\bar u + s_0 v \geq 0$ a.e.\ in $\R$ for some $s_0 > 0$.
Then it holds $0 \leq \bar u + s  v \in L^2(\R)$ for all $0 < s < s_0$,
it follows from the  local optimality of $\bar u$ that, 
for all sufficiently small $s>0$, 
we have $J(S(\bar u + s  v), \bar u + s  v) \geq J(S(\bar u), \bar u)$,
and we obtain from  $\bar u + s v = \bar u$ a.e.\ in $(-r,r)$,
 $r \geq r_P$, and \eqref{eq:truncid} that  $S(\bar u + s  v) = S(\bar u)$ holds for all $0 < s < s_0$.
In combination with the definition of $J$, this implies
\begin{equation}
\label{eq:randomineq1635sg36z}
\begin{aligned}
J(S(\bar u + s  v), \bar u + s  v) &= \frac12 \| S(\bar u) - y_D\|_{L^2(\Omega)}^2 + \nu_1 \|\bar u + s  v\|_{L^1(\R)} 
+ \frac{\nu_2}{2} \|\bar u + s  v - u_D\|_{L^2(\R)}^2 
\\
\geq J(S(\bar u), \bar u) &=
\frac12 \| S(\bar u) - y_D\|_{L^2(\Omega)}^2 + \nu_1 \|\bar u  \|_{L^1(\R)} 
+ \frac{\nu_2}{2} \|\bar u  - u_D\|_{L^2(\R)}^2 
\end{aligned}
\end{equation}
for all sufficiently small $s > 0$ and, as a consequence, 
\begin{equation}
\label{eq:randomeq63535}
\int_\R  \nu_1
   s  v  +  \frac{\nu_2}{2} (\bar u + s  v - u_D)^2 - \frac{\nu_2}{2} (\bar u  - u_D)^2 \dd x \geq 0. 
\end{equation} 
Here, we have exploited
the nonnegativity of the functions $\bar u$ and $\bar u + s v$
to rewrite the 
$L^1$-norms in \eqref{eq:randomineq1635sg36z}.
If we use the binomial identities in \eqref{eq:randomeq63535},
divide by $s\nu_2$, and let $s$ go to zero, then \eqref{eq:randomeq63535} yields
\begin{equation}
\label{eq:randomeq263636-2673gdge7ebd}
\int_\R  \left ( \frac{\nu_1}{\nu_2} - u_D + \bar u \right )v\dd x \geq 0. 
\end{equation} 
By considering functions $v$ of the form $v = \mathds{1}_{\{\bar u \geq \varepsilon\} \cap \R \setminus (-r,r)} w $
with arbitrary $\varepsilon > 0$ and $w \in L^1(\R) \cap L^\infty(\R)$ in \eqref{eq:randomeq263636-2673gdge7ebd}, 
we obtain that $\bar u$ satisfies $\bar u = u_D - \nu_1/\nu_2$  
a.e.\ in $\{\bar u > 0\} \cap \R \setminus (-r,r)$,
and by considering functions of the form $v = \mathds{1}_{\{\bar u = 0\} \cap \R \setminus (-r,r)} w $
with $0 \leq w \in L^1(\R) \cap L^2(\R)$,
that $u_D - \nu_1/\nu_2 \leq 0$ holds a.e.\ in  $\{\bar u = 0\} \cap \R \setminus (-r,r)$.
This establishes the identity \eqref{eq:projoutside} in condition \ref{th:truncate:item:i}.

Next, we prove that $\bar u$ also satisfies \ref{th:truncate:item:ii}:
Denote the radius of local optimality of  $\bar u$ 
with  $\rho > 0$  and assume that
$0 \leq \uu \in L^2(-r,r)$ is an arbitrary function 
such that $\|\bar \uu - \uu\|_{L^2(-r,r)} \leq \rho$ holds for 
$\bar \uu := \bar u|_{(-r,r)}$. 
We define $u := E_r(\uu) + \mathds{1}_{\R \setminus (-r,r)} \bar u$. 
By construction, the function $u$ satisfies $0 \leq u \in L^2(\R)$
and  $\|u - \bar u\|_{L^2(\R)} = \|\bar \uu - \uu\|_{L^2(-r,r)} \leq \rho$.
In combination with the definition of $\rho$, this yields
\begin{equation}
\label{eq:randomeq2636dge63h}
\frac12 \| S(u) - y_D\|_{L^2(\Omega)}^2 + \nu_1 \|u\|_{L^1(\R)} + \frac{\nu_2}{2} \|u  - u_D\|_{L^2(\R)}^2 
\geq
\frac12 \| S(\bar u) - y_D\|_{L^2(\Omega)}^2 + \nu_1 \|\bar u\|_{L^1(\R)} 
+ \frac{\nu_2}{2} \|\bar u  - u_D\|_{L^2(\R)}^2.
\end{equation}
Since $u$ is identical to $\bar u$ a.e.\ in $\R \setminus (-r,r)$,
since
$S(\tilde u) = S_r(\tilde u|_{(-r,r)})$ holds for all $0 \leq \tilde u \in L^2(\R)$ by \eqref{eq:truncid},
and since $u$ and $\bar u$ are nonnegative, 
 the
inequality \eqref{eq:randomeq2636dge63h} can be recast as
$F_r(\uu) \geq F_r(\bar \uu)$.
As $\uu$ was arbitrary, this shows that  $\bar \uu = \bar u|_{(-r,r)} \in L^2(-r,r)$ is indeed a local solution of \eqref{eq:Pr}
(with the same radius of optimality $\rho$ as $\bar u$). This completes the proof 
of the direction ``$\Rightarrow$'' in the asserted equivalence.

It remains to prove the reverse implication. To this end, let us assume that a 
function $\bar u \in L^2(\R)$ satisfying  the conditions \ref{th:truncate:item:i}
and \ref{th:truncate:item:ii} is given. For such a function $\bar u$, 
we obtain from the properties in \ref{th:truncate:item:i}
and \ref{th:truncate:item:ii} that $\bar u \geq 0$ holds a.e.\ in $\R$ and from 
\eqref{eq:projoutside} and a simple pointwise calculation that 
\begin{equation}
\label{eq:randomeq2723hge7wbs}
\begin{aligned}
&\nu_1 \|u \|_{L^1(\R \setminus (-r, r))} 
+ \frac{\nu_2}{2} \|u  - u_D\|_{L^2(\R \setminus (-r, r))}^2 
\\
&\qquad\qquad
\geq
\nu_1 \|\bar u  \|_{L^1(\R \setminus (-r,r))} 
+ \frac{\nu_2}{2} \|\bar u  - u_D\|_{L^2(\R \setminus (-r, r))}^2 
\quad\forall\, 0 \leq u \in L^2(\R \setminus (-r,r)). 
\end{aligned}
\end{equation}
Consider now some $0 \leq u \in L^2(\R)$
that satisfies $\|u - \bar u\|_{L^2(\R)} \leq \rho$
for  the radius of local optimality $\rho > 0$ of $\bar \uu := \bar u|_{(-r,r)} \in L^2(-r,r)$ in 
\eqref{eq:Pr}. We define $\uu := u|_{(-r,r)}$.
For this $\uu$, it clearly holds $0 \leq \uu \in L^2(-r,r)$
and $\|\bar \uu - \uu\|_{L^2(-r,r)} \leq \rho$ and we obtain from 
condition \ref{th:truncate:item:ii} that
\begin{equation*}
\begin{aligned}
&F_r(\uu) = \frac12 \|S_r(\uu) - y_D\|_{L^2(\Omega)}^2 
		+ \nu_1 \|\uu\|_{L^1(-r,r)} + \frac{\nu_2}{2} \|\uu - \uu_D\|_{L^2(-r,r)}^2
		\\
		&\qquad\qquad 
		\geq F_r(\bar \uu) = 
		\frac12 \|S_r(\bar \uu) - y_D\|_{L^2(\Omega)}^2 
		+ \nu_1 \|\bar \uu\|_{L^1(-r,r)} + \frac{\nu_2}{2} \|\bar \uu - \uu_D\|_{L^2(-r,r)}^2.
\end{aligned}
\end{equation*}
Due to \eqref{eq:truncid} and the definitions of $\uu$, $\bar \uu$, and $\uu_D$, the above inequality can be rewritten as 
\begin{equation*}
\frac12 \|S(u) - y_D\|_{L^2(\Omega)}^2 
		+ \nu_1 \| u\|_{L^1(-r,r)} + \frac{\nu_2}{2} \| u-u_D \|_{L^2(-r,r)}^2
		\geq 
		\frac12 \|S(\bar u) - y_D\|_{L^2(\Omega)}^2 
		+ \nu_1 \| \bar u\|_{L^1(-r,r)} + \frac{\nu_2}{2} \| \bar u-u_D \|_{L^2(-r,r)}^2.
\end{equation*}
In combination with \eqref{eq:randomeq2723hge7wbs}, this yields 
$J(S(u), u) \geq J(S(\bar u), \bar u)$. Since 
$0 \leq u \in L^2(\R)$
was an arbitrary function satisfying $\|u - \bar u\|_{L^2(\R)} \leq \rho$,
we may now conclude that $\bar u$ is indeed
a local solution of \eqref{eq:P} with radius of 
optimality $\rho$. This establishes the implication ``$\Leftarrow$'' 
of the asserted equivalence and 
completes the proof.
\end{proof}

Note that the inequality $r \geq r_P$ in \cref{th:truncate} is  only needed
for proving the equivalence of the problems \eqref{eq:P} and 
\eqref{eq:Pr}. For \eqref{eq:Pr} to be sensible, it suffices to assume that $r$ is positive. 
In what follows, when studying the problem \eqref{eq:Pr}, 
we work with the relaxed assumption $r>0$ whenever possible. 
For later use, we prove:

\begin{proposition}[Enforcing Global Optimality]
\label{prop:localglobal}
Let $r>0$ be given and let
$0 \leq \bar \uu \in L^2(-r,r)$ be a local solution of the problem \eqref{eq:Pr}.
Then there exists a constant $C_{\bar \uu} > 0$
such that 
$\bar \uu$ is the unique global solution of the problem
\begin{equation*}
\label{eq:Pru}
\tag{P$_{r, \bar{\uu}}$}
\left.
	\begin{aligned}
		\text{Minimize} 
		\quad & F_r(\uu) + \frac{C_{\bar \uu}}{2} \|\uu - \bar \uu\|_{L^2(-r,r)}^2
		\\
        \text{w.r.t.}
        \quad &\uu \in L^2(-r,r),\\
		\text{s.t.} \quad &  \uu \geq 0 \text{ a.e.\  in } (-r,r).
	\end{aligned}~~
\right \}
\end{equation*}
\end{proposition}

\begin{proof}
Due to the local optimality of $\bar \uu$ in \eqref{eq:Pr}, 
we know that there exists $\rho > 0$ satisfying
\begin{equation}
\label{eq:randomeq26353}
F_r(\uu) \geq F_r(\bar \uu)
\qquad \forall\, 0 \leq \uu \in L^2(-r, r),  \|\uu - \bar \uu\|_{L^2(-r,r)} \leq \rho. 
\end{equation}
Define $C_{\bar \uu} := 1 + 2  F_r(\bar \uu)/ \rho^2$. 
Then it follows from the nonnegativity of $F_r$ on $L^2_+(-r,r)$
that 
\begin{equation}
\label{eq:randomeq2636gr6}
F_r(\bar \uu)  < \frac{C_{\bar \uu}}{2}\rho^2 
\leq \frac{C_{\bar \uu}}{2} \|\uu - \bar \uu\|_{L^2(-r,r)}^2
\leq
F_r(\uu)+ \frac{C_{\bar \uu}}{2} \|\uu - \bar \uu\|_{L^2(-r,r)}^2
\end{equation}
holds for all $0 \leq \uu \in L^2(-r,r)$ with $\|\uu - \bar \uu\|_{L^2(-r,r)}\geq \rho$. 
By combining 
\eqref{eq:randomeq26353} and \eqref{eq:randomeq2636gr6} 
and by distinguishing cases, the assertion of the proposition follows.
\end{proof}

Since the function $g_{u}$ is only Hölder continuous with exponent $1/2$ for arbitrary $u \in L^2(\R)$, 
we cannot state first-order necessary optimality conditions for the problems \eqref{eq:P}, 
\eqref{eq:Pr}, and \eqref{eq:Pru}
 at this point.
To do so, we require additional information about the regularity properties 
of local minimizers. 

\section{Regularity of Optimal Controls}
\label{sec:5}

The purpose of this section is to prove that
minimizers of  the problems \eqref{eq:P}
and \eqref{eq:Pr} are regulated if $u_D$ is a smooth enough function.  
As we will see in \cref{sec:6,sec:7,sec:8}, this type of regularity is crucial for 
the derivation of stationarity conditions and the design of numerical solution algorithms.
We consider the following situation:

\begin{assumption}[Standing Assumptions for \Cref{sec:5}] 
\label{ass:StandSec5}
Throughout this section, we assume that:
\begin{itemize}
\item $r_P$, $\varepsilon_P$, and $U_P$ are defined as before (see \cref{def:Pdefs});
\item $r$ is a given positive number;
\item $\uu_D$, $E_r$, $U_{P,r}$, $S_r$, and $F_r$ are defined as in \cref{def:ExtZero};
\item $\uu_D$ satisfies $\uu_D \in L^\infty_G(-r,r)$;
\item $0 \leq \bar \uu \in L^2(-r,r)$ is a given local solution of the problem \eqref{eq:Pr};
\item  $C_{\bar \uu} > 0$ is a constant such that $\bar \uu$ and $C_{\bar \uu}$ 
satisfy the assertion of \cref{prop:localglobal};
\item $\eta\colon (-\varepsilon_P, \infty) \to \R$  is defined by $\eta(t) :=
\min(0,t)^2/(t + \varepsilon_P)$ for all $t \in (-\varepsilon_P, \infty)$.
\end{itemize}
\end{assumption}

Note that, for the function $\eta$, we have:

\begin{lemma}[Properties of $\eta$]
\label{lemma:EtaProp}
The function $\eta$  satisfies $\eta \in C^1((-\varepsilon_P, \infty))$,
$\eta(t) = 0$ for all $t \in [0, \infty)$, $\eta(t) \geq \min(0,t)^2/\varepsilon_P$ for all $t \in (- \varepsilon_P, \infty)$,
and
$\eta(t) \to \infty$ for $(-\varepsilon_P, \infty) \ni t \to -\varepsilon_P$.
Further, $\eta$ is convex.
\end{lemma}
\begin{proof}
This can be checked with trivial calculations and estimates.
\end{proof} 

The main idea of the subsequent analysis is to relax the constraint $\uu \geq 0$ a.e.\ in $(-r,r)$
in \eqref{eq:Pru} by means of the barrier function $\beta \eta \colon (-\varepsilon_P, \infty) \to \R$, $\beta > 0$,
to establish a $BV$-regularity result for the adjoint variables of the resulting family of 
approximate problems, and to transfer this $BV$-regularity  
to the original problem \eqref{eq:Pru} by passing to the limit $\beta \to \infty$. 
For a given $\beta > 0$, we thus consider the following problem:
\begin{equation*}
\label{eq:Prub}
\tag{P$_{r, \bar{\uu}, \beta }$}
\left.
	\begin{aligned}
		\text{Minimize} 
		\quad & F_r(\uu) + \frac{C_{\bar \uu}}{2} \|\uu - \bar \uu\|_{L^2(-r,r)}^2
		+
		\int_{-r}^r
		\beta \eta (\uu) \dd t
		\\
        \text{w.r.t.}
        \quad &\uu \in L^2(-r,r),\\
		\text{s.t.} \quad &  \uu > - \varepsilon_P \text{ a.e.\  in } (-r,r).
	\end{aligned}~~
\right \}
\end{equation*}
We remark that, for optimal control problems governed by variational inequalities, 
regularization approaches similar to that developed in this section 
have already been used, e.g., in \cite{Barbu1984}. We first note:

\begin{proposition}[Solvability of \eqref{eq:Prub}]
For every $\beta > 0$, 
 \eqref{eq:Prub}  possesses a global solution 
$\bar \uu_\beta \in L^2(-r, r)$.
\end{proposition}

\begin{proof}
This follows completely analogously to the proof of \cref{th:solexistence} from the direct method 
of the calculus of variations, the results on
the control-to-state mapping $S\colon U_P \to H_0^1(\Omega) \cap H^2(\Omega)$ 
in \cref{th:propS}, the definitions of the functions $S_r$ and $F_r$, and the properties of 
$\eta$. 
\end{proof}

Due to the $C_{\bar \uu}$-term in the objective of \eqref{eq:Prub} and \cref{prop:localglobal}, we 
have: 
\begin{proposition}[Convergence of Approximate Solutions]
\label{prop:rg1}
Consider for every $\beta > 0$ a global solution
$\bar \uu_\beta \in L^2(-r, r)$ of \eqref{eq:Prub}. Then
it holds  $\bar \uu_\beta \to \bar \uu$ in $L^2(-r,r)$ for $\beta \to \infty$.
\end{proposition}

\begin{proof}
From the
global optimality of $\bar \uu_\beta $ in \eqref{eq:Prub},
the inequality $\bar \uu_\beta > - \varepsilon_P$ a.e.\ in $(-r,r)$,
the nonnegativity of $\bar \uu$, 
the definition of $F_r$,
and the properties of $\eta$ in \cref{lemma:EtaProp}, 
we obtain that 
\begin{equation*}
\begin{aligned}
		F_r(\bar \uu)
		 &=
		F_r(\bar \uu) + \frac{C_{\bar \uu}}{2} \|\bar \uu - \bar \uu\|_{L^2(-r,r)}^2
		+
		\int_{-r}^r
		\beta \eta (\bar \uu) \dd t
		\\
	 &\geq F_r(\bar \uu_\beta ) 
	+ \frac{C_{\bar \uu}}{2} \|\bar \uu_\beta  - \bar \uu\|_{L^2(-r,r)}^2
		+
		\int_{-r}^r
		\beta \eta (\bar \uu_\beta ) \dd t
		\\
		&=
		 \frac12 \|S_r(\bar \uu_\beta)- y_D\|_{L^2(\Omega)}^2 
		+ \int_{-r}^r \nu_1  \bar \uu_\beta  + \frac{\nu_2}{2}( \bar \uu_\beta  - \uu_D)^2\dd t
		+ \frac{C_{\bar \uu}}{2} \|\bar \uu_\beta  - \bar \uu\|_{L^2(-r,r)}^2
		+
		\int_{-r}^r
		\beta \eta (\bar \uu_\beta ) \dd t
	\\
	&\geq
	-2r \nu_1 \varepsilon_P + \frac{C_{\bar \uu}}{2} \|\bar \uu_\beta  - \bar \uu\|_{L^2(-r,r)}^2
	+\frac{\beta}{\varepsilon_P}
	\int_{-r}^r
	\min(0,\bar \uu_\beta)^2 \dd t\qquad \forall \beta > 0.
\end{aligned}
\end{equation*}
As $F_r(\bar \uu)$ is finite (cf.\ the comments after \cref{def:localMin}),
the above shows that the family $\{\bar \uu_\beta \}$ is bounded in $L^2(-r,r)$
and that the functions $\min(0,\bar \uu_\beta)$ converge strongly to zero in $L^2(-r,r)$
for $\beta \to \infty$. 
This implies in particular that, by passing over to a sequence $\{\beta_n\} \subset (0, \infty)$
with $\beta_n \to \infty$ for $n \to \infty$, we can achieve 
that $\bar \uu_{\beta_n} \weakly \hat \uu$ holds in  $L^2(-r,r)$  
for some $\hat \uu \in L^2(-r,r)$ as  $n$ tends to infinity. 
Due to the convergence $\min(0,\bar \uu_\beta) \to 0$
in $L^2(-r,r)$, we know that $\hat \uu$ is nonnegative a.e.\ in $(-r,r)$
and, thus, admissible for \eqref{eq:Pru}.
Since $\bar \uu$ solves \eqref{eq:Pru},
since $S$ is completely continuous by \cref{th:propS}\ref{th:propS:item:i}, 
since convexity and lower semicontinuity 
imply weak lower semicontinuity, 
and again by the global optimality of $\bar \uu_\beta$ in \eqref{eq:Prub},
it now follows that 
\begin{equation*}
\begin{aligned}
F_r(\hat \uu) + \frac{C_{\bar \uu}}{2} \|\hat \uu - \bar \uu\|_{L^2(-r,r)}^2
&\geq
F_r(\bar \uu) + \frac{C_{\bar \uu}}{2} \|\bar \uu - \bar \uu\|_{L^2(-r,r)}^2
\\
&\geq
\limsup_{n \to \infty }
\left (
F_r(\bar \uu_{\beta_n}  ) 
	+ \frac{C_{\bar \uu}}{2} \|\bar \uu_{\beta_n}   - \bar \uu\|_{L^2(-r,r)}^2
		+
		\int_{-r}^r
		\beta_n \eta (\bar \uu_{\beta_n}  ) \dd t
		\right )
		\\
		&\geq
F_r(\hat \uu ) 		
	+ \frac{C_{\bar \uu}}{2} 
	\limsup_{n \to \infty }\left (
	\|\bar \uu_{\beta_n}   - \bar \uu\|_{L^2(-r,r)}^2
		\right )
		\geq
F_r(\hat \uu) + \frac{C_{\bar \uu}}{2} \|\hat \uu - \bar \uu\|_{L^2(-r,r)}^2.
\end{aligned}
\end{equation*}
The above can only be true if all appearing inequalities are identities.
Since $\bar \uu$ is the unique global 
solution of \eqref{eq:Pru}, this implies, on the one hand, 
that $\bar \uu = \hat \uu$ holds and, on the other hand, that 
\[
	\limsup_{n \to \infty }\left (
	\|\bar \uu_{\beta_n}   - \bar \uu\|_{L^2(-r,r)}^2
		\right )
		=
		\|\hat \uu  - \bar \uu\|_{L^2(-r,r)}^2 = 0.
\]
The sequence $\{\bar \uu_{\beta_n} \}$ thus converges strongly in $L^2(-r,r)$ to $\bar \uu$.
Since the above arguments can be repeated for all sequences $\{\bar \uu_{\beta_n}\}$
satisfying $\bar \uu_{\beta_n} \weakly \hat \uu$ for some $\hat \uu \in L^2(-r,r)$
and $(0, \infty) \ni \beta_n \to \infty$ for $n \to \infty$, 
the convergence of the whole family $\{\bar \uu_\beta\}$ to 
$\bar \uu$ for $\beta \to \infty$ 
now follows from a trivial contradiction argument; cf.\ \cite[Lemma~4.16]{Schweizer2013}.
This completes the proof.
 \end{proof}

By exploiting the estimate in \cref{th:propS}\ref{th:propS:item:iii}, we can prove 
the following for the relaxed problems 
\eqref{eq:Prub}: 

\begin{proposition}[$L_{BV}^\infty$-Regularity of Adjoint Variables]
\label{prop:rg2}
Let $\beta > 0$ be given and
suppose that $\bar \uu_\beta$ is a local solution of \eqref{eq:Prub}.
Denote with $\bar \pp_\beta$ the element of $L^0(-r,r)$  defined  by 
\begin{equation*}
\bar \pp_\beta := 
\nu_1  
+\nu_2 (\bar \uu_\beta - \uu_D)
+
C_{\bar \uu}   (\bar \uu_\beta - \bar \uu)  
+
\beta \eta'(\bar \uu_\beta).
\end{equation*}
Then it holds $\bar \pp_\beta \in L^\infty_{BV}(-r,r)$ and $\bar \pp_\beta$
admits a $BV[-r,r]$-representative that satisfies 
\begin{equation}
\label{eq:BVbound}
\|\bar \pp_\beta \|_{BV[-r,r]} \leq K
\end{equation}
for a constant $K > 0$
which depends only on $r$, $y_D$, $\Omega$, 
and the numbers $M, c> 0$ in \cref{th:propS}.
\end{proposition}

\begin{proof}
Consider some $\zz \in L^\infty(-r,r)$
and define $\zz_n := \mathds{1}_{\{\bar \uu_\beta \geq -\varepsilon_P + 1/n\}}\zz$, $n \in \mathbb{N}$.  
Then it holds
\begin{equation}
\label{eq:randomestimate989283u}
\bar \uu_\beta  - s \zz_n
\begin{cases}
= \bar \uu_\beta > -\varepsilon_P &\text{a.e.\ in } \{\bar \uu_\beta < -\varepsilon_P + 1/n\},
\\
\geq 
- \varepsilon_P + 1/(2n) &\text{a.e.\ in } \{\bar \uu_\beta \geq -\varepsilon_P + 1/n\},
\end{cases}
\qquad\qquad \forall \,0 < s < (2 n\|\zz\|_{L^\infty(-r,r)} + 1)^{-1},
\end{equation}
and it 
follows from the local optimality of $\bar \uu_\beta$ in \eqref{eq:Prub}
and the definition of $F_r$ that 
$\bar \uu_\beta$ satisfies 
\begin{equation*}
\begin{aligned}
&\frac12 \| S_r(\bar \uu_\beta) - y_D\|_{L^2(\Omega)}^2 
+ \int_{-r}^r \nu_1  \bar \uu_\beta + \frac{\nu_2}{2}( \bar \uu_\beta - \uu_D)^2
+  \frac{C_{\bar \uu}}{2} (\bar \uu_\beta  - \bar \uu)^2
+
\beta \eta (\bar \uu_\beta)
\dd t
\\
&\qquad\qquad\leq
\frac12 \| S_r(\bar \uu_\beta  - s \zz_n) - y_D\|_{L^2(\Omega)}^2 
+ \int_{-r}^r \nu_1  (\bar \uu_\beta  - s \zz_n)+ \frac{\nu_2}{2}( \bar \uu_\beta  - s \zz_n - \uu_D)^2
+  \frac{C_{\bar \uu}}{2} (\bar \uu_\beta  - s \zz_n  - \bar \uu)^2
\\
&\hspace{12.59cm}+
\beta \eta (\bar \uu_\beta  - s \zz_n)
\dd t
\end{aligned}
\end{equation*}
for all small enough $s>0$.
If we use the binomial identities in the above, divide by $s$, and
rearrange, then we obtain that 
\begin{equation}
\label{eq:randomeq2636g36}
\begin{aligned}
&
 \int_{\{\bar \uu_\beta \geq -\varepsilon_P + 1/n\}} \nu_1  \zz 
+\nu_2 (\bar \uu_\beta  - \uu_D) \zz -  \frac{\nu_2}{2}s  \zz ^2 
+
C_{\bar \uu} (\bar \uu_\beta  - \bar \uu) \zz  -
\frac{C_{\bar \uu}}{2} s  \zz ^2 
+  \beta  \frac{\eta(\bar \uu_\beta) -    \eta(\bar \uu_\beta - s \zz )}{s} \dd t
\\
&\qquad\qquad\leq 	
\int_\Omega 
(S_r(\bar \uu_\beta) - y_D)
\frac{S_r(\bar \uu_\beta  - s \zz_n) - S_r(\bar \uu_\beta)}{s}
\dd x 
+
\frac{s}{2} \int_\Omega
\left (\frac{S_r(\bar \uu_\beta  - s \zz_n) - S_r(\bar \uu_\beta)}{s}\right )^2
\dd x.
\end{aligned}
\end{equation}
Due to \cref{th:propS}\ref{th:propS:item:iii}, the linearity of $g_u$ in $u$, 
the inequality of Cauchy-Schwarz, 
and the definition of $S_r$, 
we know that 
there exists a constant $c>0$ independent of $\beta$ and $s$ such that 
\begin{equation*}
\|S_r(\bar \uu_\beta  - s \zz_n) - S_r(\bar \uu_\beta)\|_{L^2(\Omega)}
\leq 
\frac{s}{c}
\|
g_{E_r(\zz_n)}(S_r ( \bar \uu_\beta ))
\|_{L^2(\Omega)}
\qquad\qquad \forall \,0 < s < (2 n\|\zz\|_{L^\infty(-r,r)} + 1)^{-1}.
\end{equation*}
In combination with the properties of $\eta$,
the second inequality in \eqref{eq:randomestimate989283u},
the dominated convergence theorem, and \eqref{eq:defgu},  the above allows us to bound the 
right-hand side of \eqref{eq:randomeq2636g36} and to subsequently pass to the limit  
$s \to 0$. This yields
\begin{equation}
\label{eq:randomestimate273646}
 \int_{\{\bar \uu_\beta \geq -\varepsilon_P + 1/n\}} \nu_1  \zz 
+\nu_2 (\bar \uu_\beta  - \uu_D) \zz 
+
C_{\bar \uu} (\bar \uu_\beta  - \bar \uu) \zz 
+  \beta  \eta'(\bar \uu_\beta)\zz \dd t
\leq 	
\frac{1}{c}
\|S_r(\bar \uu_\beta) - y_D\|_{L^2(\Omega)}
\|
g_{E_r(\zz_n)}(S_r ( \bar \uu_\beta ))
\|_{L^2(\Omega)}. 
\end{equation}
Since
$\|S_r(\bar \uu_\beta) - y_D\|_{L^2(\Omega)} \leq M + \|y_D\|_{L^2(\Omega)}$
holds for the constant $M>0$ in \eqref{eq:H2estimate}
and since the function 
$\nu_1  
+\nu_2 (\bar \uu_\beta - \uu_D)
+
C_{\bar \uu}   (\bar \uu_\beta - \bar \uu)  
+
\beta \eta'(\bar \uu_\beta)\in L^0(-r,r)$
appearing on the left-hand side of \eqref{eq:randomestimate273646} 
is equal to $\bar \pp_\beta $ by definition, 
we may now conclude that,
 for all $\zz \in L^\infty(-r,r)$ and all $n \in \mathbb{N}$, we have 
\begin{equation}
\label{eq:randomeq52355}
\begin{aligned}
&
 \int_{\{\bar \uu_\beta \geq -\varepsilon_P + 1/n\}} 
  \bar \pp_\beta  \zz 
 \dd t
  \leq 	
\frac{M + \|y_D\|_{L^2(\Omega)}}{c}
\left (
\int_\Omega 
\left (
\int_0^{S_r ( \bar \uu_\beta )(x)}
(E_r \zz_n)(t) \dd t 
\right )^2
\dd x
\right )^{1/2}.
\end{aligned}
\end{equation}
Recall that,  due to \cref{th:propS}\ref{th:propS:item:ii}
and the embedding $H^2(\Omega) \hookrightarrow C(\overline{\Omega})$, 
there exists a constant $R = R(M, \Omega) > 0$
independent of $\beta$
satisfying
$\|S_r ( \bar \uu_\beta )\|_{\infty}\leq R$.
If we use this bound in \eqref{eq:randomeq52355} 
and consider the special situation $\zz := \sgn(\bar \pp_\beta) \in L^\infty(-r,r)$,
then it follows that 
\begin{equation}
\label{eq:randomeq2636}
 \int_{-r}^r |\bar \pp_\beta| \mathds{1}_{\{\bar \uu_\beta \geq -\varepsilon_P + 1/n\}}
 \dd t
  \leq 
  R
  \left (
  \frac{M + \|y_D\|_{L^2(\Omega)}}{c}
  \right )	 \lambda^d(\Omega)^{1/2}\qquad \forall n \in \mathbb{N}. 
\end{equation}
Due to the lemma of Fatou and since 
$\bar \uu_\beta > - \varepsilon_P$ holds  a.e.\ in $(-r,r)$,
we may let $n$ go to $\infty$ in \eqref{eq:randomeq2636}.
This yields that $\|\bar \pp_\beta\|_{L^1(-r,r)} \leq K_1$ holds 
for the constant 
\[
K_1 = K_1(c,M,\Omega,y_D)
:= R
  \left (
  \frac{M + \|y_D\|_{L^2(\Omega)}}{c}
  \right )	 \lambda^d(\Omega)^{1/2}.
\]
By exploiting this $L^1$-regularity of $\bar \pp_\beta$,
the dominated convergence theorem, 
and again the fact that 
 $\bar \uu_\beta > - \varepsilon_P$ holds a.e.\ in $(-r,r)$,
we can pass to the limit $n \to \infty$ in \eqref{eq:randomeq52355} for arbitrary 
$\zz\in L^\infty(-r,r)$
to obtain
\begin{equation}
\label{eq:randomeq62535}
\begin{aligned}
&
 \int_{-r}^r 
  \bar \pp_\beta  \zz 
 \dd t
  \leq 	
\frac{M + \|y_D\|_{L^2(\Omega)}}{c}
\left (
\int_\Omega 
\left (
\int_0^{S_r ( \bar \uu_\beta )(x)}
(E_r \zz)(t) \dd t 
\right )^2
\dd x
\right )^{1/2}
\qquad\qquad \forall \zz \in L^\infty(-r,r).
\end{aligned}
\end{equation}
Let us now consider functions $\zz$ of the form 
 $\zz = \vv'$, $\vv\in C^1([-r,r])$. For such $\zz$, 
 it follows from 
 the fundamental theorem of calculus, the definition of $E_r$,
 and a trivial distinction of cases that 
 \begin{equation*}
 \left |
 \int_0^{S_r ( \bar \uu_\beta )(x)}
(E_r \vv')(t) \dd t 
 \right |
 \leq
 \begin{cases}
 \left |
 \vv\left ( S_r ( \bar \uu_\beta )(x) \right )
 \right |
 +
  \left |
 \vv\left ( 0 \right )
 \right |
 & \text{ if } S_r ( \bar \uu_\beta )(x) \in [-r,r],
 \\
  \left |
 \vv\left ( r \right )
 \right |
 +
  \left |
 \vv\left ( 0 \right )
 \right |
  & \text{ if } S_r ( \bar \uu_\beta )(x) > r,
   \\
  \left |
 \vv\left ( -r \right )
 \right |
 +
  \left |
 \vv\left ( 0 \right )
 \right |
  & \text{ if } S_r ( \bar \uu_\beta )(x) < - r,
 \end{cases}
 \qquad \text{f.a.a.\ } x \in \Omega.
 \end{equation*}
 In combination with \eqref{eq:randomeq62535}, the above implies
 \begin{equation}
 \label{eq:randomeq3673646}
 \begin{aligned}
  \int_{-r}^r 
  \bar \pp_\beta  \vv'
 \dd t
  &\leq 	
\frac{M + \|y_D\|_{L^2(\Omega)}}{c}
\left (
\int_\Omega 
\left (
\int_0^{S_r ( \bar \uu_\beta )(x)}
(E_r\vv')(t) \dd t 
\right )^2
\dd x
\right )^{1/2}
\\
&\leq
\frac{M + \|y_D\|_{L^2(\Omega)}}{c}
\left (
\int_\Omega 
\left (
2 \|\vv\|_{C([-r,r])}
\right )^2
\dd x
\right )^{1/2}
\\
&=
K_2 \|\vv\|_{C([-r,r])}
\qquad \forall \vv \in C^1([-r,r]),
 \end{aligned}
 \end{equation}
 where $K_2 = K_2(c, M, \Omega, y_D)$ is the constant defined by 
 \[
 K_2 := 2 
\left (
\frac{M + \|y_D\|_{L^2(\Omega)}}{c}\right ) \lambda^d(\Omega)^{1/2}.
 \]
Using the theorem of Hahn-Banach (see \cite[Theorem 8.1.1]{Monteiro2019}), 
we may deduce from \eqref{eq:randomeq3673646}  that the linear map 
$ C^1([-r,r]) \ni \vv \mapsto  \int_{-r}^r 
  \bar \pp_\beta  \vv'
 \dd t \in \R $ can be extended to a linear and continuous functional on 
 $C([-r,r])$. In combination with the Riesz representation theorem
 (see \cite[Theorem 8.1.2]{Monteiro2019}), this yields that there exists a 
 function $\qq \in BV[-r,r]$ which satisfies 
 $\qq(-r) = 0$, $\var(\qq; [-r,r]) \leq K_2$, and 
 \begin{equation}
 \label{eq:randomeq23673zdebg4e7}
 \int_{-r}^r 
  \bar \pp_\beta  \vv'
 \dd t
 =
 \int_{-r}^r 
\vv
 \dd \qq\qquad \forall \vv \in C^1([-r,r]).
 \end{equation}
 Here, the integral on the right has to be understood in the sense 
 of Riemann-Stieltjes. If we choose test functions $\vv \in C_c^\infty(-r,r)$
 in \eqref{eq:randomeq23673zdebg4e7}
  and integrate by parts by means of \cite[Theorem 5.5.1]{Monteiro2019}, then we obtain that 
 \[
  \int_{-r}^r 
  (\bar \pp_\beta + \qq) \vv'
 \dd t = 0 \qquad \forall \vv \in C_c^\infty(-r, r). 
 \]
This shows that $\bar \pp_\beta + \qq = \mathrm{const}$ holds a.e.\ in $(-r,r)$
and that $\bar \pp_\beta$ possesses a representative that satisfies 
 $\bar \pp_\beta \in BV[-r,r]$ and $\var(\bar \pp_\beta; [-r,r]) = \var(\qq; [-r,r]) \leq K_2$.
Henceforth, with the symbol $\bar \pp_\beta$, we always refer to this representative. 
 Note that, due to the bound  $\|\bar \pp_\beta\|_{L^1(-r,r)} \leq K_1$,
 we know that there exists at least one $s \in [-r,r]$ with $|\bar \pp_\beta(s)| \leq K_1/(2r)$.
 In combination with the definition of the variation $\var(\bar \pp_\beta; [-r,r])$
 and our previous estimates,
 this allows us to conclude that 
 \[
 \|\bar \pp_\beta \|_{BV[-r,r]}
 =
 	\left |\bar \pp_\beta(-r)
 	\right | 
 	+
 	\var(\bar \pp_\beta; [-r,r]) 
 	\leq  |\bar \pp_\beta(s)| + 2\var(\bar \pp_\beta; [-r,r])
 	\leq
 	\frac{K_1}{2 r} + 2K_2.
 \]
 If we define $K := \frac{K_1}{2 r} + 2K_2$, then the last estimate takes precisely the form 
 \eqref{eq:BVbound}. This completes the proof of the proposition. 
\end{proof}

We are now in the position to prove the main result of this section:

\begin{theorem}[$L^\infty_{G}$- and $L^\infty_{BV}$-Regularity of $\bar \uu$]
\label{th:BVreg}
It holds $\bar \uu \in L^\infty_G(-r,r)$. 
If $\uu_D$ is an element of $L^\infty_{BV}(-r,r)$, then we even have $\bar \uu \in L^\infty_{BV}(-r,r)$. 
\end{theorem}

\begin{proof}
Consider some globally optimal controls
$\bar \uu_\beta$ of \eqref{eq:Prub}. From 
\cref{prop:rg1}, we know that $\bar \uu_\beta \to \bar \uu$ holds in $L^2(-r,r)$ for $\beta \to \infty$,
and from \cref{prop:rg2}, that there 
exist representatives of $\bar \uu$, $\bar \uu_\beta$, and $\uu_D$
such that the functions 
$\bar \pp_\beta := 
\nu_1  
+\nu_2 (\bar \uu_\beta - \uu_D)
+
C_{\bar \uu}   (\bar \uu_\beta - \bar \uu)  
+
\beta \eta'(\bar \uu_\beta)$
satisfy $\|\bar \pp_\beta \|_{BV[-r,r]} \leq K$ for a constant $K$ that is independent of $\beta$. 
In combination with the Helly selection theorem (see \cite[Theorem~2.7.4]{Monteiro2019})
and classical results on $L^2$-convergence, this allows us to 
find a sequence $\{\beta_n\} \subset (0, \infty)$ satisfying $\beta_n \to \infty$ for $n \to \infty$ 
 and a function $\bar \pp \in BV[-r,r]$ such that 
$\bar \uu_{\beta_n} \to \bar \uu$ holds pointwise a.e.\ in $(-r,r)$ and 
such that $\bar \pp_{\beta_n}(t) \to \bar \pp(t)$ holds for all $t \in [-r, r]$ as $n$
tends to infinity. 
Due to the definition of $\bar \pp_{\beta_n}$, 
the fact that $\eta'(s) = \eta(s) = 0$ holds for all $s \geq 0$,
and the fact that $\eta'(s) < 0$ holds for all $s \in (-\varepsilon_P, 0)$, 
the convergence properties of $\bar \uu_{\beta_n} $ and $\bar \pp_{\beta_n}$
imply that there exists a function $\zeta \in L^2(-r,r)$ such that 
$\bar \pp_{\beta_n}   -
\nu_1  
- \nu_2 (\bar \uu_{\beta_n} - \uu_D)
-
C_{\bar \uu}   (\bar \uu_{\beta_n} - \bar \uu)  
 = 
 \beta_n \eta'(\bar \uu_{\beta_n}) \to \zeta $ holds pointwise a.e.\ in $(-r,r)$ for $n \to \infty$
with $\zeta \leq 0$ a.e.\ in $(-r,r)$ and $\zeta = 0$ a.e.\ in $\{\bar \uu > 0\}$.
Passing to the limit now yields the identity 
$\bar \pp   -
\nu_1  
+ 
\nu_2 \uu_D
=
\nu_2 \bar \uu+
\zeta $
 a.e.\ in $(-r,r)$ and, 
after taking the positive part on the left and the right 
of this equality -- keeping in mind the complementarity between 
$\bar \uu $ and $\zeta$ -- that 
$\bar \uu  = \max(0, \bar \pp - \nu_1 + \nu_2 \uu_D)/\nu_2$. 
Since $\bar \pp \in BV[-r,r] \subset G[-r,r]$ holds and since the 
function $\R \ni s \mapsto \max(0,s) \in \R$ is globally Lipschitz continuous, 
the asserted regularity properties of $\bar \uu$  now follow immediately;
see \cite[Definition 4.1.1]{Monteiro2019} and \cite[Theorem 4]{Josephy1981}. 
\end{proof}

Next, we state the consequences that \cref{th:BVreg} has for our original 
problem \eqref{eq:P}. For the ease of reference, we formulate this 
result such that it is independent of \cref{ass:StandSec5}. 

\begin{theorem}[$L^\infty_{G}$- and $L^\infty_{BV}$-Regularity for \eqref{eq:P}]
\label{th:mainBV_P}
Suppose that $r \geq r_P$ is given 
and that $u_D|_{(-r,r)} \in L^\infty_G(-r,r)$
(resp., $u_D|_{(-r,r)} \in L^\infty_{BV}(-r,r)$) holds. 
Then every local solution $\bar u \in L^2(\R)$
of \eqref{eq:P} satisfies
\[
\bar u = \max\left ( 0, u_D - \frac{\nu_1}{\nu_2} \right ) \text{ a.e.\ in } \R \setminus (-r,r)
\qquad 
\text{and}
\qquad 
\bar u|_{(-r,r)} \in L^\infty_{G}(-r,r) \text{~(resp., $\bar u|_{(-r,r)} \in L^\infty_{BV}(-r,r)$).}
\]
\end{theorem}
\begin{proof}
This follows immediately from \cref{th:truncate} and \cref{th:BVreg}. 
\end{proof}

As we will make frequent use of \cref{th:BVreg,th:mainBV_P} in \cref{sec:7,sec:8}, 
we conclude this section by introducing a shorthand notation for the regularity assumption
on 
$u_D$ needed in these results:

\begin{definition}[Condition (A1)]
\label{Aone}
Given $r >0$, we say that condition \Aone{} holds if 
$u_D|_{(-r,r)} \in L^\infty_G(-r,r)$.
\end{definition}

\section{Differentiability Properties of the Control-to-State Map}
\label{sec:6}

In this section, we show that the operator $S_r\colon U_{P,r} \to H_0^1(\Omega) \cap H^2(\Omega)$,
$\uu \mapsto y$, introduced in \cref{def:ExtZero} is
Hadamard directionally differentiable at all points $\uu \in L^\infty_G(-r,r) \cap U_{P,r}$
 in all directions 
$\zz \in \TT(\uu; U_{P,r})$.
Note that, due to the identity \eqref{eq:truncid}, 
the differentiability results that we establish for $S_r$ in the following
carry over immediately to the solution map
$S\colon u \mapsto y$ of \eqref{eq:PDE}; see \cref{cor:dirdiffS} below.
We focus on $S_r$ in our analysis 
since this is the operator that appears
in the gradient projection algorithm of \cref{sec:8}. 
To clarify which notions of differentiability we are concerned with, we recall
(see \cite[Section~2.2.1]{BonnansShapiro2000}, \cite[Definition 1.1.2]{ChristofPhd2018}):

\begin{definition}[Notions of Differentiability]
\label{def:differentiability_notions}
Let $U$ and $Y$ be normed spaces.
Assume that a convex, closed, nonempty set $D \subset U$ 
and a function $H\colon  D \to Y$ are given.
Then $H$ is called:
\begin{enumerate}[label=\roman*)]
\item directionally differentiable at $u \in D$ in direction $z \in \RR(u;D)$ if there exists 
$H'(u; z) \in Y$ (the directional derivative) such that, for all
$\{s_n\} \subset (0, \infty)$  satisfying $s_n \to 0$ and $u + s_n z \in D$ for all $n$, we have
\begin{equation*}
\lim_{n \to \infty} \frac{H(u + s_n z) - H(u)}{s_n} = H'(u; z);
\end{equation*}

\item Hadamard directionally differentiable at $u \in D$ in direction $z \in \TT(u;D)$ if there exists 
$H'(u; z) \in Y$
(the Hadamard directional derivative) such that, for all
$\{s_n\} \subset (0, \infty)$ and $\{z_n\} \subset U$  
satisfying $s_n \to 0$, $z_n \to z$, and $u + s_n z_n \in D$ for all $n$, we have
\begin{equation*}
\lim_{n \to \infty} \frac{H(u + s_n z_n) - H(u)}{s_n} = H'(u; z);
\end{equation*}

\item  Hadamard-Gâteaux differentiable at $u \in D$ if $H$ is  Hadamard directionally differentiable in 
all directions  $z \in  \TT(u;D)$ at $u$ and if the derivative 
 $H'(u; \cdot)$
can be extended to a function $H'(u; \cdot) =: H'(u)\in \LL(U,Y)$;

\item\label{def:differentiability_notions:iv} Fréchet differentiable at $u \in D$ if there exists a map $H'(u) \in \LL(U,Y)$ 
satisfying\label{definition:differentiability_notions:Frechet}
\begin{equation*}
\lim_{0 < \|z\|_U \to 0, \, u + z \in D} \frac{\| H(u + z) - H(u) - H'(u)z  \|_Y}{\|z\|_U} = 0.
\end{equation*}
\end{enumerate}
If one of the above conditions holds for all $u \in D$ and all relevant $z \in U$, then we 
drop the reference to $u$ and $z$ and say that $H$ possesses
 the respective differentiability property on $D$.
If $H$ is Fréchet differentiable on~$D$ and $D \ni u \mapsto H'(u) \in \LL(U,Y)$
is a continuous function, then we call $H$ continuously Fréchet differentiable 
on~$D$.
\end{definition}

For the discussion of the differentiability properties of 
the operator $S_r$, we require some preliminary results.
 We begin by introducing
 (cf.\ \cref{def:ExtZero}\ref{def:ExtZero:ii}):

\begin{definition}[Extension by Zero in {$G[-r,r]$}]
\label{def:GextendZero}
Given  $r > 0$ and  $\uu \in G[-r,r]$, we define $\EE_r(\uu)$
to be the extension of $\uu$ by zero, i.e., 
\begin{equation}
\label{eq:EEdef}
\EE_r(\uu)\colon \R \to \R,
\qquad
\EE_r(\uu)(t)
:=
\begin{cases}
\uu(t) & \text{ if } t \in [-r,r],
\\
0 &\text{ if } t \in \R \setminus [-r,r].
\end{cases}
\end{equation}
\end{definition}

The next lemma analyzes how the extensions
of two different 
$G[-r,r]$-representatives of the same element of $L^\infty_G(-r,r)$
are related to each other.

\begin{lemma}[Extension of Regulated Representatives]
\label{lem:GReps}
Let $r > 0$ be fixed and suppose that $\uu_1, \uu_2 \in G[-r,r]$
are given such that $\uu_1(t) = \uu_2(t)$ holds for a.a.\ $t \in (-r,r)$. 
Then the following is true:
\begin{enumerate}[label=\roman*)]
\item\label{lem:GReps:i} It holds $\EE_r(\uu_1)(t-) = \EE_r(\uu_2)(t-)$ and $\EE_r(\uu_1)(t+) = \EE_r(\uu_2)(t+)$ for all $t \in \R$. 
\item\label{lem:GReps:ii} There exists an at most countable set 
$Z \subset \R$ such that $\EE_r(\uu_1)(t) = \EE_r(\uu_1)(t\pm)  = \EE_r(\uu_2)(t\pm) = \EE_r(\uu_2)(t)$
holds for all $t \in \R \setminus Z$.
\end{enumerate}
Here, with $\EE_r(\uu_i)(t\pm)$, $i=1,2$, we denote the left/right limit of the function 
$\EE_r(\uu_i)\colon \R \to \R$ at $t\in \R$; cf.\ \eqref{eq:leftrightlimits}. 
\end{lemma}\pagebreak

\begin{proof}
Suppose that functions $\uu_1, \uu_2 \in G[-r,r]$ as in the lemma are given. Then
it holds $\EE_r(\uu_1)(t) = \EE_r(\uu_2)(t)$ for a.a.\ $t \in \R$ by the definition of $\EE_r$
and we obtain from \eqref{eq:EEdef} and \cite[Definition 4.1.1]{Monteiro2019}
that the functions $\EE_r(\uu_1)$ and $\EE_r(\uu_2)$ have well-defined left and right limits 
at all points $t \in \R$. By combining these properties with the fact that 
the complement of a null set is dense,
we arrive at the identities 
$\EE_r(\uu_1)(t-) = \EE_r(\uu_2)(t-)$ and $\EE_r(\uu_1)(t+) = \EE_r(\uu_2)(t+)$
for all $t \in \R$ as desired;
cf.\ \eqref{eq:leftrightlimits}. This proves \ref{lem:GReps:i}. 
To establish \ref{lem:GReps:ii}, we note that \eqref{eq:EEdef}
and \cite[Theorem 4.1.8]{Monteiro2019} entail that there exists 
an at most countable set $Z \subset \R$ 
such that $\EE_r(\uu_1)$ and $\EE_r(\uu_2)$
satisfy
$\EE_r(\uu_1)(t-) = \EE_r(\uu_1)(t) = \EE_r(\uu_1)(t+)$ and $\EE_r(\uu_2)(t-) = \EE_r(\uu_2)(t) =  \EE_r(\uu_2)(t+)$ for all $t \in \R \setminus Z$.
In tandem with \ref{lem:GReps:i}, this yields \ref{lem:GReps:ii} and completes the proof. 
\end{proof}

Note that \cref{lem:GReps} implies that all $G[-r,r]$-representatives 
of a control $\uu \in L^\infty_{G}(-r,r)$ yield the same 
left and right limits after an 
extension by zero. 
This shows that it makes sense to talk about the functions
$\EE_r(\uu)(\,\cdot\,-)\colon \R \to \R$ and 
$\EE_r(\uu)(\,\cdot\,+)\colon \R \to \R$
of a control $\uu \in L^\infty_{G}(-r,r)$.
In what follows, we call
$\EE_r(\uu)(\,\cdot\,-)$ and $\EE_r(\uu)(\,\cdot\,+)$
the left- and the right-continuous representative
of the $ L^2(\R)$-extension $E_r(\uu)$, respectively.

\begin{definition}[Left- and Right-Continuous Representative of $E_r(\uu)$]
\label{def:EEuLG}
Given $r > 0$ and $\uu \in  L^\infty_G(-r,r)$, we define 
$\EE_r(\uu)(\,\cdot\,-)\colon \R \to \R$ and 
$\EE_r(\uu)(\,\cdot\,+)\colon \R \to \R$ to be the functions 
that are obtained by extending an arbitrary $G[-r,r]$-representative of $\uu$ by zero 
in the sense of \eqref{eq:EEdef}
and by subsequently taking left/right limits. 
\end{definition}

For the functions $\EE_r(\uu)(\,\cdot\,-)\colon \R \to \R$ and 
$\EE_r(\uu)(\,\cdot\,+)\colon \R \to \R$, we have:

\begin{lemma}[Stability of Left- and Right-Continuous Representatives]
\label{lem:RepStab}
Let $r > 0$ be given. Then it holds 
\begin{equation}
\label{eq:Linftyrepsstab}
\sup_{t \in \R}
|\EE_r(\uu_1)(t\pm) - \EE_r(\uu_2)(t\pm)|
=
\|\uu_1 - \uu_2\|_{L^\infty(-r,r)}
\quad \forall \uu_1,\uu_2 \in L^\infty_G(-r,r).
\end{equation}
\end{lemma}
\begin{proof}
The inequality ``$\geq$'' in \eqref{eq:Linftyrepsstab} follows straightforwardly 
from \eqref{eq:EEdef}, \cref{lem:GReps}, and the definition of
the norm $\|\cdot\|_{L^\infty(-r,r)}$.
To establish ``$\leq$'', one can use the exact same arguments as in 
the proof of \cref{lem:GReps}.
\end{proof}

We remark that \eqref{eq:Linftyrepsstab} does not hold
for arbitrary representatives of $E_r(\uu_1), E_r(\uu_2) \in L^2(\R)$,
even if they are regulated.
Indeed, by modifying the function $\EE_r(\uu_1)(\,\cdot\,\pm) - \EE_r(\uu_2)(\,\cdot\,\pm)$ at a single point, 
one can make the supremum-norm on the left of \eqref{eq:Linftyrepsstab} 
arbitrarily large while keeping the $L^\infty(-r,r)$-norm on the right of \eqref{eq:Linftyrepsstab}  constant. 
As a straightforward consequence of \cref{lem:RepStab}, we obtain:

\begin{corollary}[Closedness of {$L^\infty_G(-r,r)$}]
\label{cor:Gbanach}
Let $r>0$.
The set $L^\infty_G(-r,r)$ is a closed subspace of $L^\infty(-r,r)$.
In particular, $(L^\infty_G(-r,r), \|\cdot\|_{L^\infty(-r,r)})$ is a Banach space.  
\end{corollary}
\begin{proof}
That $L^\infty_G(-r,r)$ is a subspace is trivial. 
To see that $L^\infty_G(-r,r)$ is closed, suppose that $\{\uu_n\} \subset L^\infty_G(-r,r)$
is a sequence which satisfies $\uu_n \to \uu$ in $L^\infty(-r,r)$ for some $\uu \in L^\infty(-r,r)$. 
From 
\eqref{eq:EEdef} and \eqref{eq:Linftyrepsstab}, it follows that 
\begin{equation}
\label{eq:randomeq263636-37hd3u}
\sup_{t \in [-r,r]}
| \EE_r(\uu_n)(t+)  - \EE_r(\uu_m)(t+)  |
=
\|\uu_n - \uu_m\|_{L^\infty(-r,r)}
\end{equation}
holds for all $m,n\in \mathbb{N}$.
Since $\{\uu_n\}$ converges in $L^\infty(-r,r)$ and since the functions 
$\EE_r(\uu_n)(\,\cdot\,+)|_{[-r,r]}$ are regulated
by \cite[Corollary 4.1.9]{Monteiro2019},  
\eqref{eq:randomeq263636-37hd3u} shows that the sequence 
$\{\EE_r(\uu_n)(\,\cdot\,+)|_{[-r,r]}\}$ is Cauchy in $(G[-r,r], \|\cdot\|_\infty)$.
As $(G[-r,r], \|\cdot\|_\infty)$ is a Banach space, this implies that 
there exists $\hat \uu \in G[-r,r]$ such that 
$\EE_r(\uu_n)(t+) \to \hat \uu(t)$ holds for all $t \in [-r,r]$. 
From \cref{lem:GReps}, we obtain further that the 
functions $\EE_r(\uu_n)(\,\cdot\,+)|_{[-r,r]}$ are representatives of $\uu_n$
and, thus, that $\EE_r(\uu_n)(t+) \to \uu(t)$ holds for a.a.\ $t \in (-r,r)$. 
In summary, this yields $\uu = \hat \uu$ a.e.\ in $(-r,r)$. The limit $\uu$ of $\{\uu_n\}$
thus admits a $G[-r,r]$-representative, it holds $\uu \in L^\infty_G(-r,r)$, and the claim is proven. 
\end{proof}

We are now in the position 
to study the  differentiability properties of the operator $S_r$.
The point of departure for our analysis is the following observation. 

\begin{lemma}[Hadamard Directional Differentiability of $g_{E_r(\uu)}$]
\label{lem:gudirdiff}
Let $r > 0$ be fixed and let $\uu \in L^\infty_G(-r,r)$ be given. Then the function 
$g_{E_r(\uu)}\colon \R \to \R$ is Hadamard directionally differentiable on $\R$
with derivative 
\begin{equation}
\label{eq:randomeq2635456}
g_{E_r(\uu)}'(t;z) = 
\begin{cases}
\EE_r(\uu)(t-)z & \text{ if } z < 0,
\\
\EE_r(\uu)(t+)z & \text{ if } z \geq 0,
\end{cases}
\qquad \forall z \in \R\qquad \forall t \in \R.
\end{equation}
\end{lemma}

\begin{proof}
Choose an arbitrary $G[-r,r]$-representative 
of $\uu \in L^\infty_G(-r,r)$ (denoted by the same symbol)
and suppose that  $t \in \R$ and $z \geq 0$ are given. 
Then \eqref{eq:defgu},
the definition of the 
right limit $\EE_r(\uu)(t+)$ of the function $\EE_r(\uu)\colon \R \to \R$ at $t$,
and the fact that $\EE_r(\uu)$ is a representative of $E_r(\uu) \in L^2(\R)$ imply that 
\begin{equation*}
\begin{aligned}
\limsup_{0 < s \to 0}
\left | \frac{g_{E_r(\uu)}(t + sz) - g_{E_r(\uu)}(t)}{s} - \EE_r(\uu)(t+)z \right |
&
=
\limsup_{0 < s \to 0} 
\left | \frac{1}{s}\int_t^{t + sz} E_r(\uu)(\zeta) - \EE_r(\uu)(t+) \dd \zeta
\right |
\\
&
=
\limsup_{0 < s \to 0} 
\left | \frac{1}{s}\int_t^{t + sz} \EE_r(\uu)(\zeta) - \EE_r(\uu)(t+) \dd \zeta
\right |
\\
&\leq
\limsup_{0 < s \to 0} \sup_{\zeta \in (t, t + sz)} |\EE_r(\uu)(\zeta) - \EE_r(\uu)(t+)||z|
 =0. 
\end{aligned}
\end{equation*}
This shows that the function $g_{E_r(\uu)}$ is directionally differentiable 
in all nonnegative directions $z$ at all $t \in \R$ with derivative $\EE_r(\uu)(t+)z$. 
The directional differentiability in negative directions
is obtained analogously. 
To see that $g_{E_r(\uu)}$ is even
Hadamard directionally differentiable, it suffices to note that 
$g_{E_r(\uu)}$ is globally Lipschitz continuous due to the $L^\infty$-regularity of $\uu$
and to invoke \cite[Proposition 2.49]{BonnansShapiro2000}. This completes the proof. 
\end{proof}

In combination with \cref{lem:GReps}, \cref{lem:gudirdiff} yields:

\begin{corollary}[Exceptional Set of $g_{E_r(\uu)}$]
\label{eq:ExceptionalGU}
Let $r > 0$ be fixed and 
let $\uu \in L^\infty_G(-r,r)$ be given. 
Then there exists an at most countable set $Z \subset \R$
such that the function $g_{E_r(\uu)}\colon \R \to \R$ is Fréchet differentiable at all 
$t \in \R \setminus Z $. 
\end{corollary}
\begin{proof}
From \cref{lem:GReps}\ref{lem:GReps:ii} and \cref{def:EEuLG}, we obtain that 
there exists an at most countable set $Z \subset \R$
such that $\EE_r(\uu)(t-) = \EE_r(\uu)(t+)$ holds for all $t \in \R \setminus Z$. 
In view of the properties in \cref{lem:gudirdiff}, this implies 
that $g_{E_r(\uu)}$ is Hadamard-Gâteaux differentiable at all $t \in \R \setminus Z$.
As $g_{E_r(\uu)}$ is a map between finite-dimensional spaces, 
this implies the asserted Fréchet differentiability; 
see the comments after \cite[Proposition 2.49]{BonnansShapiro2000}. 
\end{proof}

Since the function $S_r$ is, by definition, the solution operator
$U_{P,r} \ni \uu \mapsto y \in H_0^1(\Omega) \cap H^2(\Omega)$ of the partial differential equation
\[
 - \Delta y + g_{E_r(\uu)}(y) = f \text{ in }\Omega,\qquad y = 0 \text{ on }\partial \Omega,
\]
we can use the properties in 
\cref{lem:gudirdiff} to prove:

\begin{theorem}[Hadamard Directional Differentiability of $S_r$]
\label{th:dirdiffSr}
Let $r > 0$ be given.
The function $S_r\colon  U_{P,r}\to H_0^1(\Omega) \cap H^2(\Omega)$
is Hadamard directionally differentiable 
at all points $\uu \in L^\infty_G(-r,r) \cap U_{P,r}$ in all directions 
$\zz \in \TT(\uu; U_{P,r}) \subset L^2(-r,r)$
in the following sense:
For all $\uu \in L^\infty_G(-r,r) \cap U_{P,r}$ and all $ \zz \in \TT(\uu; U_{P,r})$,
there exists a unique $S_r'(\uu; \zz) \in H_0^1(\Omega) \cap H^2(\Omega)$
such that, for all $\{\zz_n\} \subset L^2(-r,r)$ and $\{s_n\} \subset (0, \infty)$
satisfying $\zz_n \to \zz$ in $L^2(-r,r)$, $s_n \to 0$ in $\R$,
and $\uu + s_n \zz_n \in U_{P,r}$ for all $n \in \mathbb{N}$, it holds 
\begin{equation}
\label{eq:randomeq2636eeg367}
 \frac{S_r(\uu + s_n \zz_n) - S_r(\uu)}{s_n}
\to
S_r'(\uu; \zz) \text{ in } H^2(\Omega) \text{ for } n \to \infty.
\end{equation}
Further, the directional derivative $\delta := S_r'(\uu; \zz)   \in H_0^1(\Omega) \cap H^2(\Omega)$
of $S_r$ at a point $\uu \in L^\infty_G(-r,r) \cap U_{P,r}$ with state $y := S_r(\uu)$
in a direction $\zz \in \TT(\uu; U_{P,r})$
is uniquely characterized by the condition that it is the 
weak solution of the partial differential equation
\begin{equation}
\label{eq:dirdiffPDEr}
-\Delta \delta 
+ \EE_r(\uu)( y+)\max(0, \delta)
+ \EE_r(\uu)( y-)\min(0, \delta)
=
- 
g_{E_r(\zz)}(y) 
\text{ in }\Omega,\qquad \delta = 0 \text{ on }\partial \Omega.
\end{equation}
Here,  
$\EE_r(\uu)(\,\cdot\,-)$ 
and
$\EE_r(\uu)(\,\cdot\,+)$ 
denote the left- and the right-continuous representative of $E_r(\uu) \in L^2(\R)$.
\end{theorem}

\begin{proof}
Suppose that  
$\uu \in L^\infty_G(-r,r) \cap U_{P,r}$ and  $\zz \in \TT(\uu; U_{P,r})$
are given; define $y := S_r(\uu) \in H_0^1(\Omega) \cap H^2(\Omega)$; 
assume that $\{\zz_n\} \subset L^2(-r,r)$ and $\{s_n\} \subset (0, \infty)$
are sequences that satisfy  $\zz_n \to \zz$ in $L^2(-r,r)$, \mbox{$s_n \to 0$ in $\R$,}
and $\uu + s_n \zz_n \in U_{P,r}$ for all $n$;
and denote the difference quotients on the left-hand side of \eqref{eq:randomeq2636eeg367} 
with $\delta_n$. 
Then it follows from the definition of $S_r$ and 
\cref{th:propS}\ref{th:propS:item:iv}  that the sequence $\{\delta_n \} \subset H_0^1(\Omega) \cap H^2(\Omega)$
is bounded in $H_0^1(\Omega)$ and we obtain 
by subtracting the PDEs satisfied by $S_r(\uu + s_n \zz_n)$ and $S_r(\uu)$ 
and by rearranging the resulting equation 
that $\delta_n \in  H_0^1(\Omega) \cap H^2(\Omega)$ is the weak solution of 
\begin{equation}
\label{eq:diffquotPDE}
-\Delta \delta_n 
+ \frac{g_{E_r(\uu)}( y + s_n \delta_n) - g_{E_r(\uu)}( y)}{s_n} +  g_{E_r(\zz_n)}(y + s_n \delta_n)
=
0
\text{ in }\Omega,
\qquad 
\delta_n = 0\text{ on }\partial \Omega.
\end{equation}
Because of the boundedness of  $\{\delta_n\}$ in $H_0^1(\Omega)$
and  the theorems of Banach-Alaoglu and Rellich-Kondrachov,
we can pass over to a subsequence of $\{\delta_n\}$ (for simplicity denoted by the same symbol)
which converges weakly in $H_0^1(\Omega)$
and strongly in $L^2(\Omega)$ to some $\delta \in H_0^1(\Omega)$.
In combination with the fact that $E_r(\uu)$ is essentially bounded, 
the estimate \eqref{eq:guLqEstimate}, 
and the dominated convergence theorem, 
this yields
\begin{equation*}
\begin{aligned}
&\left \| 
\frac{g_{E_r(\uu)}( y + s_n \delta_n) - g_{E_r(\uu)}( y)}{s_n}
- g_{E_r(\uu)}'(y; \delta)
\right \|_{L^2(\Omega)}
\\
&\qquad \leq 
\left \| 
\frac{g_{E_r(\uu)}( y + s_n \delta_n) - g_{E_r(\uu)}(y + s_n \delta)}{s_n} 
\right \|_{L^2(\Omega)}
+
\left \| 
\frac{g_{E_r(\uu)}( y + s_n \delta) - g_{E_r(\uu)}( y)}{s_n} - g_{E_r(\uu)}'(y; \delta) 
\right \|_{L^2(\Omega)}
\\
&\qquad\leq
\| E_r(\uu)\|_{L^\infty(\R)} \|\delta_n - \delta\|_{L^2(\Omega)}
+
\left \| 
\frac{g_{E_r(\uu)}( y + s_n \delta) - g_{E_r(\uu)}( y)}{s_n}  - g_{E_r(\uu)}'(y; \delta)
\right \|_{L^2(\Omega)}
\to 0
\end{aligned}
\end{equation*}
and
$\left \| 
 g_{E_r(\zz_n)}(y + s_n \delta_n)  - g_{E_r(\zz)}(y)
\right \|_{L^2(\Omega)}
\leq
\|E_r(\zz_n)\|_{L^2(\R)}
s_n^{1/2}
 \left \| 
  \delta_n 
\right \|_{L^1(\Omega)}^{1/2} 
+
\left \| 
 g_{E_r(\zz_n- \zz)}(y)
\right \|_{L^2(\Omega)}
\to 0$
for $n \to \infty$. Due to \eqref{eq:diffquotPDE}, the above implies, 
on the one hand, that $\delta$ is a weak solution of
\begin{equation}
\label{eq:dirdiffPDE}
-\Delta \delta  + g_{E_r(\uu)}'(y; \delta)  +  g_{E_r(\zz)}(y) = 0
\text{ in }\Omega,
\qquad 
\delta = 0
\text{ on }\partial \Omega,
\end{equation}
and, on the other hand, 
after subtracting \eqref{eq:diffquotPDE} and \eqref{eq:dirdiffPDE}
and invoking the regularity estimate \eqref{eq:H2_reg_res},
that $\delta$ is an element of  $H_0^1(\Omega) \cap H^2(\Omega)$ which satisfies
\begin{equation*}
\begin{aligned}
\|\delta_n - \delta\|_{H^2(\Omega)}
&\leq
C 
\|-\Delta \delta_n + \Delta \delta\|_{L^2(\Omega)}
\\
&\leq
C
\left \| 
\frac{g_{E_r(\uu)}( y + s_n \delta_n) - g_{E_r(\uu)}( y)}{s_n}
- g_{E_r(\uu)}'(y; \delta)
\right \|_{L^2(\Omega)}
+
C 
\left \| 
 g_{E_r(\zz_n)}(y + s_n \delta_n)  - g_{E_r(\zz)}(y)
\right \|_{L^2(\Omega)}
\\
&\to 0
\end{aligned}
\end{equation*}
for $n \to \infty$ with a constant $C > 0$.
The sequence $\{\delta_n\}$ thus converges even strongly in 
$H^2(\Omega)$ to $\delta$. 
Next, we show that \eqref{eq:dirdiffPDE} can have at most one weak solution. 
To see this, let us suppose that 
$\delta, \hat \delta \in H_0^1(\Omega)$ are two weak solutions of \eqref{eq:dirdiffPDE}. 
Then, by using $\hat \delta - \delta$ as a test function in the weak formulation for $\hat \delta$
and $\delta - \hat \delta$ as a test function in the weak formulation for $ \delta$,
by adding the resulting identities, 
by employing \eqref{eq:randomeq2635456},
and
by using that the bound $\uu \geq - \varepsilon_P$ a.e.\ in $(-r,r)$,
\cref{def:EEuLG},
and the definition of the left and right limit imply that
$\EE_r(\uu)(t\pm) \geq - \varepsilon_P$
holds for all $t \in \R$,
we obtain that 
\begin{equation*}
\begin{aligned}
0 
&=
\int_\Omega |\nabla (\delta - \hat \delta)|^2
+
\EE_r(\uu)( y+)\big ( \max(0, \delta) - \max(0, \hat \delta) \big ) (\delta - \hat \delta)
+
\EE_r(\uu)( y-)\big ( \min(0, \delta) - \min(0, \hat \delta) \big ) (\delta - \hat \delta)
\dd x
\\
&\geq
\int_\Omega |\nabla (\delta - \hat \delta)|^2
-
\varepsilon_P \big ( \max(0, \delta) - \max(0, \hat \delta) \big ) (\delta - \hat \delta)
-
\varepsilon_P \big ( \min(0, \delta) - \min(0, \hat \delta) \big ) (\delta - \hat \delta)
\dd x
\\
&=
\int_\Omega |\nabla (\delta - \hat \delta)|^2
-
\varepsilon_P  (\delta - \hat \delta)^2
\dd x
 \geq
(c_P - \varepsilon_P) \|\delta - \hat \delta\|_{L^2(\Omega)}^2.
\end{aligned}
\end{equation*}
Since $c_P - \varepsilon_P > 0$ holds by \cref{def:Pdefs}, the above shows that $\delta$ and $\hat \delta$ 
are identical and that 
\eqref{eq:dirdiffPDE} possesses precisely one weak solution, 
namely the limit $\delta \in H_0^1(\Omega) \cap H^2(\Omega)$ of the difference quotients $\delta_n$.
Note that the unique solvability of  the 
partial differential equation \eqref{eq:dirdiffPDE} also implies that 
the limit of the subsequence chosen after \eqref{eq:diffquotPDE} 
is independent of the choice of the subsequence. 
In combination with a trivial contradiction argument (see \cite[Lemma~4.16]{Schweizer2013}),
this yields that the whole original sequence of difference quotients $\{\delta_n\}$
converges in $H_0^1(\Omega) \cap H^2(\Omega)$ to the weak 
solution $\delta $ of \eqref{eq:dirdiffPDE}.
By combining all of the above, by setting $S_r'(\uu; \zz) := \delta$, 
and by invoking \cref{lem:gudirdiff},
the claim of the theorem follows immediately. 
\end{proof}

Note that the extension operator $\EE_r$ in \eqref{eq:dirdiffPDEr}
can be dropped
if the bound $r$ and the control 
$\uu \in  L^\infty_G(-r,r) \cap U_{P,r}$  satisfy $\uu \geq 0$ a.e.\ in $(-r,r)$ and $r \geq r_P$.
Indeed, for such $\uu$ and $r$, it follows from the definition of $S_r$ and 
\cref{lem:PoissonEstimate} that
the range of the state
$y = S_r(\uu) \in H_0^1(\Omega) \cap H^2(\Omega) \subset C(\overline{\Omega})$ 
is a subset of the interval $[-r_P, r_P]$ and
we obtain from \cref{def:EEuLG}
and, in the case $r = r_P$, our conventions for the left and the right limits 
of $G[-r,r]$-functions at 
$\pm r$
that 
$\EE_r(\uu)(t\pm) = \uu(t\pm)$ holds for all $t \in \range(y)$,
where the left/right limit $\uu(t\pm)$ 
has to be understood as being defined w.r.t.\ an arbitrary
$G[-r,r]$-representative of $\uu$. The last identity allows us to recast 
the PDE \eqref{eq:dirdiffPDEr} as
\begin{equation}
\label{eq:randomPDE2525}
-\Delta \delta 
+ \uu( y+)\max(0, \delta)
+ \uu( y-)\min(0, \delta)
=
- 
\int_0^{y}
\zz(t) \dd t
\text{ in }\Omega,\qquad \delta = 0 \text{ on }\partial \Omega,
\end{equation}
without introducing any ill- or undefined terms. 
By exploiting this observation, we can also state the following tangible corollary 
for the original solution operator $S$ of \eqref{eq:PDE}: 

\begin{corollary}[Hadamard Directional Differentiability of $S$]
\label{cor:dirdiffS}
Suppose that a control $0 \leq u \in L^2(\R)$ is given such that
$u|_{(-r,r)} \in L^\infty_{G}(-r,r)$ holds for some $r \geq r_P$. 
Then the solution operator $S$ of \eqref{eq:PDE}
is Hadamard directionally differentiable at $u$ in all
directions $z \in  \TT(u; L^2_+(\R))$. 
Moreover,
the state  $y := S(u) \in C(\overline{\Omega})$ 
satisfies $\range(y) \subset [-r_P, r_P] \subset [-r,r]$ and 
the directional derivative 
$\delta := S'(u;z)   \in H_0^1(\Omega) \cap H^2(\Omega)$
of $S$ at $u$ 
in a direction $z \in \TT(u; L^2_+(\R))$
is uniquely characterized by the partial differential equation
\begin{equation}
\label{eq:randomeq367364ghh}
-\Delta \delta 
+ u( y+)\max(0, \delta)
+ u( y-)\min(0, \delta)
=
- 
g_{z}(y) 
\text{ in }\Omega,\qquad \delta = 0 \text{ on }\partial \Omega.
\end{equation}
Here,  $u(\,\cdot\,-)\colon [-r,r]\to \R$ and $u(\,\cdot\,+)\colon [-r,r] \to \R$
denote the left and right limit function
in the sense of \eqref{eq:leftrightlimits} of an arbitrary $G[-r,r]$-representative 
of $u|_{(-r,r)} \in L^\infty_{G}(-r,r)$ (acting as a Nemytskii operator).
\end{corollary}

\begin{proof}
The assertions of the corollary follow straightforwardly 
from  \cref{lemma:truncinv,th:dirdiffSr,lem:PoissonEstimate}.
\end{proof}

As a direct consequence of 
\cref{th:dirdiffSr,cor:dirdiffS,th:mainBV_P}, we obtain:

\begin{corollary}[Hadamard Directional Differentiability at Local Minimizers]
\label{cor:HadDirDifMin}
Suppose that there exists $r \geq r_P$ such that 
 $u_D \in L^2(\R)$ satisfies condition \Aone{}.
Then the solution operator $S\colon U_P \to H_0^1(\Omega) \cap H^2(\Omega)$ of \eqref{eq:PDE}
is Hadamard directionally differentiable in the sense of 
\cref{cor:dirdiffS} at 
all local minimizers $\bar u \in L^2_+(\R)$ of \eqref{eq:P}
in all directions $z \in  \TT(\bar u; L^2_+(\R))$.
Moreover, the map $S_r\colon  U_{P,r}\to H_0^1(\Omega) \cap H^2(\Omega)$
is Hadamard directionally differentiable in the sense of 
\cref{th:dirdiffSr} at all points $\bar \uu \in L^2(-r,r)$,
that are obtained by restricting a local minimizer $\bar u$ of \eqref{eq:P}
to the interval $(-r,r)$, in all directions $\zz \in \TT(\bar \uu; U_{P,r})$.
\end{corollary}

\begin{proof}
This follows immediately by combining \cref{th:dirdiffSr,cor:dirdiffS,th:mainBV_P}.
\end{proof}

As $\TT(\bar u; L^2_+(\R))$ is the tangent cone of the 
set of admissible controls of \eqref{eq:P} at $\bar u$, \cref{cor:HadDirDifMin} 
gives rise to a Bouligand stationarity condition for local solutions of \eqref{eq:P}
in a straightforward manner. Before we turn our attention to this topic, 
we study the differentiability properties of the maps $S$ and $S_r$
in a particularly nice situation, namely that  in which the 
following additional assumption is satisfied:

\begin{definition}[Condition (A2)]
\label{Atwo}
We say that condition \Atwo{} holds if the given right-hand side $f \in L^2(\Omega)$ of the 
partial differential equation \eqref{eq:PDE} 
satisfies $\lambda^d(\{f = c\}) = 0$ for all $c \in \R$. 
\end{definition}

Note that \Atwo{} is an assumption on the \emph{problem data} of \eqref{eq:P} that can be checked a-priori
and not on the \emph{problem solution}. 
This distinguishes \Atwo{} from similar conditions found in the literature on optimal control 
of nonsmooth PDEs, cf.\ \cite{Dong2022-1,Dong2022-2,Dong2022-3}.
Condition \Atwo{} is, for example, automatically satisfied if $f$ is a
 nonconstant, real-analytic function; see \cite[Proposition~1]{Mityagin2020}. 
 The main motivation for considering \Atwo{} is the following observation: 
 
 \begin{lemma}[Level Sets of States]
 \label{lem:levelstates}
 Suppose that  \Atwo{} holds. Then, for every  $y \in S(U_P)$ and every $c \in \R$, 
 the level set $\{y = c\}$ has measure zero. 
 \end{lemma}
 
 \begin{proof}
Let $y \in S(U_P)$ be given and let $u \in U_P$ be a control that satisfies  $y = S(u)$.
From the definition of $S$ and 
\cref{th:propS}, we obtain  that $y$ is an element of $H_0^1(\Omega) \cap H^2(\Omega)$ 
that satisfies 
\begin{equation}
\label{eq:randomeq26364z}
- \Delta y + g_u(y) = f \text{ a.e.\ in }\Omega,
\end{equation}
where the Laplacian is understood in the strong sense, i.e., in the sense of second weak derivatives. 
Suppose now that \Atwo{} holds and that $c \in \R$ is given. Then,
by  \cite[Lemma 4.1]{Christof2021} and due to the definition of $g_u$,
we have $\Delta y = 0$ a.e.\ on $\{y = c\}$ and $g_u(y) = g_u(c) = \mathrm{const}$ a.e.\ on $\{y = c\}$.
In combination with \eqref{eq:randomeq26364z}, these equations imply 
that $- \Delta y + g_u(y) =  g_u(c)  = f = \text{const}$ holds a.e.\ on $\{y = c\}$.
Because of \Atwo{}, the last identity can only be true if the measure of the set $\{y = c\}$ is zero. 
This proves the claim.  
 \end{proof}
 
As an immediate consequence of \cref{lem:GReps,lem:levelstates}, we obtain:

\begin{lemma}[{$G[-r,r]$}-Representatives and States]
\label{lem:GstatesA2}
Suppose that \Atwo{} holds and 
that $r > 0$, $y \in S(U_P)$,
and  $\uu_1, \uu_2 \in G[-r,r]$
satisfying $\uu_1(t) = \uu_2(t)$ for a.a.\ $t \in (-r,r)$ are given.
Then it holds
\begin{equation}
\label{eq:eveqEE}
\EE_r(\uu_1)(y) = \EE_r(\uu_1)(y+)   = \EE_r(\uu_1)(y-) 
=
\EE_r(\uu_2)(y) = \EE_r(\uu_2)(y+)   = \EE_r(\uu_2)(y-) 
\text{ a.e.\ in }\Omega.
\end{equation}
\end{lemma}
\begin{proof}
From \cref{lem:GReps}\ref{lem:GReps:ii}, we obtain that 
there exists an at most countable set 
$Z \subset \R$ satisfying
$\EE_r(\uu_1)(t) = \EE_r(\uu_1)(t\pm)  = \EE_r(\uu_2)(t\pm) = \EE_r(\uu_2)(t)$
for all $t \in \R \setminus Z$,
and from \cref{lem:levelstates}, we know that 
 $\lambda^d(\{y = c\}) = 0$
holds for all $c \in Z$. If we combine these two observations 
with the fact that
a countable union of  sets of measure zero is still of measure zero, 
then \eqref{eq:eveqEE} follows immediately.
Note that, as regulated functions are Borel measurable and 
(every representative of) $y$ is Lebesgue measurable, no measurability problems arise here.
\end{proof}

Analogously to \cref{def:EEuLG},
we obtain from \cref{lem:GstatesA2} that, under condition \Atwo{},
it makes sense to talk about the composition $\EE_r(\uu)(y) \in L^\infty(\Omega)$
of a control $\uu \in L^\infty_G(-r,r)$ and a state $y \in S(U_P)$.
\begin{definition}[Control-State-Composition]
\label{def:EEuReal}
Suppose that $r>0$ is given and that \Atwo{} holds. 
For all $\uu \in L^\infty_G(-r,r)$ and all $y \in S(U_P)$,
we define $\EE_r(\uu)(y)$ to be the element of $L^\infty(\Omega)$
that is obtained by extending an arbitrary
$G[-r,r]$-representative of $\uu$ by zero 
in the sense of \eqref{eq:EEdef} and by composing 
the resulting function with $y$. 
If the $C(\overline{\Omega})$-representative of $y$ satisfies 
$\range(y) \subset [-r,r]$, then we also write $\uu(y)$ instead of $\EE_r(\uu)(y)$.
\end{definition}

By combining \cref{th:dirdiffSr} and \cref{lem:GstatesA2}, 
keeping in mind \cref{def:EEuLG,def:EEuReal}, we arrive at: 
 
\begin{proposition}[Hadamard-Gâteaux Differentiability of $S_r$]
\label{prop:GatdiffSr}
Let $r > 0$ be given and suppose that condition \Atwo{} holds.
Then the function $S_r\colon  U_{P,r}\to H_0^1(\Omega) \cap H^2(\Omega)$
is Hadamard-Gâteaux differentiable 
at all points $\uu \in L^\infty_G(-r,r) \cap U_{P,r}$ 
in the sense that, for every $\uu \in L^\infty_G(-r,r) \cap U_{P,r}$, 
the Hadamard directional derivative $S_r'(\uu; \cdot)\colon \TT(\uu; U_{P,r}) \to H_0^1(\Omega) \cap H^2(\Omega)$
in \eqref{eq:randomeq2636eeg367} 
can be extended to a linear and continuous 
function $S_r'(\uu)\colon  L^2(-r,r) \to  H_0^1(\Omega) \cap H^2(\Omega)$. 
Moreover, for every $\uu \in L^\infty_G(-r,r) \cap U_{P,r}$ with state $y := S_r(\uu)$,
the Hadamard-Gâteaux derivative 
$S_r'(\uu) \in \LL(L^2(-r,r), H_0^1(\Omega) \cap H^2(\Omega))$ 
is precisely the solution operator $\zz \mapsto \delta$ of the linear elliptic
partial differential equation
\begin{equation}
\label{eq:GatPDE}
-\Delta \delta 
+ \EE_r(\uu)( y)\delta
=
- 
g_{E_r(\zz)}(y) 
\text{ in }\Omega,\qquad \delta = 0 \text{ on }\partial \Omega.
\end{equation}
Here, the composition $\EE_r(\uu)( y) \in L^\infty(\Omega)$ is understood in the sense of \cref{def:EEuReal}.
\end{proposition}

\begin{proof}
Consider some $\uu \in L^\infty_G(-r,r) \cap U_{P,r}$ 
with state $y := S_r(\uu)$. 
From \cref{def:EEuLG,def:EEuReal} and \cref{lem:GstatesA2},
we obtain that $\EE_r(\uu)(y) = \EE_r(\uu)(y+) = \EE_r(\uu)(y-)$ holds 
a.e.\ in $\Omega$. This identity implies that the
weak formulations of the 
PDEs \eqref{eq:dirdiffPDEr} and \eqref{eq:GatPDE} are identical and that the
directional derivatives $S_r'(\uu; \zz)$, $\zz \in \TT(\uu; U_{P,r})$,
are indeed characterized by \eqref{eq:GatPDE}. 
To complete the proof, 
it remains to check that the solution operator 
$\zz \mapsto \delta$ of the PDE \eqref{eq:GatPDE} is linear, continuous, and well-defined
as a function from $L^2(-r,r)$ to $H_0^1(\Omega) \cap H^2(\Omega)$.
Since the inequality $\uu \geq - \varepsilon_P$ a.e.\ in $(-r,r)$ implies that 
$\EE_r(\uu)(t+) \geq - \varepsilon_P$ holds for all $t \in \R$ (see the proof of \cref{th:dirdiffSr}),
due to the definition of $\EE_r(\uu)(y)$, and due to \eqref{eq:eveqEE},
we know that $\EE_r(\uu)(y) \geq - \varepsilon_P$ holds a.e.\ in $\Omega$. 
Analogously to \eqref{eq:randomeq263gg67wgwzb}, this allows us to conclude that 
the operator $H_0^1(\Omega) \ni v \mapsto -\Delta v
+  \EE_r(\uu)( y)v \in H^{-1}(\Omega)$ satisfies the coercivity estimate 
\begin{equation}
\label{eq:randomeq263636}
\left \langle 
-\Delta v
+ \EE_r(\uu)(y)v,
v
\right \rangle_{H_0^1(\Omega)}
\geq
\int_\Omega |\nabla v|^2 \dd x - \varepsilon_P \int_\Omega v^2 \dd x
\geq
\frac{c_P - \varepsilon_P}{2 \max(1, c_P)} \|v\|_{H^1(\Omega)}^2\qquad \forall v \in H_0^1(\Omega).
\end{equation}
Since the mapping $H_0^1(\Omega) \ni v \mapsto -\Delta v
+ \EE_r(\uu)( y)v \in H^{-1}(\Omega)$  is continuous and since 
\begin{equation}
\label{eq:randomeq263636-2}
\|g_{E_r(\zz)}(y) \|_{L^2(\Omega)} \leq \|y\|_{L^1(\Omega)}^{1/2} \|\zz\|_{L^2(-r,r)}
\end{equation}
holds for all $\zz \in L^2(-r,r)$ by \eqref{eq:randomeq273647},
it follows from the lemma of Lax-Milgram
(see \cite[Lemma~2.2]{Troeltzsch2010}) that  \eqref{eq:GatPDE}
possesses a linear and continuous solution operator
$L^2(-r,r) \ni \zz \mapsto \delta \in H_0^1(\Omega)$. 
Due to \eqref{eq:H2_reg_res}
and again \eqref{eq:randomeq263636-2},
we further obtain that there exists a constant $C>0$
such that the solution $\delta$ of \eqref{eq:GatPDE} satisfies 
\begin{equation}
\label{eq:randomeq273hdhe73bsu4eb}
\|\delta \|_{H^2(\Omega)} \leq 
C  \| \EE_r(\uu)( y)\delta + g_{E_r(\zz)}(y)  \|_{L^2(\Omega)} 
\leq
C \left ( \|\EE_r(\uu)( y)\|_{L^\infty(\Omega)}\|\delta \|_{L^2(\Omega)} + \|y\|_{L^1(\Omega)}^{1/2} \|\zz\|_{L^2(-r,r)}\right ).
\end{equation}
In combination with the already established continuity and linearity
of the map $ \zz \mapsto \delta $ as a function from $L^2(-r,r)$ to $H_0^1(\Omega)$,
the above shows that the solution mapping of the PDE \eqref{eq:GatPDE} is indeed 
an element of $\LL(L^2(-r,r), H_0^1(\Omega) \cap H^2(\Omega))$.
This completes the proof of the proposition. 
\end{proof}

We remark that regularization effects similar to that in 
\cref{prop:GatdiffSr} have already been exploited in \cite{Christof2021}
for the analysis of multi-objective optimal control problems 
governed by nonsmooth semilinear elliptic equations
and in \cite{Clason2021} for the optimal control of 
nonsmooth quasilinear PDEs. 
(In both of these contributions, the appearing nonsmooth superposition operator 
is fixed, though, and the control enters via the right-hand side.)

We further would like to point out that it is
crucial to work with the $G[-r,r]$-extension operator $\EE_r$ in 
\cref{prop:GatdiffSr}, and not with the $L^2(-r,r)$-extension operator $E_r$
introduced in \cref{def:ExtZero}\ref{def:ExtZero:ii}.
Indeed, only by restricting the attention to regulated representatives, 
we obtain from \cref{lem:GReps,lem:GstatesA2,lem:levelstates}
that the composition $\EE_r(\uu)(y)$ of the control and the state in \eqref{eq:GatPDE} is well defined. 
If we write, for example, $E_r(\uu)(y)$,
then it is not clear anymore whether this expression is sensible as 
changing the $L^2(\R)$-representative of $E_r(\uu)$ on an uncountable null set 
might affect the composition $E_r(\uu)(y)$ on a nonnegligible subset of $\Omega$,
even when all level sets of $y$ have measure zero.
Compare also with the comments  in \cite[Section 5]{Roesch1996}
in this context. 

By considering $S_r$ as an operator on the Banach space $(L^\infty_G(-r,r), \|\cdot\|_{L^\infty(-r,r)})$,
we can refine \cref{prop:GatdiffSr} as follows:

\begin{theorem}[{Continuous Fréchet Differentiability of $S_r$ on $L^\infty_G(-r,r) \cap U_{P,r}$}]
\label{th:Sregulated}
Let $r>0$ be given and suppose that condition \Atwo{} holds. 
Consider the map $S_r$ as a function $S_r\colon L^\infty_G(-r,r) \cap U_{P,r} \to H_0^1(\Omega) \cap H^2(\Omega)$,
i.e., as an operator defined on a subset of the space 
$(L^\infty_G(-r,r), \|\cdot\|_{L^\infty(-r,r)})$. Then
$S_r$ is continuously Fréchet differentiable on $L^\infty_G(-r,r) \cap U_{P,r}$ 
with the derivatives $S_r'(\uu)$ from \cref{prop:GatdiffSr}
in the following sense:
\begin{enumerate}[label=\roman*)]
\item\label{th:Sregulated:i} For all sequences $\{\uu_n\} \subset L^\infty_G(-r,r) \cap U_{P,r}$
satisfying $\uu_n \to \uu$ in $L^\infty(-r,r)$ for some $\uu \in L^\infty_G(-r,r) \cap U_{P,r}$, 
it holds $S_r'(\uu_n) \to S_r'(\uu)$ in $ \LL(L^2(-r,r), H_0^1(\Omega) \cap H^2(\Omega))$.

\item\label{th:Sregulated:ii} For every $\uu \in L^\infty_G(-r,r) \cap U_{P,r}$, it holds 
\begin{equation}
\label{eq:FrechetSr}
\lim_{0 < \|\zz \|_{L^\infty(-r,r)} \to 0, \, \uu + \zz \in L^\infty_G(-r,r) \cap U_{P,r}} 
\frac{\| S_r(\uu + \zz) - S_r(\uu) - S_r'(\uu)\zz \|_{H^2(\Omega)}}{\|\zz\|_{L^2(-r,r)}} = 0.
\end{equation}
\end{enumerate}
\end{theorem}

\begin{proof}
We begin with the proof of \ref{th:Sregulated:i}:
Let $\{\uu_n\} \subset L^\infty_G(-r,r) \cap U_{P,r}$
be a sequence satisfying
 $\uu_n \to \uu$ in $L^\infty(-r,r)$ for some $\uu \in L^\infty_G(-r,r) \cap U_{P,r}$
 and let $\{\zz_n\} \subset  L^2(-r,r)$ be given 
such that  $\|\zz_n\|_{L^2(-r,r)} \leq 1$ holds for all $n$.
Define $y_n := S_r(\uu_n)$ and $y := S_r(\uu)$.
From \cref{th:propS}\ref{th:propS:item:i} and the definition of $S_r$, 
we obtain that $y_n \to y$ holds in $H^2(\Omega)$,
and from \eqref{eq:eveqEE}, \eqref{eq:GatPDE}, and 
\cref{def:EEuLG,def:EEuReal}, that the functions
$\delta_n := S_r'(\uu_n)\zz_n$, $\eta_n := S_r'(\uu)\zz_n$,
and $\mu_n :=\delta_n -  \eta_n$ satisfy (in the weak sense)
\begin{equation}
\label{eq:randomeq3636}
-\Delta \mu_n
+
 \EE_r(\uu )( y)\mu_n
 =
 \left [
 \EE_r(\uu )( y+) 
-  \EE_r(\uu_n)( y_n+)\right ]\delta_n
+
g_{E_r(\zz_n)}(y) 
- 
g_{E_r(\zz_n)}(y_n) 
\text{ in }\Omega,
\qquad \mu_n= 0 \text{ on }\partial \Omega.
\end{equation}
Due to \eqref{eq:guLqEstimate},
the convergence $y_n \to y$ in $H^2(\Omega)$,
and the bound $\|\zz_n\|_{L^2(-r,r)} \leq 1$, we know that 
\[
\left \|
g_{E_r(\zz_n)}(y) 
- 
g_{E_r(\zz_n)}(y_n) 
\right \|_{L^2(\Omega)}
\leq
\|E_r(\zz_n)\|_{L^2(\R)}
\|y - y_n\|_{L^1(\Omega)}^{1/2}
\leq
\|y - y_n\|_{L^1(\Omega)}^{1/2} 
\to 0
\]
holds for $n \to \infty$. 
Using the estimates \eqref{eq:randomeq263636} and \eqref{eq:randomeq263636-2},
the PDE \eqref{eq:GatPDE}, 
and again the bound  $\|\zz_n\|_{L^2(-r,r)} \leq 1$,
it is further easy to check that $\delta_n = S_r'(\uu_n)\zz_n$ satisfies 
\[
\|\delta_n\|_{H^1(\Omega)} 
\leq 
C \| g_{E_r(\zz_n)}(y_n) \|_{L^2(\Omega)}
\leq
C 
 \|y_n\|_{L^1(\Omega)}^{1/2} \| \zz_n \|_{L^2(-r,r)}
 \leq
 C \|y_n \|_{L^1(\Omega)}^{1/2}
\]
for all $n$ with a constant $C>0$ that is independent of $n$. 
In combination with the same bootstrapping argument as in \eqref{eq:randomeq273hdhe73bsu4eb},
\cref{th:propS}\ref{th:propS:item:ii}, \cref{lem:RepStab},
and the convergence $\uu_n \to \uu$ in $L^\infty(-r,r)$,
this shows that the sequence $\{\delta_n\}$ is bounded in $H^2(\Omega)$ and that 
\begin{equation*}
\begin{aligned}
\left \|
\left [
 \EE_r(\uu )( y_n+) 
-  \EE_r(\uu_n)( y_n+)\right ]\delta_n
\right \|_{L^2(\Omega)}
&\leq
\| \EE_r(\uu )( y_n+) 
-  \EE_r(\uu_n)( y_n+)\|_{L^\infty(\Omega)}
\|\delta_n\|_{L^2(\Omega)}
\\
&
\leq
\sup_{t \in \R}
|\EE_r(\uu)(t+) - \EE_r(\uu_n)(t+)|
\|\delta_n\|_{L^2(\Omega)}
\\
&= 
\|\uu - \uu_n\|_{L^\infty(-r,r)} \|\delta_n\|_{L^2(\Omega)}
\to 0
\end{aligned}
\end{equation*}
holds for $n \to \infty$. 
Let us now consider the differences
$ [ \EE_r(\uu)( y+) -  \EE_r(\uu)( y_n+)]\delta_n$. Since the function 
$\EE_r(\uu)(\,\cdot\,+)\colon \R \to \R$
is continuous at all points $t \in \R \setminus Z$
with an at most countable exceptional set $Z \subset \R$ by
\cref{lem:GReps}
and \cref{def:EEuLG},
since $y_n \to y$ holds in $C(\overline{\Omega})$ due to 
the convergence $y_n \to y$ in $H^2(\Omega)$
and the Sobolev embeddings,
and since all level sets of $y$ have measure zero by \cref{lem:levelstates}, we have 
\[
\EE_r(\uu )( y+) -  \EE_r(\uu)( y_n+) 
=
\mathds{1}_{\{y \not \in Z\}}
\left (
\EE_r(\uu )( y+) -  \EE_r(\uu)( y_n+) 
\right )
\to 0
\quad \text{ a.e.\ in }\Omega. 
\]
In combination with the boundedness of $\{\delta_n\}$ in $H^2(\Omega) \hookrightarrow C(\overline{\Omega})$,
the fact that 
$\EE_r(\uu )( y+) -  \EE_r(\uu)( y_n+) $ is essentially bounded, 
and the dominated convergence theorem, this yields
that $ [\EE_r(\uu )( y+) -  \EE_r(\uu)( y_n+) ]\delta_n$ 
converges strongly to zero in $L^2(\Omega)$. In total,
we have now proven that 
\begin{equation*}
\begin{aligned}
&\left \| 
 \left [
 \EE_r(\uu )( y+) 
-  \EE_r(\uu_n)( y_n+)\right ]\delta_n
+
g_{E_r(\zz_n)}(y) 
- 
g_{E_r(\zz_n)}(y_n) 
\right \|_{L^2(\Omega)}
\\
&\quad\qquad \leq
\left \| 
g_{E_r(\zz_n)}(y) 
- 
g_{E_r(\zz_n)}(y_n) 
\right \|_{L^2(\Omega)}
+
\left \| 
 \left [
 \EE_r(\uu )( y_n+) 
-  \EE_r(\uu_n)( y_n+)\right ]\delta_n
\right \|_{L^2(\Omega)}
\\
&\hspace{6.66cm}\hspace{-1em}
+
\left \| 
 \left [
 \EE_r(\uu )( y+) 
-  \EE_r(\uu)( y_n+)\right ]\delta_n
\right \|_{L^2(\Omega)} 
\\
&\quad\qquad \to 0
\end{aligned}
\end{equation*}
holds for $n \to \infty$. Due to the PDE \eqref{eq:randomeq3636},
the estimate \eqref{eq:randomeq263636}, and the lemma of Lax-Milgram (see \cite[Lemma~2.2]{Troeltzsch2010}), 
this shows that $\mu_n$ converges to zero in $H^1(\Omega)$
and, again by the bootstrapping argument in \eqref{eq:randomeq273hdhe73bsu4eb},
that $\mu_n  \to 0$ holds in $H^2(\Omega)$. As $\mu_n$ 
satisfies $\mu_n = \delta_n - \eta_n = S_r'(\uu_n)\zz_n -  S_r'(\uu)\zz_n$ by definition,
this yields $\|S_r'(\uu_n)\zz_n -  S_r'(\uu)\zz_n\|_{H^2(\Omega)} \to 0$. 
As $\{\zz_n\}$ was an arbitrary sequence satisfying $\|\zz_n\|_{L^2(-r,r)} \leq 1$,
it follows that $S_r'(\uu_n) \to S_r'(\uu)$ holds in $ \LL(L^2(-r,r), H_0^1(\Omega) \cap H^2(\Omega))$. 
This completes the proof of \ref{th:Sregulated:i}.
 
 It remains to establish \ref{th:Sregulated:ii}:
 Suppose that $\uu \in L^\infty_G(-r,r)  \cap U_{P,r}$ 
 and a sequence $\{\zz_n\} \subset L^\infty_G(-r,r) $
 satisfying $0 < \|\zz_n\|_{L^\infty(-r,r)} \to 0$  and $\uu + \zz_n \in L^\infty_G(-r,r)  \cap U_{P,r}$
 for all $n$ are given. From \ref{th:Sregulated:i}, 
 \cref{prop:GatdiffSr}, and the convexity of the set $L^\infty_G(-r,r)  \cap U_{P,r}$,
 we obtain that the map 
 $[0,1] \ni s \mapsto S_r(\uu + s \zz_n) \in H_0^1(\Omega) \cap H^2(\Omega)$
 is well-defined and Hadamard-Gâteaux differentiable on $[0,1]$ with 
 a derivative that depends continuously on $s \in [0,1]$.
 By applying the mean value theorem to this function
 (see \cite[Theorem 3.2.6]{Drabek2007}), we obtain that 
 \begin{equation*}
 \begin{aligned}
&\frac{\| S_r(\uu + \zz_n) - S_r(\uu) - S_r'(\uu)\zz_n \|_{H^2(\Omega)}}{\|\zz_n\|_{L^2(-r,r)}}
\\
 &\qquad=
  \frac{1}{\|\zz_n\|_{L^2(-r,r)}}
  \left \| 
  \int_{0}^1 S_r'(\uu + s\zz_n)\zz_n \dd s   - S_r'(\uu)\zz_n 
  \right \|_{H^2(\Omega)}
  \\
  &\qquad=
    \frac{1}{\|\zz_n\|_{L^2(-r,r)}}
  \left \| 
   \int_{0}^1 [ S_r'(\uu + s\zz_n)  - S_r'(\uu) ]\zz_n \dd s   
  \right \|_{H^2(\Omega)} 
  \\
  &\qquad \leq
  \sup_{\zz \in L^\infty_G(-r,r) , \|\zz\|_{L^\infty(-r,r)} \leq \|\zz_n\|_{L^\infty(-r,r)}, 
  \uu + \zz \in L^\infty_G(-r,r) \cap U_{P,r}}
  \left \| S_r'(\uu + \zz )  - S_r'(\uu) \right \|_{\LL(L^2(-r,r), H^2(\Omega))}.
  \end{aligned}
 \end{equation*}
 Since $\|\zz_n\|_{L^\infty(-r,r)} \to 0$ holds by assumption and due to \ref{th:Sregulated:i}, 
 the right-hand side of the last estimate goes to zero for $n \to \infty$. 
 As the sequence $\{\zz_n\}$ was arbitrary, this establishes \ref{th:Sregulated:ii} and 
 completes the proof. 
\end{proof}

Note that, by combining the Fréchet differentiability property 
in point \ref{th:Sregulated:ii} of \cref{th:Sregulated}
with the elementary 
estimate $\|\zz\|_{L^2(-r,r)} \leq \sqrt{2r} \|\zz\|_{L^\infty(-r,r)}$ for all $\zz \in L^\infty(-r,r)$, we obtain
\begin{equation*}
\begin{aligned}
&\lim_{0 < \|\zz \|_{L^\infty(-r,r)} \to 0, \, \uu + \zz \in L^\infty_G(-r,r) \cap U_{P,r}} 
\frac{\| S_r(\uu + \zz) - S_r(\uu) - S_r'(\uu)\zz \|_{H^2(\Omega)}}{\|\zz\|_{L^\infty(-r,r)}} 
\\
& \qquad\leq
\sqrt{2r}
\lim_{0 < \|\zz \|_{L^\infty(-r,r)} \to 0, \, \uu + \zz \in L^\infty_G(-r,r) \cap U_{P,r}} 
\frac{\| S_r(\uu + \zz) - S_r(\uu) - S_r'(\uu)\zz \|_{H^2(\Omega)}}{\|\zz\|_{L^2(-r,r)}}
\\
&\qquad= 0. 
\end{aligned}
\end{equation*}
Completely analogously, one can also check that 
the convergence $S_r'(\uu_n) \to S_r'(\uu)$ in $ \LL(L^2(-r,r), H_0^1(\Omega) \cap H^2(\Omega))$
proven in point \ref{th:Sregulated:i} of \cref{th:Sregulated}
implies that $S_r'(\uu_n) \to S_r'(\uu)$ holds in $ \LL(L_G^\infty(-r,r), H_0^1(\Omega) \cap H^2(\Omega))$.
In summary, this shows that \cref{th:Sregulated} establishes a 
notion of continuous Fréchet differentiability
for the operator $S_r\colon L^\infty_G(-r,r) \cap U_{P,r} \to H_0^1(\Omega) \cap H^2(\Omega)$
 that is stronger than the classical 
one for a function 
defined on a subset of 
$(L^\infty_G(-r,r), \|\cdot\|_{L^\infty(-r,r)})$
with values in $(H_0^1(\Omega) \cap H^2(\Omega), \|\cdot\|_{H^2(\Omega)})$;
cf.\  \cref{def:differentiability_notions}\ref{def:differentiability_notions:iv}. 

We remark that \cref{prop:GatdiffSr} and \cref{th:Sregulated}
again give rise to corresponding results for the solution operator $S$ of \eqref{eq:PDE}, 
similarly to the relationship between 
\cref{th:dirdiffSr} and 
\cref{cor:dirdiffS}. We omit stating these reformulations here and instead 
turn our attention to the consequences that \cref{th:Sregulated} has 
for the reduced objective function $F_r\colon U_{P,r} \to \R$ of the problem \eqref{eq:Pr}:

\begin{corollary}[{Continuous Fréchet Differentiability of $F_r$ on $L^\infty_G(-r,r) \cap U_{P,r}$}]
\label{cor:Fr_regulated}
Let $r>0$ be given and suppose that \Atwo{} holds. 
Consider the map $F_r$ introduced  in \cref{def:ExtZero}\ref{def:ExtZero:v}
as a function \mbox{$F_r\colon L^\infty_G(-r,r) \cap U_{P,r} \to \R$,}
i.e., as a map defined on a subset of the Banach space 
$(L^\infty_G(-r,r), \|\cdot\|_{L^\infty(-r,r)})$ with values in $\R$. 
Then
$F_r$ is continuously Fréchet differentiable on the set $L^\infty_G(-r,r) \cap U_{P,r}$
in the sense of \cref{def:differentiability_notions}\ref{def:differentiability_notions:iv}.
Further, for all $\uu \in L^\infty_G(-r,r) \cap U_{P,r}$
with associated state $y:= S_r(\uu)$,
the evaluation $F_r'(\uu)\zz \in \R$ of the Fréchet derivative $F_r'(\uu) \in L^\infty_G(-r,r)^*$ 
of $F_r$ at $\uu$ in direction $\zz \in L^\infty_G(-r,r)$
is given by 
\begin{equation}
\label{eq:pformula-12}
\begin{aligned}
F_r'(\uu)\zz = 
\left (
- \pp_2 + \nu_1 + \nu_2(\uu - \uu_D), \zz
\right )_{L^2(-r,r)},
\end{aligned}
\end{equation}
where $\pp_2\colon [-r,r] \to \R$ is the function 
defined by the system
\begin{equation}
\label{eq:psys_1}
\begin{gathered}
  - \Delta p_1 + \EE_r(\uu)(y) p_1 =  y - y_D
  \text{ in }\Omega,\quad p_1 = 0 \text{ on }\partial \Omega,
  \qquad
  \pp_2 (s) =
  \begin{cases}
  \displaystyle
  \int_{\{y \geq s\}} p_1 \dd x  & \text{ if } s \in [0, r],\vspace{0.1cm}
  \\
   \displaystyle
    - \int_{\{y \leq s\}}  p_1 \dd x  & \text{ if }s \in [-r, 0).
  \end{cases}
\end{gathered}
\end{equation}
\end{corollary}

\begin{proof}
From \cref{th:Sregulated} and the structure of $F_r$ (or, alternatively, the chain rule),
we obtain that the function $F_r\colon L^\infty_G(-r,r) \cap U_{P,r} \to \R$
is continuously Fréchet differentiable in the sense 
of \cref{def:differentiability_notions}
(with $U :=  L^\infty_G(-r,r)$, $Y := \R$, and 
$D := L^\infty_G(-r,r) \cap U_{P,r}$)
and that the derivatives of $F_r$ satisfy 
\begin{equation}
\label{eq:shsbsbeuwbdube838dh}
F_r'(\uu)\zz = \left (S_r(\uu) - y_D, S_r'(\uu)\zz \right )_{L^2(\Omega)}
+
\int_{-r}^r \nu_1 \zz + \nu_2 (\uu - \uu_D)\zz \dd s
\quad 
\forall \uu \in L^\infty_G(-r,r) \cap U_{P,r}\quad \forall \zz \in L^\infty_G(-r,r).
\end{equation}
Consider now some $\uu \in L^\infty_G(-r,r) \cap U_{P,r}$
with associated state $y:= S_r(\uu)$.
Using the same arguments as in the proof of
\cref{prop:GatdiffSr}, one checks that the elliptic PDE in 
\eqref{eq:psys_1} possesses a unique solution 
$p_1 \in H_0^1(\Omega) \cap H^2(\Omega) \subset C(\overline{\Omega})$. 
If we choose functions $S_r'(\uu)\zz \in H_0^1(\Omega) \cap H^2(\Omega)$,
$\zz \in L^\infty_G(-r,r)$, 
as test functions in the weak form of the PDE for $p_1$ and exploit \eqref{eq:GatPDE},
then we obtain 
\begin{equation}
\label{eq:randomeq2737dh78sg7wghd}
\begin{aligned}
\int_\Omega
(y - y_D) S_r'(\uu)\zz
\dd x
&=
\int_\Omega
\nabla p_1 \cdot \nabla S_r'(\uu)\zz
+
\EE_r(\uu)(y)p_1 S_r'(\uu)\zz
\dd x
=
-
\int_\Omega
g_{E_r(\zz)}(y)
p_1
\dd x
\qquad \forall \zz \in L^\infty_G(-r,r).
\end{aligned}
\end{equation}
By plugging in the definition of $g_{E_r(\zz)}$,
by invoking the theorem of Fubini
(see \cite[Theorem 3.4.4]{Bocharev2007}),
by exploiting \cref{lem:levelstates},
and by using the definition of $\pp_2$,
we can rewrite 
 the right-hand side of \eqref{eq:randomeq2737dh78sg7wghd} as follows:
\begin{equation}
\label{eq:srsvdbe7whdu8whbdu8whsn}
\begin{aligned}
-
\int_\Omega
g_{E_r(\zz)}(y)
p_1
\dd x
&=
- \int_{\{y > 0\}}
\int_0^{ y(x)}
E_r(\zz)(s) \dd s\,p_1(x)
\dd x
+
\int_{\{y < 0\}}
\int_{ y(x)}^0 
E_r(\zz)(s) \dd s \,p_1(x)
\dd x
\\
&=
- \int_{ \{ y > 0\}}
\int_\R \mathds{1}_{[0, y(x)]}(s)
E_r(\zz)(s) \dd s\, p_1(x)
\dd x
+
\int_{\{ y < 0\}}
\int_\R \mathds{1}_{[ y(x), 0]}(s)
E_r(\zz)(s) \dd s \,p_1(x)
\dd x
\\
&=
\int_\R E_r(\zz)(s) 
\left ( 
- \int_{\{ y > 0\}}
 \mathds{1}_{[0,  y(x)]}(s)
  p_1(x)
\dd x  
+
\int_{\{ y < 0\}}
\mathds{1}_{[ y(x), 0]}(s)
p_1(x)
\dd x \right )\dd s 
\\
&=
\int_\R E_r(\zz)(s) 
\left ( 
- \mathds{1}_{(0, \infty)}(s)\int_{\{ y \geq s\}}
  p_1(x)
\dd x  
+
 \mathds{1}_{(-\infty, 0)}(s)
\int_{\{ y \leq s\}}
p_1(x)
\dd x \right )\dd s 
\\
&=
- \int_{-r}^r \zz(s)  
\pp_2(s) \dd s\qquad \forall \zz \in L^\infty_G(-r,r).
\end{aligned}
\end{equation}
In combination with \eqref{eq:shsbsbeuwbdube838dh} and \eqref{eq:randomeq2737dh78sg7wghd},
this establishes \eqref{eq:pformula-12} and completes the proof. 
\end{proof}

Analogously to the map
$S_r\colon L^\infty_G(-r,r) \cap U_{P,r} \to H_0^1(\Omega) \cap H^2(\Omega)$, 
the function $F_r\colon L^\infty_G(-r,r) \cap U_{P,r} \to \R$
is, in fact, continuously Fréchet differentiable in a sense that is 
slightly stronger than that in \cref{def:differentiability_notions}.
Indeed, from \eqref{eq:shsbsbeuwbdube838dh} and  points \ref{th:Sregulated:i}  and \ref{th:Sregulated:ii} 
of \cref{th:Sregulated}, 
it follows immediately that we not only have 
$F_r'(\uu_n) \to F_r'(\uu)$ in $L^\infty_G(-r,r)^*$
for all sequences $\{\uu_n\} \subset L^\infty_G(-r,r) \cap U_{P,r} $
which satisfy $\uu_n \to \uu$ for some $\uu \in L^\infty_G(-r,r) \cap U_{P,r}$
in $L^\infty(-r,r)$, but even 
$F_r'(\uu_n) \to F_r'(\uu)$ in $L^2(-r,r)^* \cong L^2(-r,r)$, 
and that we may even divide by the $L^2(-r,r)$-norm in the 
definition of Fréchet differentiability, and not only by the $L^\infty(-r,r)$-norm;
cf.\ \eqref{eq:FrechetSr}. We omit stating these properties in a separate result here because they are
not needed for the subsequent analysis.

Note that, because of the $L^2(-r,r)$-scalar product in \eqref{eq:pformula-12}, 
 the function $\pp_2\colon [-r,r] \to \R$ in \eqref{eq:psys_1}
 can be redefined on a set of measure zero without affecting 
the validity of the formula for the Fréchet derivative $F_r'(\uu)$. 
In particular, we could also interpret $\pp_2$ as an element of $L^\infty(-r,r)$ in 
\cref{cor:Fr_regulated} (where the essential boundedness of $\pp_2$ follows from 
the regularity properties of $p_1$).
We have introduced $\pp_2$
as a classical function and not as an element of $L^\infty(-r,r)$
here because the representative in \eqref{eq:psys_1} enjoys very convenient additional 
regularity properties. We collect these in: 

\begin{lemma}[{Regularity Properties of the Adjoint Variable $\pp_2$}]
\label{lem:ps_props}
Let $r>0$ be given and suppose that \Atwo{} holds. 
Assume that a control $\uu \in L^\infty_G(-r,r) \cap U_{P,r}$ 
with associated state $y := S_r(\uu)$ is given and that 
$p_1 \in H_0^1(\Omega) \cap H^2(\Omega)$ and 
$\pp_2\colon [-r,r] \to \R$
are defined as in \eqref{eq:psys_1}. Then the following is true:
\begin{enumerate}[label=\roman*)]
\item\label{lem:ps_props:i} $\pp_2$ is continuous at every point $s \in [-r,r]\setminus \{0\}$
and right-continuous at $s = 0$.
\item\label{lem:ps_props:ii} $\pp_2$ is an element of the space $BV[-r,r]$.
\item\label{lem:ps_props:iii} The distributional derivative of $\pp_2$ is the Borel measure on $[-r,r]$ given by 
\begin{equation}
\label{eq:pushypushy}
\pp_2' = \left (\int_\Omega p_1 \dd x\right )\delta_0 - ([y]_{-r}^r)_* (p_1 \lambda^d ).
\end{equation}
Here, 
$\delta_0$
is the Dirac measure at zero and 
$([y]_{-r}^r)_* (p_1 \lambda^d )$ denotes the
image (a.k.a.\ pushforward) of the weighted Lebesgue measure $p_1 \lambda^d$ 
in the sense of \cite[Section~3.6]{Bocharev2007}
under the continuous representative of 
\[
[y]_{-r}^r 
:= \max(0, \min(r, y)) + \min(0, \max(-r, y))
=
\begin{cases}
r &\text{ if  } y \geq r,
\\
y &\text{ if } y \in (-r,r),
\\
-r &\text{ if } y \leq -r.
\end{cases}
\]
\item\label{lem:ps_props:iv} If $r$ and $\uu$ satisfy $r \geq r_P$ and $\uu \geq 0$ a.e.\ in $(-r,r)$, 
then $\pp_2$ is identical zero in $[-r,r] \setminus (-r_P, r_P)$ 
and the formula \eqref{eq:pushypushy} holds with $[y]_{-r}^r$ replaced by $y$.
\end{enumerate}
\end{lemma}

\begin{proof}
Consider a sequence $\{s_n\} \subset [0,r]$ that satisfies $s_n \to s$ 
for some $s \in [0, r]$.
We claim that the functions $\mathds{1}_{\{y \geq s_n\}}$ converge 
pointwise a.e.\ in $\Omega$ to $\mathds{1}_{\{y \geq s\}}$ for $n \to \infty$.
To see this, let us assume that the maps $\mathds{1}_{\{y \geq s_n\}}$
and $\mathds{1}_{\{y \geq s\}}$ are defined w.r.t.\ 
some representative of $y$.
Then, for every $x \in \Omega$ with $y(x) > s$, it holds $y(x) > s_n$ for all sufficiently large $n$
and, thus, $\mathds{1}_{\{y \geq s\}}(x) = \mathds{1}_{\{y \geq s_n\}}(x) = 1$.
Analogously, we obtain that, for every $x \in \Omega$ with $y(x) < s$,
we have $y(x) < s_n$ for all large enough $n$ and, thus, 
$\mathds{1}_{\{y \geq s\}}(x) = \mathds{1}_{\{y \geq s_n\}}(x) = 0$.
Since the set $\{y = s\}$ has measure zero by \cref{lem:levelstates}, 
this shows that 
the sequence
$\mathds{1}_{\{y \geq s_n\}}$  indeed satisfies
$\mathds{1}_{\{y \geq s_n\}} \to \mathds{1}_{\{y \geq s \}}$
a.e.\ in $\Omega$ for $n \to \infty$. 
By exploiting that $p_1 \in H_0^1(\Omega) \cap H^2(\Omega)$
is  essentially bounded by the Sobolev embeddings
and by invoking the dominated 
convergence theorem, we may now deduce that 
\[
\lim_{n \to \infty}\pp_2(s_n)
=
\lim_{n \to \infty}
\int_{\Omega} \mathds{1}_{\{y \geq s_n\}} p_1  \dd x
=
\int_{\Omega} \mathds{1}_{\{y \geq s\}} p_1  \dd x
=
\pp_2(s). 
\]
This shows that $\pp_2$ is continuous on $(0, r]$ and right-continuous at $s = 0$. As the 
above arguments can also be used to establish the continuity of $\pp_2$ on $[-r,0)$,
\ref{lem:ps_props:i} now follows immediately. 

To prove \ref{lem:ps_props:ii}, we revisit \eqref{eq:srsvdbe7whdu8whbdu8whsn} 
in the special case $\zz = \vv'$, $\vv \in C^1([-r,r])$. 
This yields 
\begin{equation}
\label{eq:randomeq2737ehdu9d392po12h2poi}
\begin{aligned}
\int_{-r}^r \vv'(s)  
\pp_2(s) \dd s
&=
\int_{\{y > 0\}}
\int_0^{ y(x)}
E_r(\vv')(s) \dd s\,p_1(x)
\dd x
-
\int_{\{y < 0\}}
\int_{ y(x)}^0 
E_r(\vv')(s) \dd s \,p_1(x)
\dd x
\\
&=
\int_{\{y > 0\}}
\int_0^{\min(r, y(x))}
E_r(\vv')(s) \dd s\,p_1(x)
\dd x
-
\int_{\{y < 0\}}
\int_{\max(-r, y(x))}^0 
E_r(\vv')(s) \dd s \,p_1(x)
\dd x
\\
&=
\int_{\{y > 0\}}
\left [
\vv(\min(r, y(x)))
-
\vv(0)
\right ]
 \,p_1(x)
\dd x
-
\int_{\{y < 0\}}
\left [
\vv(0) - 
\vv(\max(-r, y(x)))
\right ] \,p_1(x)
\dd x
\\
&=
-  \vv(0) \int_\Omega p_1 \dd x
+
\int_\Omega 
\vv([y]_{-r}^r) p_1 \dd x \qquad \forall \vv \in C^1([-r,r]). 
\end{aligned}
\end{equation}
Here, in the last step, we have exploited 
that the level set $\{y = 0\}$ has measure zero by \cref{lem:levelstates}
and that $[y]_{-r}^r = \max(0, \min(r, y)) + \min(0, \max(-r, y))$ holds by definition. 
By proceeding along the exact same lines as in the proof of
\cref{prop:rg2}, i.e., by invoking the theorem of Hahn-Banach (see \cite[Theorem 8.1.1]{Monteiro2019}), 
the Riesz representation theorem 
(see \cite[Theorem 8.1.2]{Monteiro2019}),
 and the integration by parts formula for the Riemann-Stieltjes integral (see \cite[Theorem 5.5.1]{Monteiro2019}),
 we obtain from \eqref{eq:randomeq2737ehdu9d392po12h2poi} that there exists a 
 function $\qq \in BV[-r,r]$ which satisfies
 $\qq = \pp_2$ a.e.\ in $(-r,r)$. 
 Since we already know that $\pp_2$ is continuous in $[-r,r] \setminus \{0\}$,
 since elements of $BV[-r,r]$ are regulated,
 and since the complement of a set of Lebesgue measure zero is dense, 
 the last equality implies that
 $\qq(t+) = \pp_2(t+) = \pp_2(t) = \pp_2(t-) = \qq(t-) $ holds for all $t \in (-r,r) \setminus \{0\}$. 
 Recall that, as an element of $BV[-r,r]$, 
 the function $\qq$ can be written 
 as the sum of a step function $\qq_1$ and a function 
 $\qq_2 \in BV[-r,r] \cap C([-r,r])$; see \cite[Theorem 2.6.1]{Monteiro2019}. 
 Due to the identity $\qq(t+) = \qq(t-)$ for all $t \in (-r,r) \setminus \{0\}$ derived above, 
 we know that $\qq(t+) - \qq(t-) = \qq_1(t+) - \qq_1(t-) = 0$
 holds for all $t \in (-r,r) \setminus \{0\}$. This shows that there exist constants 
 $c_1, c_2 \in \R$ which satisfy $\qq_1(t+) = \qq_1(t-) = c_1$ for all $t \in (-r, 0)$
 and $\qq_1(t+) = \qq_1(t-) = c_2$ for all $t \in (0, r)$;
 see the formulas in \cite[Equations (2.5.3), (2.5.7)]{Monteiro2019}.
 We thus have 
 $\pp_2(t) = \qq_2(t) + c_1$ for all $t \in (-r,0)$
 and 
 $\pp_2(t) = \qq_2(t) + c_2$ for all $t \in (0,r)$.
 In summary, we may now conclude that 
 \[
 \pp_2  = 
 \qq_2  
 + (\pp_2(-r) - \qq_2(-r) )\mathds{1}_{\{-r\}} 
 + c_1\mathds{1}_{(-r,0)} 
 + (\pp_2(0) - \qq_2(0) )\mathds{1}_{\{0\}}  
 + c_2\mathds{1}_{(0, r)} 
 + (\pp_2(r) - \qq_2(r) )\mathds{1}_{\{r\}} 
 \]
  holds on  $[-r,r]$. This proves that $\pp_2$ is the sum of functions of bounded variation
  and that $\pp_2 \in BV[-r,r]$ holds as claimed in \ref{lem:ps_props:ii}. 
  
  To establish \ref{lem:ps_props:iii}, it suffices to note that 
  \eqref{eq:randomeq2737ehdu9d392po12h2poi} 
  and \cite[Theorem 3.6.1]{Bocharev2007} (which is applicable here due to the 
  comments on  \cite[page 191]{Bocharev2007}) imply 
  \begin{equation*}
\begin{aligned}
\int_{-r}^r \vv'(s)  
\pp_2(s) \dd s
=
-  \vv(0) \int_\Omega p_1 \dd x
+
\int_\Omega 
\vv([y]_{-r}^r) p_1 \dd x 
=
-  \vv(0) \int_\Omega p_1 \dd x
&+
\int_{-r}^r
\vv  \dd \left [ ([y]_{-r}^r)_* (p_1 \lambda^d ) \right ] 
\\
&\qquad\qquad \forall \vv \in C^1([-r,r]). 
\end{aligned}
\end{equation*}
This yields \ref{lem:ps_props:iii} as desired; see \cite[Definition 2.2.4]{Attouch2006}.

To finally prove \ref{lem:ps_props:iv}, we recall that, 
in the case $r \geq r_P$ and $\uu \geq 0$ a.e.\ in $(-r,r)$, 
we obtain from 
\cref{lem:PoissonEstimate,lem:levelstates} and the definition of $S_r$
that 
the continuous representative of $y$ 
satisfies $\range(y) \subset [-r_P, r_P]$ and  
$\lambda^d(\{y = r_P\}) = 
\lambda^d(\{y = -r_P\})
 = \lambda^d(\{y \geq r_P\}) = 
 \lambda^d(\{y \leq -r_P\})= 0$. 
In combination with \eqref{eq:psys_1} and the definition of $[y]_{-r}^r$,
these properties immediately yield
the assertions of \ref{lem:ps_props:iv}.
This completes the proof.
\end{proof}
 
As we will see in the next section, the properties of the adjoint variable $\pp_2$ 
 established in \cref{lem:ps_props} carry over to the local minimizers of 
 \eqref{eq:P} and \eqref{eq:Pr} in many situations.
Point \ref{lem:ps_props:ii} of \cref{lem:ps_props}
 further makes it possible to interpret the adjoint variable $\pp_2$ in \eqref{eq:psys_1} as an element  
 of the space $L^\infty_{BV}(-r,r) \subset L_G^\infty(-r,r)$. 
 We will exploit this observation when setting up the 
 gradient projection algorithm in 
 \cref{sec:8}. 

\section{First-Order Optimality Conditions of Bouligand and Primal-Dual Type}
\label{sec:7}

With the results of \cref{sec:4,sec:5,sec:6} at hand, 
the derivation of first-order necessary optimality conditions for
the problems \eqref{eq:P} and \eqref{eq:Pr} is straightforward. 
We begin by stating Bouligand stationarity conditions:

\begin{theorem}[Bouligand Stationarity Condition for \eqref{eq:Pr}]
\label{th:BouligandStationarityPr}
Let $r>0$ be given, suppose that $u_D$ satisfies \Aone{},
and assume that $\bar \uu$ is a local solution of \eqref{eq:Pr}
with state $\bar y := S_r(\bar \uu) \in H_0^1(\Omega) \cap H^2(\Omega)$. 
Then $\bar \uu$ is an element of  $L^\infty_G(-r,r)$,
the function $S_r\colon  U_{P,r}\to H_0^1(\Omega) \cap H^2(\Omega)$
is Hadamard directionally differentiable 
at $\bar \uu$ in all directions $\zz \in \TT(\bar \uu; U_{P,r})$
in the sense of \cref{th:dirdiffSr}, and it holds 
\begin{equation}
\label{eq:BouligandPr}
\left (\bar y - y_D, S_r'(\bar \uu;\zz) \right )_{L^2(\Omega)} 
		+ \int_{-r}^r \nu_1  \zz +  \nu_2 (\bar \uu - \uu_D) \zz\dd t 
		\geq 0
		\qquad \forall \zz \in \TT(\bar \uu; L^2_+(-r,r)).
\end{equation}
\end{theorem}
\begin{proof}
The regularity $\bar \uu \in L^\infty_G(-r,r)$ has been proven in
\cref{th:BVreg} and the 
Hadamard directional differentiability of $S_r$ at $\bar \uu$
in all directions $\zz \in \TT(\bar \uu; U_{P,r})$
follows from \cref{th:dirdiffSr}. To establish \eqref{eq:BouligandPr}, 
let us suppose that $\zz \in \TT(\bar \uu;  L^2_+(-r,r))$ is given.
From the definition of the tangent cone $\TT(\bar \uu;  L^2_+(-r,r))$,
we obtain that there exist sequences $\{\zz_n\} \subset L^2(-r,r)$ and $\{s_n\} \subset (0, \infty)$
which satisfy  $s_n \to 0$ and $\zz_n \to \zz$ in $L^2(-r,r)$ for $n \to \infty$ 
and $\bar \uu + s_n \zz_n \in L^2_+(-r,r)$
for all $n$; see \cite[Section 2.2.4]{BonnansShapiro2000}.
In combination with the local optimality of $\bar \uu$ and the chain rule 
for Hadamard directionally differentiable functions (see 
\cite[Proposition 2.47]{BonnansShapiro2000}),
this yields
\begin{equation}
\label{eq:randomeq2736sg73ihude2iup}
0 \leq \lim_{n \to \infty} \frac{F_r(\bar \uu + s_n \zz_n) - F_r(\bar \uu)}{s_n}
=
\left (\bar y - y_D, S_r'(\bar \uu;\zz) \right )_{L^2(\Omega)} 
		+ \int_{-r}^r \nu_1  \zz +  \nu_2 (\bar \uu - \uu_D) \zz\dd t.
\end{equation}
As $\zz \in \TT(\bar \uu;  L^2_+(-r,r))$ was arbitrary, 
\eqref{eq:randomeq2736sg73ihude2iup} implies \eqref{eq:BouligandPr}
and the proof is complete.
\end{proof}

\begin{theorem}[Bouligand Stationarity Condition for \eqref{eq:P}]
\label{th:BouligandStationarityP}
Suppose that $u_D \in L^2(\R)$ satisfies \Aone{} for some $r \geq r_P$
 and assume that $\bar u \in L^2(\R)$ is a local solution of \eqref{eq:P} with state $\bar y := S(\bar u)$.
 Then it holds $\bar u|_{(-r,r)} \in L^\infty_{G}(-r,r)$, 
 the solution operator $S\colon U_P \to H_0^1(\Omega) \cap H^2(\Omega)$ of \eqref{eq:PDE}
 is Hadamard directionally differentiable  at $\bar u$ 
 in all directions $z \in  \TT(\bar u; L^2_+(\R))$,
 and $\bar u$ satisfies 
\begin{equation}
\label{eq:BouligandP}
\left (\bar y - y_D, S'(\bar u; z) \right )_{L^2(\Omega)} 
		+ \int_\R \nu_1 z  +  \nu_2 (\bar u - u_D) z \dd t 
		\geq 0
		\qquad \forall z \in L^1(\R) \cap \TT(\bar u; L^2_+(\R)).
\end{equation}
\end{theorem}
\begin{proof}
The proof is along the lines of that of \cref{th:BouligandStationarityPr} and thus omitted.
\end{proof}

Note that, in the case $\nu_1 = 0$, the 
assumption $z \in L^1(\R)$ in \eqref{eq:BouligandP} can be dropped. 
An immediate consequence of \cref{th:BouligandStationarityP}
is the following result:

\begin{corollary}[Tight Controls]
\label{cor:tightytight}
In the case $u_D \equiv 0$, every local solution $\bar u \in L^2(\R)$
of \eqref{eq:P} satisfies 
$\bar u = 0$ a.e.\ in $\R \setminus \range(\bar y)$.
Here, $\range(\bar y)$ denotes the range of the 
continuous representative of the state $\bar y := S(\bar u)$. 
\end{corollary}

\begin{proof}
This follows straightforwardly from the Bouligand stationarity condition 
\eqref{eq:BouligandP} and the observation that $S'(\bar u; z) = 0$ 
holds for all $z \in \TT(\bar u; L^2_+(\R))$ that vanish a.e.\ in $\range(\bar y)$; see 
\eqref{eq:randomeq367364ghh} and \eqref{eq:defgu}.
\end{proof}

As already mentioned in \cref{subsec:1.2}, 
the notion of Bouligand stationarity is rarely helpful in practice. 
Under assumption \Atwo{}, we can establish optimality
systems that are far more suitable for applications. 

\begin{theorem}[Primal-Dual Stationarity System for \eqref{eq:Pr}]
\label{th:PDStationarityPr}
Let $r>0$ be given, suppose that $u_D$ and $f$ satisfy \Aone{} and \Atwo{},
and assume that $\bar \uu$ is a local solution of \eqref{eq:Pr}.
Then there exist adjoint variables
$\bar p_1$ and $\bar \pp_2$ such that 
$\bar \uu$ and its state $\bar y$ 
satisfy the stationarity system
\begin{equation}
\label{eq:pdsysPr}
\begin{gathered}
\bar \uu \in L^\infty_G(-r,r),
\qquad 
\bar y, \bar p_1 \in H_0^1(\Omega) \cap H^2(\Omega),
\qquad 
\bar \pp_2 \in L^\infty_{BV}(-r,r),
\\
 - \Delta \bar y + g_{E_r(\bar \uu)}(\bar y) = f \text{ in }\Omega,\quad \bar y = 0 \text{ on }\partial \Omega,
 \\
  - \Delta \bar p_1 + \EE_r(\bar \uu)(\bar y) \bar p_1 = \bar y - y_D\text{ in }\Omega,\quad \bar p_1 = 0 \text{ on }\partial \Omega,
  \\
  \bar \pp_2(s) =
  \begin{cases}
  \displaystyle
  \int_{\{\bar y \geq s\}}\bar p_1 \dd x  & \text{ if } s \geq 0,\vspace{0.1cm}
  \\
   \displaystyle
    - \int_{\{\bar y \leq s\}}\bar p_1 \dd x  & \text{ if } s < 0,
  \end{cases}
  \quad \text{ f.a.a.\ } s \in (-r,r),
\\
  \bar \uu(s) = \max\left ( 0, \uu_D(s) + \frac{\bar \pp_2(s) - \nu_1}{\nu_2} \right )
  \quad \text{ f.a.a.\ } s \in (-r,r).
\end{gathered}
\end{equation}
Here, $E_r$ and $\EE_r$ are defined as in \cref{def:ExtZero,def:GextendZero} and 
the composition $\EE_r(\bar \uu)(\bar y)$ is understood in the sense of 
\cref{def:EEuReal}. If $\uu_D = u_D|_{(-r,r)} \in L^\infty_{BV}(-r,r)$ holds,
then we additionally have
$\bar \uu \in L^\infty_{BV}(-r,r)$.
\end{theorem}

\begin{proof}
From the local optimality of $\bar \uu$ in \eqref{eq:Pr},
\cref{th:BVreg},
\cref{cor:Fr_regulated},
and the convexity of the set $L^\infty_G(-r,r) \cap L^2_+(-r,r)$, 
it follows that 
$\bar \uu$ is a solution of the variational inequality 
\begin{equation}
\label{eq:randomeq272n83}
\tilde \uu \in L^\infty_G(-r,r) \cap L^2_+(-r,r),
\qquad 
\left (
- \bar \pp_2 + \nu_1 + \nu_2(\tilde \uu - \uu_D), \uu - \tilde \uu
\right )_{L^2(-r,r)} \geq 0
\quad \forall \uu \in  L^\infty_G(-r,r) \cap L^2_+(-r,r),
\end{equation}
where $\bar \pp_2 \in L^\infty_{BV}(-r,r)$ is defined 
by the relations in \eqref{eq:pdsysPr}.
Using a simple contradiction argument and a pointwise distinction of cases,
one easily checks that  \eqref{eq:randomeq272n83} can admit at most 
one solution $\tilde \uu$ and that the function $\max(0, \uu_D + (\bar \pp_2 - \nu_1)/\nu_2)$
solves \eqref{eq:randomeq272n83}. Thus, 
$\bar \uu = \max(0, \uu_D + (\bar \pp_2 - \nu_1)/\nu_2)$ as claimed.
Note that this projection formula also implies that $\bar \uu \in L^\infty_{BV}(-r,r)$ holds
if $u_D$ satisfies $ u_D|_{(-r,r)} \in L^\infty_{BV}(-r,r)$.
By collecting all of the available information, 
\eqref{eq:pdsysPr} now follows and the proof is complete. 
\end{proof}

\begin{theorem}[Primal-Dual Stationarity System for \eqref{eq:P}]
\label{th:PDStationarityP}
Suppose that $u_D$ satisfies \Aone{} for some $r \geq r_P$
and that \Atwo{} holds. Then, for every local solution 
$\bar u \in L^2(\R)$ of the problem \eqref{eq:P}
with state $\bar y$, the range
of the continuous representative of $\bar y$ satisfies 
$\range(\bar y) \subset [-r_P, r_P] \subset [-r,r]$ and 
there exist $\bar p_1$ and $\bar p_2$
such that the following stationarity system holds true:
\begin{equation}
\label{eq:pdsysP}
\begin{gathered}
\bar u, \bar p_2 \in L^2(\R),
\qquad
\bar y, \bar p_1 \in H_0^1(\Omega) \cap H^2(\Omega),
\qquad 
\bar u|_{(-r,r)} \in L^\infty_{G}(-r,r),
\qquad
\bar p_2|_{(-r,r)} \in L^\infty_{BV}(-r,r),
\\
 - \Delta \bar y + g_{\bar u}(\bar y) = f \text{ in }\Omega,\quad \bar y = 0 \text{ on }\partial \Omega,
 \\
  - \Delta \bar p_1 + \bar u(\bar y) \bar p_1 = \bar y - y_D\text{ in }\Omega,\quad \bar p_1 = 0 \text{ on }\partial \Omega,
  \\
  \bar p_2(s) =
  \begin{cases}
  \displaystyle
  \int_{\{\bar y \geq s\}}\bar p_1 \dd x  & \text{ if } s \geq 0,\vspace{0.1cm}
  \\
   \displaystyle
    - \int_{\{\bar y \leq s\}}\bar p_1 \dd x  & \text{ if } s < 0,
  \end{cases}
  \quad \text{ f.a.a.\ } s \in \R,
\\
  \bar u(s) = \max\left ( 0, u_D(s) + \frac{\bar p_2(s) - \nu_1}{\nu_2} \right )
  \quad \text{ f.a.a.\ } s \in \R.
\end{gathered}
\end{equation}
Here, the composition $\bar u(\bar y) = \bar u|_{(-r,r)}(\bar y)$ is defined 
as in \cref{def:EEuReal}. 
If $u_D|_{(-r,r)} \in L^\infty_{BV}(-r,r)$ holds, then the inclusion
$\bar u|_{(-r,r)} \in L^\infty_{G}(-r,r)$ in \eqref{eq:pdsysP} 
can be replaced by 
$\bar u|_{(-r,r)} \in L^\infty_{BV}(-r,r)$. 
\end{theorem}

\begin{proof}
This follows straightforwardly from \cref{lem:PoissonEstimate,th:truncate,th:PDStationarityPr}.
Note that $\bar p_2$ vanishes 
by definition a.e.\ in $\R \setminus [-r_P,r_P]$. 
\end{proof}

The primal-dual stationarity systems \eqref{eq:pdsysPr} and \eqref{eq:pdsysP}
are not only appropriate points of departure for the design and analysis of 
numerical solution procedures but also provide interesting insights into the 
structural properties of locally optimal controls $\bar u$. 
From \cref{lem:ps_props}, 
we obtain, for example, the following:

\begin{corollary}[Similarities Between Optimal Superposition Operators and Activation Functions]
\label{cor:actisim}
Suppose that 
\Atwo{} is satisfied 
and that $u_D \in C(\R) \cap L^1(\R)$ holds.
Assume further that $\bar u$ is a local solution of \eqref{eq:P}.
Then the function $g_{\bar u}\colon \R \to \R$ associated with $\bar u$ possesses 
the following properties:
\begin{enumerate}[label=\roman*)]
\item\label{cor:actisim:i}
$g_{\bar u}$ is Hadamard directionally differentiable on $\R$;
\item\label{cor:actisim:ii}
$g_{\bar u}$ is continuously differentiable in $\R \setminus \{0\}$;
\item\label{cor:actisim:iii}
$g_{\bar u}$ is sigmoidal in the sense that it is nondecreasing, 
satisfies $g_{\bar u}(0) = 0$, and admits well-defined and finite limits 
$g_{\bar u}(-\infty) := \lim_{t \to - \infty}g_{\bar u}(t)$ and $g_{\bar u}(\infty) := \lim_{t \to \infty}g_{\bar u}(t)$. 
\end{enumerate}
\end{corollary}

\begin{proof}
Due to the regularity $u_D \in C(\R)$, we know that 
\Aone{} holds for all $r>0$. 
This allows us to invoke 
\cref{th:PDStationarityP} and 
points \ref{lem:ps_props:i}, \ref{lem:ps_props:ii}, and \ref{lem:ps_props:iv} 
of \cref{lem:ps_props}
to obtain 
that the function  $\bar p_2 \in L^2(\R)$ in \eqref{eq:pdsysP} possesses
a representative which
is continuous in $\R \setminus \{0\}$, identical zero in $\R \setminus [-r_P, r_P]$,
and regulated on $[-r_P, r_P]$. 
Since $u_D$ is an element of $C(\R)$ and since
$\bar u = \max(0, u_D + (\bar p_2 - \nu_1)/\nu_2)$ holds by \eqref{eq:pdsysP},
the properties of $\bar p_2$ imply that $\bar u$ possesses a representative 
that is continuous in $\R \setminus \{0\}$
and regulated in every neighborhood of the origin. 
As $g_{\bar u}$ arises from $\bar u$ by integration, 
the assertions
 \ref{cor:actisim:i} and \ref{cor:actisim:ii} of the corollary 
 now follow immediately; cf.\  \cref{lem:gudirdiff}. 
 It remains to establish \ref{cor:actisim:iii}.
 To this end, we note that
 $g_{\bar u}\colon \R \to \R$ is trivially 
 nondecreasing and zero at the origin
 by the definition of  $g_{\bar u}$
 and the inequality $\bar u \geq 0$ a.e.\ in $\R$.
 From \eqref{eq:pdsysP} 
 and  the fact that $\bar p_2$ has a compact essential support,
 we further obtain that $\bar u \in L^1(\R)$ holds. Again by the definition of
 $g_{\bar u}$
 and due to the dominated convergence theorem,
 this yields
 \begin{equation*}
 \lim_{t \to \infty}g_{\bar u}(t)
 =
 \lim_{t \to \infty} \int_{(0, \infty)} \mathds{1}_{(0,t)} \bar u \dd s
 =
 \int_{0}^\infty  \bar u \dd s
 \end{equation*}
 and 
  \begin{equation*}
 \lim_{t \to - \infty}g_{\bar u}(t)
 =
 \lim_{t \to \infty} -\int_{(-\infty, 0)} \mathds{1}_{(-t, 0)} \bar u \dd s
 =
 -\int_{-\infty}^0  \bar u  \dd s.
 \end{equation*}
 The limits $g_{\bar u}(-\infty)$ and $g_{\bar u}(\infty)$ thus exist and the 
 proof is complete.
\end{proof}

Note that, from \cref{lem:levelstates} and the stationarity system \eqref{eq:pdsysP} 
(or, alternatively, from the expression for the distributional derivative in 
\eqref{eq:pushypushy}), we obtain that
 $\bar p_2$ can only possess a continuous representative  in the situation 
of \cref{cor:actisim} if $\bar p_1$ satisfies
\[
\int_{\{\bar y \geq 0\}} \bar p_1 \dd x 
+
 \int_{\{\bar y < 0\}}  \bar p_1 \dd x
 =
 \int_{\Omega}  \bar p_1 \dd x
 =
 0.
\]
Since cases in which $\bar p_1$ has zero mean  
are exceptional, this shows that the control
$\bar u = \max(0, u_D + (\bar p_2 - \nu_1)/\nu_2)$ tends to 
possess 
a proper jump at zero
(unless, of course, the parameters 
$u_D$, $\nu_1$, and $\nu_2$ are such that the jump is removed by the projection formula).
In particular, one has to expect that 
the function $g_{\bar u}$ arising from an optimal control $\bar u$ of \eqref{eq:P}
has a proper kink at the origin.
In \cref{sec:9}, we will see that such a nondifferentiability at zero is indeed 
what is typically observed in practice
and that solving a problem of the type \eqref{eq:P}
usually gives rise to a ``ReLU-like'' sigmoidal function
$g_{\bar u}\colon \R \to \R$.
 
Regarding the variable $\nu_1$ in \eqref{eq:pdsysP},
we would like to point out that this parameter promotes sparsity properties 
of the optimal controls $\bar u$ of the problem \eqref{eq:P}.
In fact, using the bound in \cref{th:propS}\ref{th:propS:item:ii},
the form of the adjoint equation for $\bar p_1$ in \eqref{eq:pdsysP},
and the estimate \eqref{eq:randomeq263636}, 
it is easy to check that there exists a constant $C>0$,
which depends only on $\Omega$, $y_D$, $\varepsilon_P$, and the bound $M > 0$ in 
\cref{th:propS}\ref{th:propS:item:ii}, such that 
every $\bar p_2 \in L^2(\R)$ associated with a local solution $\bar u \in L^2(\R)$ of \eqref{eq:P}
satisfies $\|\bar p_2\|_{L^\infty(\R)} \leq C$. Due to the projection formula
in \eqref{eq:pdsysP}, this implies that, in the case $u_D \in L^\infty(\R)$,
there exists a value $\bar \nu_1 > 0$ such 
that \eqref{eq:P} is solved only by $\bar u = 0$ for all $\nu_1 \geq \bar \nu_1$.
In the numerical experiments of \cref{sec:9}, we will 
demonstrate that this effect is also present in practice and that the essential support of 
the optimal controls of \eqref{eq:P} shrinks as $\nu_1$ goes to infinity. 

\section{A Gradient Projection Algorithm and its Convergence Analysis}
\label{sec:8}

In this section, we demonstrate that the differentiability results 
in \cref{sec:6} not only make it possible to derive first-order 
necessary optimality conditions for problems of the type \eqref{eq:P} and \eqref{eq:Pr}
but also to set up numerical solution algorithms 
for the calculation of local minimizers.
To be able to cover the 
control-discrete and the control-continuous case simultaneously,
we henceforth consider the following variant of problem \eqref{eq:Pr}:

\begin{equation}
\label{eq:PrU}
\tag{P$_{r, \UU}$}
\left.
	\begin{aligned}
		\text{Minimize} 
		\quad & F_r(\uu) = \frac12 \|S_r(\uu) - y_D\|_{L^2(\Omega)}^2 
		+ \int_{-r}^r \nu_1  \uu + \frac{\nu_2}{2}( \uu - \uu_D)^2\dd t \\
        \text{w.r.t.}
        \quad &\uu \in \UU,\\
		\text{s.t.} \quad &  \uu \geq 0 \text{ a.e.\  in } (-r,r).
	\end{aligned}~~
\right \}
\end{equation}

We collect our assumptions on the quantities in 
\eqref{eq:PrU} in:

\begin{assumption}[Standing Assumptions for \cref{sec:8}]
\label{ass:standingPrU}
Throughout this section, we assume that:
\begin{itemize}
\item $r$ is a given positive number;
\item $\uu_D$, $E_r$, $S_r$, $U_{P,r}$, and $F_r$ are defined as in \cref{def:ExtZero};
\item $u_D$ satisfies condition \Aone{};
\item $f$ satisfies condition \Atwo{};
\item  $\UU$ is either the Banach space $(L^\infty_G(-r,r), \|\cdot\|_{L^\infty(-r,r)})$
or a finite-dimensional control space of piecewise constant 
finite element functions subordinate to a 
grid $- r = r_0 < r_1 < r_2 < ... < r_N = r$, $N \in \mathbb{N}$, i.e., 
\begin{equation}
\label{eq:UUFEcase}
\UU = 
\left \{
\uu \in L^\infty (-r,r)
\mid 
\uu = \mathrm{const} \text{ a.e.\ on } (r_{n-1}, r_n) ~\forall n=1,...,N
\right \} \subset  L^\infty_G(-r,r).
\end{equation}
\end{itemize}
\end{assumption}

In addition, $u_D$, $y_D$, $\Omega$, $d$, $\nu_1$, $\nu_2$, and $f$
are, of course, still assumed to satisfy the conditions in \cref{ass:standing}. 
Note that, by setting up a numerical algorithm for the solution of 
\eqref{eq:PrU}, we also obtain a solution method for our original identification 
problem \eqref{eq:P} due to \cref{th:BVreg} and the equivalence in \cref{th:truncate}. 
Using \cref{th:BVreg} and the direct method of the calculus of variations,
one further easily checks that \eqref{eq:PrU}  possesses at least one global solution $\bar \uu \in \UU$.
To discuss the first-order necessary optimality condition of 
\eqref{eq:PrU} and the associated gradient projection method in an adequate manner, we introduce:

\begin{definition}[$L^2$-Projection onto $\UU$]
We define $\PP_\UU^{L^2} \colon L^\infty_G(-r,r) \to \UU$ to be 
the function that maps a given $\uu \in L^\infty_G(-r,r)$
to the unique solution $\PP_\UU^{L^2} (\uu)$ of the variational problem 
\[
\PP_\UU^{L^2} (\uu) \in \UU, 
\qquad
\left (
\uu - \PP_\UU^{L^2} (\uu), \zz
\right )_{L^2(-r,r)} =0
\quad 
\forall \zz \in \UU.
\]
\end{definition}
Note that, in the case $\UU = L^\infty_G(-r,r)$, the function $\PP_\UU^{L^2} $
is just the identity map on $L^\infty_G(-r,r)$. 
(In particular, it is not a problem in this case that $\UU =  L^\infty_G(-r,r)$ is not a 
closed subspace of $(L^2(-r,r), \|\cdot\|_{L^2(-r,r)})$ as would normally be
required to define the orthogonal projection in $L^2(-r,r)$.)
If $\UU$ is a finite element space as in \eqref{eq:UUFEcase}, then $\UU$ is 
a closed subspace of $L^2(-r,r)$
and $\PP_\UU^{L^2} \colon L^\infty_G(-r,r) \to \UU$ is the restriction of the orthogonal 
projection in $L^2(-r,r)$ onto $\UU$ to the set
$L^\infty_G(-r,r) \subset L^2(-r,r)$.
As a further preparation for our analysis, we introduce 
an abbreviation for the $L^2$-gradient of the 
reduced objective function $F_r$ of \eqref{eq:PrU}:

\begin{definition}[$L^2$-Gradient of $F_r$]
\label{def:gradient}
Given a control $0 \leq \uu \in \UU$, we define
$\nabla_{L^2} F_r(\uu) \in \UU$ to be the function
\begin{equation}
\label{eq:L2gradformula}
\nabla_{L^2}  F_r(\uu) 
:= 
\PP_\UU^{L^2}
\left(
- \pp_2 + \nu_1 + \nu_2(\uu - \uu_D)
\right ).
\end{equation}
Here, $\pp_2 \in L^\infty_{BV}(-r,r)$ is characterized by the system 
\eqref{eq:psys_1} with $y = S_r(\uu)$.
\end{definition}

Recall that $\pp_2$ is indeed an element of 
$ L^\infty_{BV}(-r,r)$ 
in the situation of \cref{def:gradient}
by \cref{lem:ps_props}.
In combination with the regularity of $\uu$ and assumption \Aone{},
this implies that $- \pp_2 + \nu_1 + \nu_2(\uu - \uu_D) \in L^\infty_G(-r,r)$
holds and that the right-hand side of \eqref{eq:L2gradformula} is well defined. 
Due to the properties
of $\smash{\PP_\UU^{L^2}}$ and
the identity  \eqref{eq:pformula-12}, we further have 
\begin{equation*}
\begin{aligned}
F_r'(\uu)\zz  &= 
\left (
- \pp_2 + \nu_1 + \nu_2(\uu - \uu_D), \zz
\right )_{L^2(-r,r)}
\\
&=
\left (
\PP_\UU^{L^2}\left (
- \pp_2 + \nu_1 + \nu_2(\uu - \uu_D) \right ), \zz
\right )_{L^2(-r,r)}
=
\left (
\nabla_{L^2} F_r(\uu) , \zz
\right )_{L^2(-r,r)}
\qquad \forall  \zz \in \UU\qquad \forall\, 0 \leq \uu \in \UU,
\end{aligned}
\end{equation*}
where $F_r'(\uu) \in L^\infty_G(-r,r)^*$  denotes the Fréchet derivative 
of $F_r\colon L^\infty_G(-r,r) \cap U_{P,r} \to \R$ at $\uu$
in the sense of  \cref{cor:Fr_regulated}.  
The function $\nabla_{L^2} F_r(\uu) $ is thus indeed the $L^2$-gradient of $F_r$ at $\uu$. 
For later use, we note:

\begin{lemma}[Boundedness of Adjoint Variables]
\label{lemma:boundedadjointvars}
There exists a constant $C>0$ such that,
for all $0 \leq \uu \in \UU$ with associated adjoint functions 
$p_1 \in H_0^1(\Omega) \cap H^2(\Omega)$ 
and $\pp_2 \in L^\infty_{BV}(-r,r)$ as in \eqref{eq:psys_1}, it holds
\[
\left \|p_1 \right \|_{H^1(\Omega)}
+
\left \|\pp_2 \right \|_{L^\infty(-r,r)} \leq C.
\]
\end{lemma}

\begin{proof}
Let $0 \leq \uu \in \UU$ with associated $y, p_1 \in H_0^1(\Omega) \cap H^2(\Omega)$ 
and $\pp_2 \in L^\infty_{BV}(-r,r)$ as in \eqref{eq:psys_1} be given. 
From \cref{lem:PoissonEstimate},
we obtain that  $\|y\|_{L^\infty(\Omega)} \leq r_P$ holds.
In combination with the PDE characterizing $p_1$ in \eqref{eq:psys_1},
 the fact that $0 \leq \EE_r(\uu)(y) \in L^\infty(\Omega)$ holds, 
and standard a-priori estimates for elliptic equations, 
this implies that there exists a constant $C_1>0$ 
satisfying $\|p_1\|_{H^1(\Omega)} \leq C_1$.
Due to the definition of $\pp_2$ and the boundedness of $\Omega$,
we may now deduce that there 
exists a constant $C_2 > 0$ satisfying 
$
\| \pp_2  \|_{L^\infty(-r,r)}
\leq
C_2 \|p_1\|_{L^2(\Omega)}
\leq
C_1 C_2
$.
In summary, this yields $\smash{\left \|p_1 \right \|_{H^1(\Omega)}
+
 \|\pp_2 \|_{L^\infty(-r,r)} \leq C_1( 1+ C_2)}$
which proves the claim. 
\end{proof}

\begin{lemma}[Continuity of Gradients]
\label{lemma:contigrads}
Suppose that a sequence $\{\uu_n\} \subset \UU$ satisfying
$0 \leq \uu_n \to \uu$ in $L^\infty(-r,r)$
for some $\uu \in L^\infty(-r,r)$ is given. Then it holds $0 \leq \uu \in \UU$ and 
$\nabla_{L^2} F_r(\uu_n) \to \nabla_{L^2} F_r(\uu)$ in $L^2(-r,r)$.
\end{lemma}

\begin{proof}
Assume that a sequence $\{\uu_n\} \subset \UU$  as in the assertion of the lemma 
is given. Then it follows from \cref{cor:Gbanach} and the finite-dimensionality of 
spaces of the form \eqref{eq:UUFEcase} that $0 \leq \uu \in \UU$. 
From \cref{th:Sregulated}, \cref{cor:Fr_regulated}, and the structure of the objective $F_r$,
we further obtain that 
\begin{equation}
\label{eq:randomid2737h-2}
\left (
S_r(\uu_n) - y_D, S_r'(\uu_n) \zz
\right )_{L^2(\Omega)}
+
\int_{-r}^r \nu_1  \zz +  \nu_2 ( \uu_n - \uu_D) \zz \dd t
=
F_r'(\uu_n) \zz = 
\left (
- \pp_{2, n} + \nu_1 + \nu_2(\uu_n - \uu_D),   \zz
\right )_{L^2(-r,r)}
\end{equation}
holds for all $n$ and all $\zz \in L^\infty_G(-r,r)$,
where $S_r'(\uu_n) \in \LL(L^2(-r,r), H_0^1(\Omega) \cap H^2(\Omega))$
denotes the Gâteaux derivative of $S_r$ at $\uu_n$ in the sense of 
\cref{th:Sregulated} and $\pp_{2,n} \in L^\infty_{BV}(-r,r)$
the function that is associated with $\uu_n$ via the system \eqref{eq:psys_1}.
Note that, as  $S_r'(\uu_n)$ is a linear and continuous operator 
from $L^2(-r,r)$ to  $H_0^1(\Omega) \cap H^2(\Omega)$ for all $n$, 
the equality between the left- and the right-hand side of \eqref{eq:randomid2737h-2} 
remains true for all test functions $\zz \in L^2(-r,r)$. 
Since $S_r'(\uu_n) \to S_r'(\uu)$ holds in $ \LL(L^2(-r,r), H_0^1(\Omega) \cap H^2(\Omega))$
by \cref{th:Sregulated}\ref{th:Sregulated:i}
and since $S_r(\uu_n) \to S_r(\uu)$ holds in $H_0^1(\Omega) \cap H^2(\Omega)$
by \cref{th:propS}\ref{th:propS:item:i}, this allows us to deduce that 
\begin{equation}
\label{eq:randomeq73747}
 \pp_{2,n}
=
- S_r'(\uu_n)^*\left (
S_r(\uu_n) - y_D
\right )
\to
- S_r'(\uu )^*\left (
S_r(\uu) - y_D \right )
\end{equation}
holds strongly in $L^2(-r,r)$, where a star denotes an adjoint operator. 
Using the exact same reasoning as in \eqref{eq:randomid2737h-2},
one obtains that the right-hand side of \eqref{eq:randomeq73747} 
is identical to the function $\pp_2 \in L^\infty_{BV}(-r,r)$ associated with $\uu$
via the system \eqref{eq:psys_1}. 
We thus have $\pp_{2,n} \to \pp_2$ in $L^2(-r,r)$. 
The assertion of the lemma now follows 
immediately from the definition of the $L^2$-gradient of $F_r$
and the properties of $\PP_\UU^{L^2}$. This completes the proof. 
\end{proof}

Using the $L^2$-gradient, we can formulate the first-order necessary optimality
condition of \eqref{eq:PrU} as follows:

\begin{theorem}[First-Order Necessary Optimality Condition for \eqref{eq:PrU}]
Suppose that $0 \leq \bar \uu \in \UU$ is a local solution of \eqref{eq:PrU}.
Then it holds
\begin{equation}
\label{eq:1stOrdPrU}
\nabla_{L^2} F_r(\bar \uu)
\begin{cases}
\geq 0 & \text{ a.e.\ in } \{\bar \uu = 0\},
\\
= 0 & \text{ a.e.\ in } \{\bar \uu > 0\}.
\end{cases}
\end{equation}
\end{theorem}

\begin{proof}
From \cref{cor:Fr_regulated} and \cref{def:gradient}, 
we obtain (analogously to the proof of \cref{th:PDStationarityPr})
that $\bar \uu \in \UU$ satisfies 
\begin{equation*}
\bar \uu \in \UU \cap L^2_+(-r,r),
\qquad 
\left (
\nabla_{L^2} F_r(\bar \uu), \uu - \bar \uu
\right )_{L^2(-r,r)} \geq 0
\quad \forall \uu \in  \UU \cap L^2_+(-r,r).
\end{equation*}
Again analogously to the proof of  \cref{th:PDStationarityPr},
one checks that, 
for both the case $\UU = L^\infty_G(-r,r)$ and the case 
where $\UU$ is of the form  \eqref{eq:UUFEcase},
the above variational inequality implies
$
 \bar \uu = 
 \max\left (
 0, \bar \uu - \nabla_{L^2} F_r(\bar \uu)
 \right )
$
a.e.\ in $(-r,r)$. Using a distinction of cases, 
\eqref{eq:1stOrdPrU} now follows immediately. 
\end{proof}

Motivated by the necessary optimality condition \eqref{eq:1stOrdPrU}, 
we introduce the following measures for the degree of stationarity
of an admissible control:

\begin{definition}[Stationarity Indicators]
\label{def:statindicators}
Given a control $0 \leq \uu \in \UU$  
and a number  $\epsilon > 0$, 
we define 
\[
\Theta_0(\uu)
:=
\left \|
\nabla_{L^2} F_r(\uu)\mathds{1}_{\{\uu > 0\}}
+
\min\left (0, \nabla_{L^2} F_r(\uu)\right ) \mathds{1}_{\{\uu = 0\}}
\right \|_{L^2(-r,r)}
\]
and
\[
\Theta_{\epsilon}(\uu)
:=
\left (
\int_{-r}^r 
\min\left (
\frac{\uu}{\epsilon}, \nabla_{L^2} F_r(\uu)
\right )
\nabla_{L^2} F_r(\uu)
\dd t
\right 
)^{1/2}.
\]
Here, the min-function appearing in the integral/the norm acts by superposition. 
\end{definition}

Using the definitions of $\Theta_0$ and $\Theta_\epsilon$, one easily 
establishes:

\begin{lemma}[Properties of the Stationarity Measures]
\label{lemma:propertiesTheta}
Let $0 \leq \uu \in \UU$ be given. Then the following is true:
\begin{enumerate}[label=\roman*)]
\item\label{lemma:propertiesTheta:i} $\uu$ satisfies the first-order condition \eqref{eq:1stOrdPrU} if and only if $\Theta_0(\uu) = 0$;
\item\label{lemma:propertiesTheta:ii} $\uu$ satisfies the first-order condition \eqref{eq:1stOrdPrU} if and only if 
$\Theta_\epsilon(\uu) = 0$ for one (and thus all) $\epsilon > 0$;
\item\label{lemma:propertiesTheta:iii} It holds 
$0 \leq \Theta_{\epsilon_1}(\uu) \leq  \Theta_{\epsilon_2}(\uu)\leq \Theta_0(\uu) $ for all $\epsilon_1 \geq \epsilon_2 > 0$;
\item\label{lemma:propertiesTheta:iv} It holds $\Theta_\epsilon(\uu)  \to \Theta_0(\uu)$ for $\epsilon \to 0$. 
\end{enumerate}
\end{lemma}

\begin{proof}
Point \ref{lemma:propertiesTheta:i} is trivial. 
To prove the remaining assertions, we note that the nonnegativity of $\uu$ implies that,
for every $\epsilon > 0$, we have 
\begin{equation*}
\begin{aligned} 
&\min\left (
\frac{\uu}{\epsilon}, \nabla_{L^2} F_r(\uu)
\right )
\nabla_{L^2} F_r(\uu)
\\
&\qquad=
\min\left (
\frac{\uu}{\epsilon}, \nabla_{L^2} F_r(\uu)
\right )
\min(0, \nabla_{L^2} F_r(\uu))
+
\min\left (
\frac{\uu}{\epsilon}, \nabla_{L^2} F_r(\uu)
\right ) \max(0,  \nabla_{L^2} F_r(\uu)) 
\\
&\qquad= 
\min(0, \nabla_{L^2} F_r(\uu))^2 
+
\min\left (
\frac{\uu}{\epsilon}, \max(0,  \nabla_{L^2} F_r(\uu))
\right ) \max(0,  \nabla_{L^2} F_r(\uu))
 \quad \text{a.e.\ in }(-r,r)
\end{aligned}
\end{equation*}
and
\begin{equation*}
\begin{aligned}
0 &\leq \min(0, \nabla_{L^2} F_r(\uu))^2 
+
\min\left (
\frac{\uu}{\epsilon}, \max(0,  \nabla_{L^2} F_r(\uu))
\right ) \max(0,  \nabla_{L^2} F_r(\uu)) 
\\
&\leq
\min(0, \nabla_{L^2} F_r(\uu))^2 
+
\mathds{1}_{\{ \uu > 0\}}\max(0,  \nabla_{L^2} F_r(\uu))^2 \quad \text{a.e.\ in }(-r,r).
\end{aligned}
\end{equation*}
Using the above identities and inequalities, integration,  the dominated 
convergence theorem, and the definitions 
of $\Theta_0$ and $\Theta_\epsilon$, the assertions in 
  \ref{lemma:propertiesTheta:ii}, \ref{lemma:propertiesTheta:iii}, and \ref{lemma:propertiesTheta:iv}
  follow immediately. This completes the proof. 
\end{proof}

The stationarity measures $\Theta_\epsilon$, $\epsilon > 0$, are
advantageous from both the analytical and the numerical point of view because 
they are more stable w.r.t.\ small perturbations of $\uu$ than $\Theta_0$. 
Compare also with the 
convergence analysis of the standard gradient projection method 
in finite dimensions in this context; see \cite{Beck2014,Christof2020-TV}. 
We are now in the position to formulate the algorithm that we use for the 
numerical
solution of the problem \eqref{eq:PrU}.

\begin{Algorithm}[Gradient Projection Method for the Solution of \eqref{eq:PrU}]\label{alg:gradproj}
~\hspace{-10cm}
\begin{algorithmic}[1]
\STATE{Choose an initial guess $0 \leq \uu_0 \in \UU$ 
and parameters $\epsilon_1, \epsilon_2> 0$, $\sigma  > 0$, and $\omega, \theta \in (0,1)$.}
\FOR{$i=0,1,2,3,...$}
\STATE{\label{alg:line:3}Calculate $\nabla_{L^2}F_r(\uu_i) \in \UU$ 
by means of the system \eqref{eq:psys_1} and formula \eqref{eq:L2gradformula}.}
\IF{$\Theta_{\epsilon_1}(\uu_i)  \leq \epsilon_2$\label{alg:line:4}}
\STATE{{\bf \emph{break} }}
\ENDIF
\STATE{Initialize $\sigma_0 := \sigma$ and calculate a step size as follows:}\label{alg:line:7}
\FOR{$j=0,1,2,3,...$}\label{alg:line:8}
\STATE{\label{alg:line:9}Calculate the quantity
$
e_j := F_r(\uu_i) - F_r(\max(0, \uu_i - \sigma_j  \nabla_{L^2}F_r(\uu_i)))
 - 
 \sigma_{j}
\theta
\Theta_{\epsilon_1}(\uu_i)^2.
$}
\IF{$e_j < 0$\label{alg:line:10}}
\STATE{Define $\sigma_{j+1} := \omega \sigma_j$.}
\ELSE
\STATE{Define $\tau_i := \sigma_j$ and {\bf \emph{break}}}
\ENDIF
\ENDFOR\label{alg:line:15}
\STATE{\label{alg:line:16} Define $\uu_{i+1} := \max(0, \uu_i - \tau_i \nabla_{L^2}F_r(\uu_i))$.}
\ENDFOR
\end{algorithmic}
\end{Algorithm}

We emphasize that the $\max$-functions  
in steps \ref{alg:line:9} and \ref{alg:line:16} of \cref{alg:gradproj}
again act by superposition. The next two lemmas show that 
\cref{alg:gradproj} is sensible. 

\begin{lemma}[Finite Termination of the Line-Search]
\label{lem:finitelinesearch}
Suppose that 
numbers $\epsilon_1, \epsilon_2> 0$, $\sigma  > 0$, and $\omega, \theta \in (0,1)$
and a function $0 \leq \uu_i \in \UU$ satisfying $\Theta_{\epsilon_1}(\uu_i)  > \epsilon_2$
are given. Then the line-search procedure in lines \ref{alg:line:7} to \ref{alg:line:15}
of \cref{alg:gradproj} 
terminates after finitely many steps  with a step size $\tau_i > 0$. 
\end{lemma}

\begin{proof}
Suppose that a function $0 \leq \uu_i \in \UU$ satisfying $\Theta_{\epsilon_1}(\uu_i)  > \epsilon_2$ 
is given and 
that the line-search does not terminate after finitely many steps. 
Then it holds  $e_j < 0$ for all $j \in \N_0$ and, as a consequence,
\[
\frac{F_r(\max(0, \uu_i - \sigma_j  \nabla_{L^2}F_r(\uu_i))) - F_r(\uu_i) }{\sigma_j}
>
 - 
\theta
\Theta_{\epsilon_1}(\uu_i)^2
\geq 
- \theta \Theta_{0}(\uu_i)^2 \qquad \forall j \in \N_0
\]
with $0 < \sigma_j \to 0$ for $j \to \infty$.
Here, the inequality involving $\Theta_0(\uu_i)$ follows from \cref{lemma:propertiesTheta}\ref{lemma:propertiesTheta:iii}.
  Let us define 
\begin{equation}
\label{eq:zfdef}
\zz_j := 
\frac{\max(0, \uu_i - \sigma_j  \nabla_{L^2}F_r(\uu_i)) - \uu_i}{\sigma_j}
=
\frac{\max(0, \uu_i - \sigma_j  \nabla_{L^2}F_r(\uu_i)) - \max(0,\uu_i)}{\sigma_j},
\qquad j \in \mathbb{N}_0.
\end{equation}
Using Lebesgue's dominated convergence theorem and the directional differentiability
and global Lipschitz 
continuity of the function  $\R \ni s \mapsto \max(0, s) \in \R$,
it is easy to check that $\zz_j \to \zz$ holds in $L^2(-r,r)$ with 
\begin{equation}
\label{eq:zzdef33}
\zz := - \nabla_{L^2}F_r(\uu_i) \mathds{1}_{\{\uu_i > 0\}} 
+ \max\left (0, - \nabla_{L^2}F_r(\uu_i) \right ) \mathds{1}_{\{\uu_i = 0\}}.
\end{equation}
In view of \cref{th:dirdiffSr} and \cref{prop:GatdiffSr}, this implies that 
\[
 \frac{S_r(\uu_i + \sigma_j \zz_j) - S_r(\uu_i)}{\sigma_j}
\to
S_r'(\uu_i)\zz
\]
holds in $H^2(\Omega)$ for $j \to \infty$,
where $S_r'(\uu_i) \in \LL(L^2(-r,r), H_0^1(\Omega) \cap H^2(\Omega))$
denotes the Hadamard-Gâteaux derivative of $S_r$ at $\uu_i$
in the sense of \cref{prop:GatdiffSr}. 
Due to the structure of $F_r$, it follows that
\begin{equation*}
\begin{aligned}
- \theta \Theta_{0}(\uu_i)^2
&\leq
\frac{F_r(\max(0, \uu_i - \sigma_j  \nabla_{L^2}F_r(\uu_i))) - F_r(\uu_i) }{\sigma_j}
\\
&=
\frac{F_r(\uu_i + \sigma_j \zz_j) - F_r(\uu_i) }{\sigma_j}
\\
&\to \left (
S_r(\uu_i) - y_D, S_r'(\uu_i)\zz
\right )_{L^2(\Omega)}
+
\int_{-r}^r \nu_1  \zz +  \nu_2 ( \uu_i - \uu_D)\zz \dd t
\qquad \text{for }j \to \infty.
\end{aligned}
\end{equation*}
Note that, for all $\tilde \zz \in L^\infty_G(-r,r)$, 
the Fréchet differentiability properties 
in \cref{th:Sregulated} and \cref{cor:Fr_regulated} imply that 
\begin{equation}
\label{eq:randomid2737h}
\left (
S_r(\uu_i) - y_D, S_r'(\uu_i)\tilde \zz
\right )_{L^2(\Omega)}
+
\int_{-r}^r \nu_1  \tilde \zz +  \nu_2 ( \uu_i - \uu_D)\tilde \zz \dd t
=
F_r'(\uu_i)\tilde \zz = 
\left (
- \pp_{2,i} + \nu_1 + \nu_2(\uu_i - \uu_D), \tilde \zz
\right )_{L^2(-r,r)}
\end{equation}
with $\pp_{2,i} \in L^\infty_{BV}(-r,r)$ defined via the system \eqref{eq:psys_1} for $\uu_i$.
As $S_r'(\uu_i) \in \LL(L^2(-r,r), H_0^1(\Omega) \cap H^2(\Omega))$ holds,
by approximation, we obtain that the equality between the left- and the right-hand side of \eqref{eq:randomid2737h} remains true 
for all $\tilde \zz \in L^2(-r,r)$. In particular, we have 
\[
- \theta \Theta_{0}(\uu_i)^2
\leq
\left (
- \pp_{2,i} + \nu_1 + \nu_2(\uu_i - \uu_D),  \zz
\right )_{L^2(-r,r)}.
\]
In the case $\UU = L^\infty_G(-r,r)$, the above inequality, 
the 
definitions of the functions $\zz$ and 
 $\nabla_{L^2}  F_r(\uu_i)$,
and
the identity
$- \pp_{2,i} + \nu_1 + \nu_2(\uu_i - \uu_D) = \PP_\UU^{L^2}
(
- \pp_{2,i} + \nu_1 + \nu_2(\uu_i - \uu_D)
 )$ 
  imply that 
\begin{equation}
\label{eq:randomeq273764zdh47}
- \theta \Theta_{0}(\uu_i)^2
\leq 
\left (
\nabla_{L^2}  F_r(\uu_i),  \zz
\right )_{L^2(-r,r)}
=
-
\|\zz\|_{L^2(-r,r)}^2
=
 - \Theta_{0}(\uu_i)^2.
\end{equation}
In the case where $\UU$ is a space of piecewise constant finite element functions
as in \eqref{eq:UUFEcase}, 
 it follows from \eqref{eq:zzdef33} that $\zz \in \UU$ holds,
and we may exploit the properties of $\PP_\UU^{L^2}$ 
and argue analogously to \eqref{eq:randomeq273764zdh47} to obtain that 
\begin{equation*}
- \theta \Theta_{0}(\uu_i)^2
 \leq 
\left (
- \pp_{2,i} + \nu_1 + \nu_2(\uu_i - \uu_D),  \zz
\right )_{L^2(-r,r)}
 =
\left (
\nabla_{L^2}  F_r(\uu_i),  \zz
\right )_{L^2(-r,r)}
=
-
\|\zz\|_{L^2(-r,r)}^2
=
- \Theta_{0}(\uu_i)^2.
\end{equation*}
In both of the above cases, one arrives at the inequality
$-\theta \Theta_{0}(\uu_i)^2 \leq - \Theta_{0}(\uu_i)^2$
which is impossible due to 
the assumptions $\Theta_{\epsilon_1}(\uu_i) > \epsilon_2 > 0$ 
and $\theta \in (0,1)$ and the inequality $ \Theta_{\epsilon_1}(\uu_i) \leq \Theta_{0}(\uu_i)$.
The line-search algorithm thus has to terminate after finitely many steps as claimed.
\end{proof}

Note that it is necessary to take the detour via 
\cref{th:dirdiffSr} and \cref{prop:GatdiffSr} in the proof of \cref{lem:finitelinesearch}
as the sequence $\{\zz_j\}$ in \eqref{eq:zfdef} only converges 
in $L^q(-r,r)$ for all $q \in [1, \infty)$ but \emph{not} in $L^\infty(-r,r)$ in general 
(as one may easily check by means of examples). This makes it impossible 
to invoke the Fréchet differentiability result in, e.g., \cref{cor:Fr_regulated}
directly. Next, we clarify that the update of $\uu_i$ in 
\cref{alg:gradproj} is sensible.

\begin{lemma}[Well-Definedness of Iterates]
Suppose that a function $0 \leq \uu_i \in \UU$
with associated gradient $\nabla_{L^2}  F_r(\uu_i) \in \UU$
and
a number $\tau_i > 0$ are given. Then
$\uu_{i+1} := \max(0, \uu_i - \tau_i \nabla_{L^2}F_r(\uu_i))$
satisfies $0 \leq \uu_{i+1} \in \UU$.
\end{lemma}
\begin{proof}
From \eqref{eq:L2gradformula} and the comments after 
\cref{def:gradient}, we obtain that $\nabla_{L^2}  F_r(\uu_i) \in \UU$.
This also implies that $\uu_i - \tau_i  \nabla_{L^2}  F_r(\uu_i) \in \UU$. 
In the case $\UU = L^\infty_G(-r,r)$, the
inclusion 
$\uu_{i+1} = \max \left (0, \uu_i - \tau_i  \nabla_{L^2}  F_r(\uu_i)\right ) \in \UU$
 now follows from the fact that the positive part of a regulated function is again regulated. 
In the case where $\UU$ is a
finite element space as in \eqref{eq:UUFEcase}, 
it is trivially true that the Nemytskii operator induced by 
$\R \ni s \mapsto \max(0, s) \in \R$ maps $\UU$ into itself,
so $\uu_{i+1} \in \UU$ again follows. Since we trivially have $\uu_{i+1} \geq 0$, this completes the proof. 
\end{proof}

We emphasize that the above lemma is crucial 
as it shows that updates with the $L^2(-r,r)$-gradient 
do not cause the sequence of iterates $\{\uu_i\}$ to leave the 
set $L^\infty_G(-r,r) \cap U_{P,r}$ where the derivatives of the functions
$S_r$ and $F_r$ are known to exist. 
Having checked that
\cref{alg:gradproj} is sensible, we can now turn our attention to the convergence analysis.
We start with:

\begin{lemma}[Descent Behavior]
\label{lemma:mon}
Suppose that the termination criterion 
in line \ref{alg:line:4} of 
\cref{alg:gradproj} is not triggered in iteration $i \in \mathbb{N}_0$. 
Then it holds 
\begin{equation}
\label{eq:suffdecrease}
F_r(\uu_i)  \geq 
F_r(\uu_{i+1})
+
\tau_i
\theta
\Theta_{\epsilon_1}(\uu_i)^2
\geq
F_r(\uu_{i+1}).
\end{equation}
\end{lemma}
\begin{proof}
This follows immediately from the line-search procedure in \cref{alg:gradproj}
and \cref{lemma:propertiesTheta}.
\end{proof}

We are now in the position to prove our main convergence result for \cref{alg:gradproj}.

\begin{theorem}[Finite Convergence of the Gradient Projection Method]
\label{th:GradMain}
\cref{alg:gradproj} -- executed with arbitrary 
parameters  $\epsilon_1, \epsilon_2, \sigma > 0$  and $\omega, \theta \in (0,1)$
and with an arbitrary initial guess $0 \leq \uu_0 \in \UU$ -- terminates
after finitely many iterations with a final iterate $\uu_i$ that satisfies 
$\Theta_{\epsilon_1}(\uu_i)  \leq \epsilon_2$.
\end{theorem}

\begin{proof}
We argue by contradiction: 
Suppose that \cref{alg:gradproj}  does not terminate 
after finitely many steps. Then it produces a sequence of 
iterates $0 \leq \uu_i \in  \UU$ which satisfies 
$
F_r(\uu_i)  \geq 
F_r(\uu_{i+1})
+
\tau_i
\theta
\Theta_{\epsilon_1}(\uu_i)^2
$
and
$
\Theta_{\epsilon_1}(\uu_i)  > \epsilon_2
$
for all $i \in \mathbb{N}_0$ due to the termination criterion 
in line \ref{alg:line:4} of \cref{alg:gradproj} and \cref{lemma:mon}. 
As $F_r(\uu_i)   \geq 0$ holds for all $i \in \mathbb{N}_0$ 
by the definition of $F_r$ and the nonnegativity of $\uu_i$, 
this yields 
\[
F_r(\uu_0)
\geq
F_r(\uu_n)
+
\sum_{i=0}^{n-1} 
\tau_i
\theta
\Theta_{\epsilon_1}(\uu_i)^2
\geq
\sum_{i=0}^{n-1} 
\tau_i
\theta
\epsilon_2^2\qquad \forall n \in \mathbb{N}_0
\]
and, thus, $\sum_{i=0}^\infty \tau_i < \infty$. 

We now first prove that the sequence of iterates $\{\uu_i\} \subset \UU$ is bounded in $L^\infty(-r,r)$.
To this end, we note that the update formula in line \ref{alg:line:16}  of 
\cref{alg:gradproj}, \cref{def:gradient}, and the properties of $\PP_\UU^{L^2}$ imply that 
\begin{equation*}
\begin{aligned}
\|\uu_{i+1}\|_{L^\infty(-r,r)}
&=
\| \max\left (0, \uu_i - \tau_i \nabla_{L^2}F_r(\uu_i) \right ) \|_{L^\infty(-r,r)}
\\
&\leq
\left \|  \uu_i - \tau_i \PP_\UU^{L^2}
\left(
- \pp_{2, i} + \nu_1 + \nu_2(\uu_i - \uu_D)
\right )   
\right \|_{L^\infty(-r,r)}
\\
&=
\left \| (1 - \nu_2 \tau_i) \uu_i 
-
\tau_i \PP_\UU^{L^2}
\left(
- \pp_{2, i} + \nu_1 - \nu_2 \uu_D
\right )   
\right \|_{L^\infty(-r,r)} 
\end{aligned}
\end{equation*}
holds for all $i \in \mathbb{N}_0$.
Here, $\pp_{2,i} \in L^\infty_{BV}(-r,r)$ is again defined via the system \eqref{eq:psys_1} for $\uu_i$.
Since $\sum_{i=0}^\infty \tau_i < \infty$ holds, we know that $\tau_i>0$ converges to zero. 
This implies  that there exists an index $m \in \mathbb{N}$
such that $|1 - \nu_2 \tau_i| \leq 1$ holds for all $i \geq m$. Using the fact that 
$\PP_\UU^{L^2}$ is just the identity map in the case $\UU = L^\infty_G(-r,r)$
and given by the cell-wise averaging operator
\[
\PP_\UU^{L^2}(\vv) 
=
\frac{1}{r_n - r_{n-1}} \int_{r_{n-1}}^{r_n} \vv(s) \dd s
\quad
\text{ a.e.\ in } (r_{n-1}, r_n)
 \qquad 
\forall n = 1,...,N\qquad \forall \vv \in  L^\infty_G(-r,r)
\]
in the case where $\UU$ is a finite element space of the form 
\eqref{eq:UUFEcase} (as one may easily check), we further obtain that 
\begin{equation}
\label{eq:projstable}
\big \| \PP_\UU^{L^2}(\vv) \big \|_{L^\infty(-r,r)}
\leq
\left \| \vv \right \|_{L^\infty(-r,r)}
\end{equation}
holds for all $\vv \in L^\infty_G(-r,r)$. In combination with 
\cref{lemma:boundedadjointvars}, all of this allows us to conclude that 
\begin{equation*}
\begin{aligned}
\|\uu_{i+1}\|_{L^\infty(-r,r)}
&\leq
\left \| (1 - \nu_2 \tau_i) \uu_i 
-
\tau_i \PP_\UU^{L^2}
\left(
- \pp_{2, i} + \nu_1 - \nu_2 \uu_D
\right )   
\right \|_{L^\infty(-r,r)}
\\
&\leq
|1 - \nu_2 \tau_i|
\left \|  \uu_i \right \|_{L^\infty(-r,r)}
+
\tau_i 
\left \|
- \pp_{2, i} + \nu_1 - \nu_2 \uu_D
\right \|_{L^\infty(-r,r)}
\\
&\leq
\left \|  \uu_i \right \|_{L^\infty(-r,r)}
+
C \tau_i\qquad \forall i \geq m
\end{aligned}
\end{equation*}
holds with some constant $C>0$. After a trivial induction, the above yields 
\begin{equation*}
\begin{aligned}
\|\uu_{i}\|_{L^\infty(-r,r)}
&\leq
\left \|  \uu_m  \right \|_{L^\infty(-r,r)}
+
C \sum_{j=0}^\infty \tau_j  < \infty \qquad \forall i \geq m.
\end{aligned}
\end{equation*}
The sequence of iterates $\{\uu_i\} \subset \UU$ is thus indeed bounded in 
$L^\infty(-r,r)$ as claimed. 

Next, we show that the sequence $\{\uu_i\} \subset \UU$ even converges 
in $L^\infty(-r,r)$ to a limit function $0 \leq \bar \uu \in \UU$. To this end, 
we note that the update formula in line \ref{alg:line:16}  of 
\cref{alg:gradproj}, the global Lipschitz continuity of the function 
$\R \ni s \mapsto \max(0, s) \in \R$, the nonnegativity of the 
iterates $\uu_i$,  \cref{def:gradient}, the boundedness of the sequence 
$\{\uu_i\}$ in $L^\infty(-r,r)$, 
\cref{lemma:boundedadjointvars}, and again the stability estimate \eqref{eq:projstable} 
imply that there exists a constant $C>0$ satisfying 
\begin{equation}
\label{eq:randomeq272zwbdww2}
\begin{aligned}
\|\uu_{i+1} - \uu_i \|_{L^\infty(-r,r)}
&=
\| \max\left (0, \uu_i - \tau_i \nabla_{L^2}F_r(\uu_i) \right ) - \max(0, \uu_i)\|_{L^\infty(-r,r)}
\\
&\leq
\| \tau_i \nabla_{L^2}F_r(\uu_i) \|_{L^\infty(-r,r)}
\\
&\leq
\tau_i
\left \| \PP_\UU^{L^2}
\left(
- \pp_{2, i} + \nu_1 + \nu_2(\uu_i - \uu_D)
\right )   
\right \|_{L^\infty(-r,r)}
\\
&\leq C \tau_i\qquad \forall i \in \mathbb{N}_0,
\end{aligned}
\end{equation}
where  $\pp_{2,i} \in L^\infty_{BV}(-r,r)$ is defined as before.
Due to the estimate $\sum_{i=0}^\infty \tau_i < \infty$,
the above shows that 
\[
\|\uu_{m} - \uu_n \|_{L^\infty(-r,r)}
\leq
C \sum_{i=n}^\infty \tau_i\qquad \forall m,n \in \mathbb{N},~m \geq n,
\]
holds with $\sum_{i=n}^\infty \tau_i \to 0$ for $n \to \infty$.
The sequence $\{\uu_i\} \subset \UU$ is thus Cauchy 
in $L^\infty(-r,r)$. Since, for both the case $\UU =  L^\infty_G(-r,r)$
and the case where $\UU$ is of the form \eqref{eq:UUFEcase},
$(\UU, \|\cdot\|_{L^\infty(-r,r)})$ is  complete 
(see \cref{cor:Gbanach}), it now follows that there indeed exists
$\bar \uu \in \UU$ such that $\uu_i \to \bar \uu$ holds in $L^\infty(-r,r)$ for $i \to \infty$ as claimed. 
That this limit $\bar \uu$ is nonnegative is obvious. 

Having established that $\{\uu_i\} \subset \UU$ satisfies 
$0 \leq \uu_i \to \bar \uu$ in $L^\infty(-r,r)$
for a function $0 \leq \bar \uu \in \UU$, we can now produce a contradiction as desired:
Since $\tau_i \to 0$ holds, we know that there exists $n \in \mathbb{N}$ 
such that $\tau_i < \sigma$  and $\tau_i/\omega < \epsilon_1$ holds for all $i \geq n$. 
In view of the line-search procedure in lines \ref{alg:line:7} to \ref{alg:line:15}
of \cref{alg:gradproj}, this implies that the 
step size $\tau_i /\omega$ was rejected in line \ref{alg:line:10} and, as 
a consequence, that 
\begin{equation}
\label{eq:randomFr2636gd7}
F_r\left (\max\left (0, \uu_i - \frac{\tau_i}{\omega}  \nabla_{L^2}F_r(\uu_i)\right)\right)
-
F_r(\uu_i) 
>
-
 \frac{\tau_i}{\omega}
\theta
\Theta_{\epsilon_1}(\uu_i)^2
\geq 
-
 \frac{\tau_i}{\omega}
\theta
\Theta_{\tau_i/\omega}(\uu_i)^2\qquad \forall i \geq n,
\end{equation}
where the last estimate follows from \cref{lemma:propertiesTheta}\ref{lemma:propertiesTheta:iii} 
and the inequality $\tau_i/\omega < \epsilon_1$.
As \mbox{$F_r\colon L^\infty_G(-r,r) \cap U_{P,r} \to \R$}
is continuously Fréchet differentiable on $ L^\infty_G(-r,r) \cap U_{P,r}$ by \cref{cor:Fr_regulated},
we can apply the 
mean value theorem \cite[Theorem 3.2.6]{Drabek2007} 
 and \cref{def:gradient}
 to rewrite \eqref{eq:randomFr2636gd7} as
 \begin{equation}
 \label{eq:randomeq27373zhdh3u2odbiud}
\begin{aligned}
&-
\theta
\Theta_{\tau_i/\omega}(\uu_i)^2
\\
&\leq
\frac{\omega}{\tau_i}
\left 
(
F_r\left (\max\left (0, \uu_i - \frac{\tau_i}{\omega}  \nabla_{L^2}F_r(\uu_i)\right)\right)
-
F_r(\uu_i) 
\right )
\\
&=
\int_{0}^1 
F_r'\left ( (1 - s)\uu_i + s \max\left (0, \uu_i - \frac{\tau_i}{\omega}  \nabla_{L^2}F_r(\uu_i)\right) \right )
\left (
\frac{\max\left (0, \uu_i - \frac{\tau_i}{\omega}  \nabla_{L^2}F_r(\uu_i)\right) - \uu_i}{\tau_i / \omega}
\right )\dd s 
\\
&=
-
\int_{0}^1 
\left (
\nabla_{L^2} F_r\left ( (1 - s)\uu_i + s \max\left (0, \uu_i - \frac{\tau_i}{\omega}  \nabla_{L^2}F_r(\uu_i)\right) \right ),
 \min\left ( \frac{\omega}{\tau_i} \uu_i,  \nabla_{L^2}F_r(\uu_i)\right)
\right )_{L^2(-r,r)}\dd s 
\\
&=
-
\left (
\nabla_{L^2} F_r\left ( \uu_i   \right ),
 \min\left ( \frac{\omega}{\tau_i} \uu_i,  \nabla_{L^2}F_r(\uu_i)\right)
\right )_{L^2(-r,r)} + \varrho_i
\\
&=
- \Theta_{\tau_i/\omega}(\uu_i)^2 + \varrho_i \qquad \forall i \geq n,
\end{aligned}
\end{equation}
where the last line holds by \cref{def:statindicators} and where $\varrho_i$ denotes the remainder term
\begin{equation}
\label{eq:varrhodef}
\begin{aligned}
&\varrho_i
:=
\int_{0}^1 
\left (
\nabla_{L^2} F_r\left ( \uu_i  \right )
-
\nabla_{L^2} F_r\left  ( \uu_i - s \min\left ( \uu_i, \frac{\tau_i}{\omega}  \nabla_{L^2}F_r(\uu_i)\right) \right ),
 \min\left ( \frac{\omega}{\tau_i} \uu_i,  \nabla_{L^2}F_r(\uu_i)\right)
\right )_{L^2(-r,r)}\dd s.
\end{aligned}
\end{equation}
Note that, due to the convergence $ \uu_i \to \bar \uu $
in $L^\infty(-r,r)$,
the estimates in \eqref{eq:randomeq272zwbdww2}, and the nonnegativity of $\uu_i$, we know that 
\begin{equation}
\label{eq:randomeq273hdbd3833}
\begin{aligned}
\left \|
\uu_i - s \min\left ( \uu_i, \frac{\tau_i}{\omega}  \nabla_{L^2}F_r(\uu_i) \right )
- \bar \uu
\right \|_{L^\infty(-r,r)}
&\leq
\left \|
\uu_i -  \bar \uu
\right \|_{L^\infty(-r,r)}
+
\left \|
 \min\left ( \uu_i, \frac{\tau_i}{\omega}  \nabla_{L^2}F_r(\uu_i)\right )
\right \|_{L^\infty(-r,r)}
\\
&\leq\left \|
\uu_i -  \bar \uu
\right \|_{L^\infty(-r,r)}
+
\frac{1}{\omega}
\left \|
\tau_i  \nabla_{L^2}F_r(\uu_i)
\right \|_{L^\infty(-r,r)}
\\
&\to 0\qquad \text{ for }i \to \infty\qquad \forall s \in [0,1].
\end{aligned}
\end{equation}
In combination with \cref{lemma:contigrads}, this implies that 
\[
\nabla_{L^2} F_r\left ( \uu_i  \right )
-
\nabla_{L^2} F_r\left  ( \uu_i - s \min\left ( \uu_i, \frac{\tau_i}{\omega}  \nabla_{L^2}F_r(\uu_i)\right) \right )
\to
\nabla_{L^2} F_r\left ( \bar \uu  \right )
-
\nabla_{L^2} F_r\left ( \bar \uu\right )
=
0
\]
holds in $L^2(-r,r)$ for all $s \in [0,1]$ as $i$ goes to infinity.  
By exploiting the formula \eqref{eq:L2gradformula}, the stability estimate 
\eqref{eq:projstable}, \cref{lemma:boundedadjointvars},
the boundedness of $\{\uu_i\}$ in $L^\infty(-r,r)$, and the same estimates 
as in \eqref{eq:randomeq273hdbd3833},
we further obtain that 
there exist constants $C_1, C_2, C_3 > 0$ independent of $i$ and $s \in [0,1]$ satisfying 
\begin{equation*}
\begin{aligned}
&\left \|
\nabla_{L^2} F_r\left  ( \uu_i - s \min\left ( \uu_i, \frac{\tau_i}{\omega}  \nabla_{L^2}F_r(\uu_i)\right) \right )
\right \|_{L^\infty(-r,r)}
\\
&\qquad \leq 
C_1
+
\left \|
\nu_1 +  
\nu_2
 \left  ( \uu_i - s \min\left ( \uu_i, \frac{\tau_i}{\omega}  \nabla_{L^2}F_r(\uu_i)\right) \right )
  - \nu_2\uu_D
\right \|_{L^\infty(-r,r)}
\\
&\qquad \leq 
C_2
+\nu_2
\left \|
\min\left ( \uu_i, \frac{\tau_i}{\omega}  \nabla_{L^2}F_r(\uu_i)\right)
\right \|_{L^\infty(-r,r)}
\\
&\qquad \leq C_3\qquad \forall s \in [0,1]\qquad \forall i \in \mathbb{N}.
\end{aligned}
\end{equation*}
Again due to the nonnegativity of the iterates $\uu_i$,
we may now conclude that 
there exist constants $C_4, C_5, C_6>0$ independent of $s \in [0,1]$ and $i \in \mathbb{N}$
such that 
\begin{equation*}
\begin{aligned}
&\left |
\left (
\nabla_{L^2} F_r\left ( \uu_i  \right )
-
\nabla_{L^2} F_r\left  ( \uu_i - s \min\left ( \uu_i, \frac{\tau_i}{\omega}  \nabla_{L^2}F_r(\uu_i)\right) \right ),
 \min\left ( \frac{\omega}{\tau_i} \uu_i,  \nabla_{L^2}F_r(\uu_i)\right)
\right )_{L^2(-r,r)}
\right |
\\
&\leq
C_4
\left \|
\min\left ( \frac{\omega}{\tau_i} \uu_i,  \nabla_{L^2}F_r(\uu_i)\right)
\right \|_{L^2(-r,r)}
\\
&\leq
C_5
\end{aligned}
\end{equation*}
holds for all $s \in [0,1]$ and all $i \in \mathbb{N}$ and such that 
\begin{equation*}
\begin{aligned}
&\left |
\left (
\nabla_{L^2} F_r\left ( \uu_i  \right )
-
\nabla_{L^2} F_r\left  ( \uu_i - s \min\left ( \uu_i, \frac{\tau_i}{\omega}  \nabla_{L^2}F_r(\uu_i)\right) \right ),
 \min\left ( \frac{\omega}{\tau_i} \uu_i,  \nabla_{L^2}F_r(\uu_i)\right)
\right )_{L^2(-r,r)}
\right |
\\
&\leq
C_6
\left \|
\nabla_{L^2} F_r\left ( \uu_i  \right )
-
\nabla_{L^2} F_r\left  ( \uu_i - s \min\left ( \uu_i, \frac{\tau_i}{\omega}  \nabla_{L^2}F_r(\uu_i)\right) \right )
\right \|_{L^2(-r,r)}
\\
&\to 0
\end{aligned}
\end{equation*}
holds for all $s \in [0,1]$ as $i$ goes to infinity. 
If we combine all of the above with \eqref{eq:varrhodef}
and invoke the dominated convergence theorem, then 
it follows that $\varrho_i \to 0$ holds for $i \to \infty$. Due to the estimate in \eqref{eq:randomeq27373zhdh3u2odbiud}
and the fact that 
$\epsilon_2^2 \leq \Theta_{\epsilon_1}(\uu_i)^2 \leq \Theta_{\tau_i/\omega}(\uu_i)^2$
holds for all large enough $i$
(cf.\ \eqref{eq:randomFr2636gd7}), this yields 
\[
0
<
(1 - \theta) \epsilon_2^2
\leq
(1 - \theta) \Theta_{\tau_i/\omega}(\uu_i)^2
\leq
\varrho_i \to 0\qquad \text{ for } i \to \infty,
\]
which is impossible. 
\Cref{alg:gradproj} thus has to terminate after finitely many iterations 
as claimed in the theorem 
and the proof is complete.  
\end{proof}

We would like to point out that -- for the control space 
 $\UU = L^\infty_G(-r,r)$ --
 it is crucial 
to establish the existence of the limit $\bar \uu$ in the proof of \cref{th:GradMain} 
via the summability condition $\sum_{i=0}^\infty \tau_i < \infty$ 
that is obtained in the case of nonfinite convergence from the sufficient decrease condition 
\eqref{eq:suffdecrease}.
Indeed, only by exploiting this summability, 
we are able to prove that the sequence $\{\uu_i\}$
is convergent in $L^\infty(-r,r)$ and, thus, convergent in a topology 
for which continuity results for the gradient are available by \cref{lemma:contigrads}.
If one tries, for example, to exploit that the coercivity properties 
of the objective function $F_r$ imply that sequences 
of iterates generated by a gradient descent 
method for the numerical solution of \eqref{eq:PrU}
possess weak accumulation points in $L^2(-r,r)$,
then this turns out to be not very useful
since it is completely unclear how to establish that these accumulation points are
stationary in a meaningful sense if the space $\UU$ is infinite-dimensional.
(For a finite-dimensional space $\UU$ of the type \eqref{eq:UUFEcase}, 
these issues related to the topologies are, of course, completely irrelevant.)

\section{Numerical Experiments}
\label{sec:9}

We conclude this paper with numerical experiments that validate the analytical results of 
\cref{sec:3,sec:4,sec:5,sec:6,sec:7,sec:8}.
As examples, we consider two different instances of the problem 
\begin{equation}
\label{eq:modprobnum}
\left.
	\begin{aligned}
		\text{Minimize} 
		\quad &  \frac12 \|S_r(\uu) - y_D\|_{L^2(\Omega)}^2 
		+ \int_{-r}^r \nu_1  \uu + \frac{\nu_2}{2}( \uu - \uu_D)^2\dd t \\
        \text{w.r.t.}
        \quad &\uu \in L^\infty_G(-r,r),\\
		\text{s.t.} \quad &  \uu \geq 0 \text{ a.e.\  in } (-r,r),
	\end{aligned}~~
\right \}
\end{equation}
studied in \cref{sec:8}.

\subsection{Example 1:  A Test Case With a Known Analytical Solution} 
\label{ExampleA}
In our first test configuration,  the parameters in \eqref{eq:modprobnum} 
are chosen as
\begin{equation}
\label{eq:ex1defs}
\begin{gathered}
 r = 3,\qquad
\Omega = (0,1)^2,\qquad
 \nu_1 = 0,\qquad
 \nu_2 = 10^{-3}, 
 \\
 y_D(x_1,x_2) = \sin(2 \pi x_1) \sin(2 \pi x_2)  \text{ f.a.a.\ } (x_1, x_2)^\top \in \Omega,
 \\
 \uu_D(s) = 1  \text{ f.a.a.\ } s \in (-r,r),
 \\
  f(x_1,x_2) = (8 \pi^2 + 1)\sin(2 \pi x_1) \sin(2 \pi x_2)  \text{ f.a.a.\ } (x_1, x_2)^\top \in \Omega.
 \end{gathered}
\end{equation}
Note that, for the above choice of $f$, the solution $y_P$
of the Poisson problem \eqref{eq:fPoisson} is given by
\[
y_P(x_1, x_2) = \frac{(8 \pi^2 + 1)}{8 \pi^2 }\sin(2 \pi x_1) \sin(2 \pi x_2) 
\text{ f.a.a.\ } (x_1, x_2)^\top \in \Omega.
\]
In particular, we have $r_P = 2\|y_P\|_\infty = (8 \pi^2 + 1)/(4 \pi^2) \approx  2.0253$ and,
as a consequence, $r \geq r_P$. In combination with \cref{th:truncate,th:BVreg}, 
the $L^\infty_G(-r,r)$-regularity of $\uu_D$, 
and the fact that the right-hand side $f$ in \eqref{eq:ex1defs} 
satisfies \Atwo{} by \cite[Proposition~1]{Mityagin2020},
this shows that \eqref{eq:modprobnum} can be reformulated 
as a problem of the type \eqref{eq:Pr} or
\eqref{eq:P}, respectively, (with $u_D \in L^2(\R)$ as an arbitrary extension
of $\uu_D$ to the whole real line) and that the setting \eqref{eq:ex1defs} is covered by 
the results of \cref{sec:3,sec:4,sec:5,sec:6,sec:7,sec:8}. 
Using the formulas for $y_D$, $\uu_D$, and $f$,
one further easily checks that \eqref{eq:modprobnum} possesses the unique 
global solution $\bar \uu  = \uu_D \equiv 1$
in the situation of \eqref{eq:ex1defs} and that the optimal value of \eqref{eq:modprobnum} is zero.
A known analytical solution is thus indeed available under the assumptions of \eqref{eq:ex1defs}.
We remark that the configuration in \eqref{eq:ex1defs} 
is a rather exceptional one since, for the chosen $\uu_D$ and $y_D$, 
the adjoint variables $\bar p_1$ and $\bar \pp_2$ in the stationarity system \eqref{eq:pdsysPr}
satisfied by $\bar \uu$ and $\bar y$ are zero and the coefficient 
in front of the Dirac measure in the formula for the distributional derivative 
of $\bar \pp_2$ vanishes; see \eqref{eq:pushypushy}. This is a consequence of the construction 
that we have used in \eqref{eq:ex1defs} 
to obtain an instance of \eqref{eq:modprobnum} with a known analytical solution. 

\subsection{Example 2: A Test Case Without a Known Analytical Solution} 
\label{ExampleB}

In our second example, we suppose that the quantities in \eqref{eq:modprobnum}
are given by 
\begin{equation}
\label{eq:ex2defs}
\begin{gathered}
 r = 2,\qquad
\Omega = (0,1)^2,\qquad
 \nu_1 \in \left \{0, \frac{1}{4096}, \frac{1}{2048},\frac{1}{1024}, \frac{1}{512}, \frac{1}{256}, \frac{1}{128} \right \},\qquad
 \nu_2 = 10^{-4}, 
 \\
 y_D(x_1,x_2) = -0.125 + 0.275\,\sin(2 \pi x_1) \sin(2 \pi x_2)   \text{ f.a.a.\ } (x_1, x_2)^\top \in \Omega,
 \\
 \uu_D(s) = 0 \text{ f.a.a.\ } s \in (-r,r),
 \\
  f(x_1,x_2) =  8 \pi^2 \sin(2 \pi x_1) \sin(2 \pi x_2) \text{ f.a.a.\ } (x_1, x_2)^\top \in \Omega.
 \end{gathered}
\end{equation}
For this configuration, one easily checks that 
$r_P = r = 2$ holds, that $\uu_D$ is an element of $L^\infty_G(-r,r)$, and that $f$ satisfies 
condition \Atwo{}.
 The problem \eqref{eq:modprobnum} 
can thus again be recast in the form \eqref{eq:Pr} or \eqref{eq:P}, respectively,  (with $u_D = 0 \in L^2(\R)$)
and all of the results in \cref{sec:3,sec:4,sec:5,sec:6,sec:7,sec:8} are applicable.
To our best knowledge, an analytical solution  is not available
for the problem \eqref{eq:modprobnum} in the 
situation of \eqref{eq:ex2defs}. 

\subsection{Details on the Numerical Realization}
\label{subsec:ImplementationDetails}

To be able to solve the problem \eqref{eq:modprobnum} numerically, 
one has to discretize the control space $L^\infty_G(-r,r)$ and 
the governing PDE \eqref{eq:PDE}. For the 
discretization of the controls $\uu$, we 
consider piecewise constant 
finite element functions on an equidistant mesh of width $h_\uu = (2 r)/N_\uu$, 
$N_\uu \in \mathbb{N}$, i.e., 
we replace $L^\infty_G(-r,r)$ with the space 
\[
\UU_{h_\uu}
:=
\left \{
\uu \in L^\infty (-r,r)
\mid 
\uu = \mathrm{const} \text{ a.e.\ on } ((n-1)h_\uu - r, n h_\uu - r) ~\forall n=1,...,N_\uu
\right \} \subset  L^\infty_G(-r,r).
\]
Note that, if we discretize the controls $\uu$ in \eqref{eq:modprobnum}
along the above lines and keep the exact PDE, then the 
resulting semidiscrete minimization problem 
still satisfies the assumptions of \cref{sec:8}; cf.\ \eqref{eq:UUFEcase}.
This means that the convergence result established for 
\cref{alg:gradproj} in \cref{th:GradMain} applies to \eqref{eq:modprobnum}
in both the fully continuous and the control-discrete (but state-continuous) setting.
For the discretization of the PDE, 
this is different. In fact, it is easy to check that the results in
\cref{sec:6,sec:7,sec:8} \emph{do not} carry over to the 
state-discretized setting; see the arguments used in the proof of \cref{lem:levelstates}.
At the end of the day, this means that the algorithm that is run in practice 
in the numerical experiments presented below
should not be interpreted as a gradient descent algorithm for 
a control- and state-discrete version of \eqref{eq:modprobnum} but rather 
as a gradient descent algorithm for a 
control-discrete and state-continuous version of \eqref{eq:modprobnum} that is executed with 
inexact gradient and function value evaluations where the inexactness arises 
from the discretization of the governing PDE (\emph{first optimize -- then discretize}; cf.\
\cite{Cohen2018OnAW,Hertlein2019,Hinze2012}). 
As we will see in \cref{subsec:9.4}, these peculiarities of the discretization are also reflected 
by the numerical results. To keep it simple,
in what follows, we restrict our attention to 
the case where the state $y$ and the governing PDE
are discretized by means of
globally continuous and piecewise affine finite element functions 
on Friedrichs-Keller triangulations 
of width $h_y = 1/N_y$, $N_y \in \mathbb{N}$.
For the solution of the state equation and the PDE in \eqref{eq:psys_1},
that has to be solved to evaluate the 
$L^2$-gradient $\nabla_{L^2}F_r(\uu_i) \in 
\UU_{h_\uu}$ in step \ref{alg:line:3} of \cref{alg:gradproj},
we use the exact same scheme as in \cite{Christof2021,Christof2018nonsmoothPDE}, i.e., 
we employ mass lumping in the terms associated with the semilinearity and 
solve the resulting nonlinear system of equations with a semismooth Newton method. 
We remark that a very delicate point in the numerical realization of  \cref{alg:gradproj}
is the computation of the integrals appearing in \eqref{eq:psys_1}.
In our experiments, we evaluate the integrals in the formula for $\pp_2$ in \eqref{eq:psys_1}
by rewriting them as $L^2(\Omega)$-scalar products and by subsequently employing 
a quadrature rule based on the mass-matrix of the finite element space used for the discretization of $y$. 
For the cell-wise averaging that is necessary to calculate the 
$L^2(-r,r)$-projection onto $\UU_{h_\uu}$ in \eqref{eq:L2gradformula}, we use 
the quadrature rule 
\begin{equation}
\label{eq:cell_wise_quad_rule}
\frac{1}{h_\uu} \int_{(n-1)h_\uu - r}^{n h_\uu - r}
\vv(s) \dd s
\approx 
\frac{1}{Q}
\sum_{j=1}^Q 
\vv\left ((n-1)h_\uu - r +  \frac{j}{Q+1}h_\uu \right ),
\qquad  n=1,...,N_\uu,\qquad Q \in \N.
\end{equation}
Note that this rule avoids evaluating the integrand at the 
endpoints of the cells where the elements of $\UU_{h_\uu}$
are discontinuous. Numerical experiments suggest that it is advantageous 
to determine the $L^2(-r,r)$-projection onto $\UU_{h_\uu}$ in \eqref{eq:L2gradformula}
by means of quadrature rules 
whose evaluation points are relatively far away from the discontinuities of the controls.
(This is most likely caused by an interaction with the quadrature rule used in \eqref{eq:psys_1}.)

\subsection{Numerical Results}
\label{subsec:9.4}

The results that are obtained when 
\cref{alg:gradproj} is applied to problem \eqref{eq:modprobnum} in the situation of 
Example~\ExA{} and the discretization described in \cref{subsec:ImplementationDetails} 
is used for the numerical realization can be seen in 
\Cref{fig:Ex1-1,fig:Ex1-2} and \Cref{tab:Ex1Table1,tab:Ex1Table2} below. In all of these experiments, 
the line-search in \cref{alg:gradproj} 
was performed with the parameters $\sigma = 512$ and $\omega = \theta = 0.8$,
the semismooth Newton method used for the state equation
was run with the tolerance $10^{-12}$, the initial iterate $\uu_0$ was 
zero, the relaxation parameter $\epsilon_1$ appearing
in the stationarity measure $\Theta_{\epsilon_1}$ was set to $10^{-16}$,
and the number of cell-wise quadrature points in \eqref{eq:cell_wise_quad_rule} was 
chosen as
$Q = 5$.
\Cref{fig:Ex1-1,fig:Ex1-2} and \Cref{tab:Ex1Table1}
depict the convergence behavior of \cref{alg:gradproj} in the case 
where the tolerance $\epsilon_2$ in the termination criterion 
is set to zero.
(We stopped the algorithm in these experiments when the step size $\tau_i$ in the line-search loop 
of  \cref{alg:gradproj} became smaller than $10^{-10}$). As \Cref{fig:Ex1-1} shows, 
in this test configuration, the level that our gradient projection algorithm 
can drive the objective value and the stationarity measure to depends 
greatly on the  widths used for the discretization. 
For the test case with the largest widths ($h_y = h_\uu = 1/32$),
\cref{alg:gradproj} only achieves a final objective value of $\approx 1.424\cdot 10^{-5}$ 
and a final stationarity measure of $\approx 7.580\cdot 10^{-7}$. 
If the widths are reduced, then the ultimately obtained accuracy
improves significantly until, for the smallest widths ($h_y = h_\uu = 1/1024$),
we arrive at a final objective value of $\approx 1.397\cdot 10^{-11}$ and a final stationarity measure 
of $\approx 1.522 \cdot 10^{-11}$. Note that the latter values are equal
to the objective value/stationarity measure of the exact solution $\bar \uu \equiv 1$
(namely, zero)
up to an error that is not much larger than the tolerance $10^{-12}$ of the 
semismooth Newton method used for the solution of the state equation.
At this level of accuracy, the tolerance of the Newton method and 
the quadrature errors begin to limit the gains that are obtained by choosing smaller widths $h_y$ and $h_\uu$;
see the last two lines in
the second plot of \Cref{fig:Ex1-1}.

We remark that the above observations are 
consistent with the interpretation of the implemented algorithm as a gradient descent method for a 
control-discrete and state-continuous version of \eqref{eq:modprobnum} that is executed with 
inexact gradient and function value evaluations; see \cref{subsec:ImplementationDetails}: 
If the discretization is too coarse, 
then the level of inexactness is too high and the perturbation of the gradients 
prevents the method from reducing $\Theta_{\epsilon_1}$ to acceptable levels. Only once the widths are small enough, 
the approximations of the ``real'' gradients of the control-discrete and state-continuous 
objective function obtained from the FE-discretization of the involved PDEs
are sufficiently accurate to achieve an adequate reduction.
We remark that, by varying the width $h_\uu$ in the situation of 
\Cref{fig:Ex1-1,fig:Ex1-2}, one can check that the width $h_y$ is decisive for the 
final accuracy level while the choice of $h_\uu$ only has little influence.
 This reflects that the discretization of the PDEs
is indeed the source of the inexactness of the gradients
as described in \cref{subsec:ImplementationDetails}.
Compare also with the analysis in \cref{sec:8} in this context 
which is completely unaffected by the discretization of the control space.

A further important feature of \cref{alg:gradproj}
visible in \Cref{fig:Ex1-1} is the property of asymptotic mesh independence:
As can be seen in \Cref{fig:Ex1-1}, the graphs that
depict how the objective value and the stationarity measure 
decrease with the iteration number
initially all have the same slope  and only branch off once the maximum accuracy
of the respective discretization level is reached.
In practice, this means that, if the goal is to calculate a control
whose stationarity measure $\Theta_{\epsilon_1}$  is smaller than a 
given tolerance $\epsilon_2 > 0$, then \cref{alg:gradproj} is able to accomplish this 
with a number of iterations that does not tend to infinity as $h_\uu$ and $h_y$
go to zero but is effectively constant for all widths that are sufficiently 
small; see \cref{tab:Ex1Table2}. 

As a final observation, we would like to point out that, even though 
the analytical solution $\bar \uu$ is continuous in the situation of 
Example \ExA, the approximate solutions 
calculated by \cref{alg:gradproj}
possess discontinuities at the origin; see \Cref{fig:Ex1-2}. This illustrates that
problems of the type \eqref{eq:P} indeed promote a 
jump at zero and that cases with continuous 
solutions are exceptional and typically instable w.r.t.\ small perturbations. 
We remark that this can also be seen in the formula \eqref{eq:pushypushy} 
in which the coefficient in front of the Dirac measure only vanishes 
if $p_1$ has mean value zero. Note that, in the limit $h_y \to 0$ and $h_\uu \to 0$,
the discontinuities in \Cref{fig:Ex1-2} vanish and the approximate solutions 
converge to the exact optimal control $ \bar \uu \equiv 1$. For the finest discretization,
the $L^\infty(-r,r)$-error between the approximate and the exact solution is 
$\approx 5.684\cdot 10^{-5}$; see \Cref{tab:Ex1Table1}. 

\vspace{0.0cm}
\begin{figure}[H]
\begin{center}
\subfigure{\includegraphics[width=7.5cm,height=8.25cm,keepaspectratio]{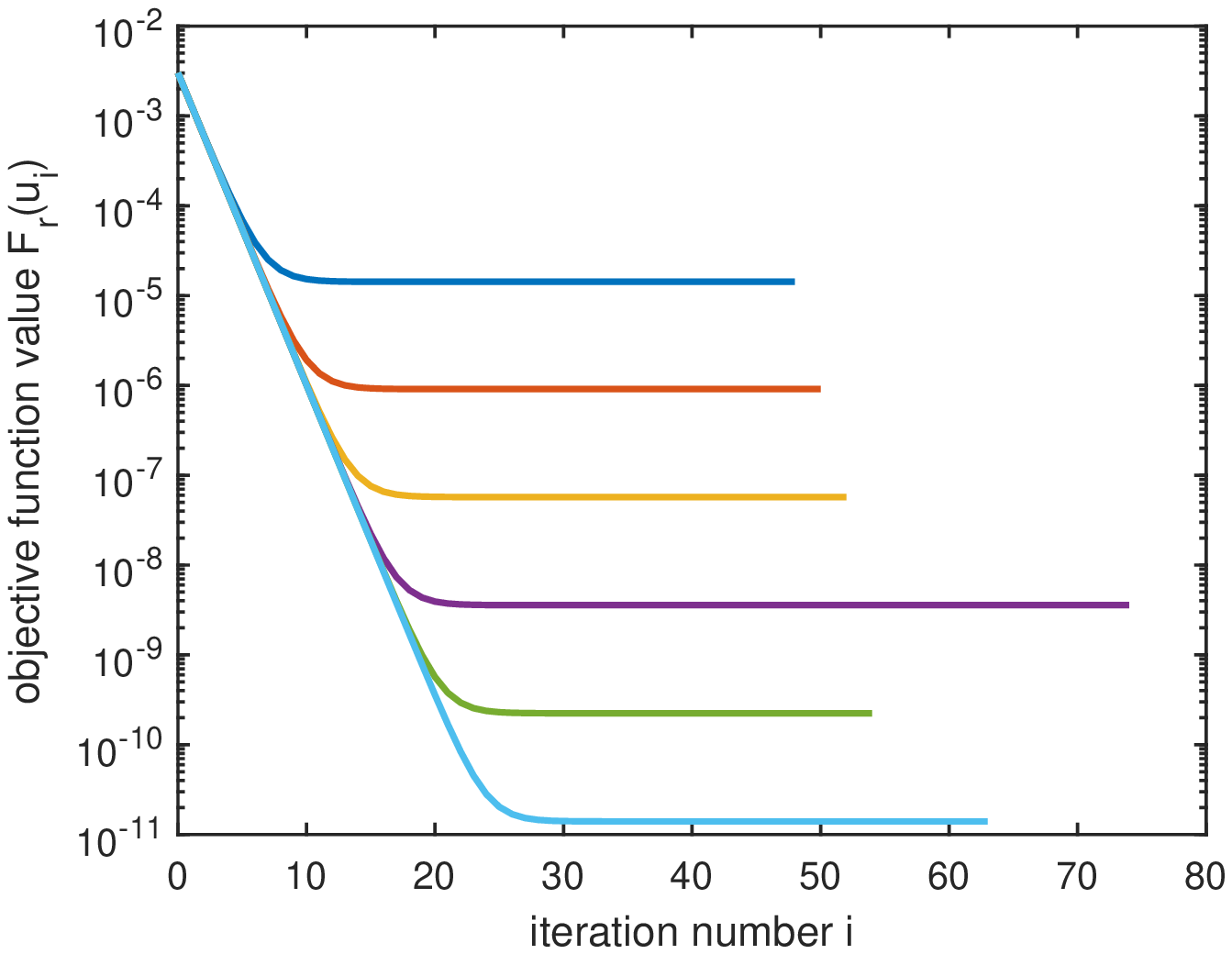}} 
\hspace{0.3cm}
\subfigure{\includegraphics[width=7.5cm,height=8.25cm,keepaspectratio]{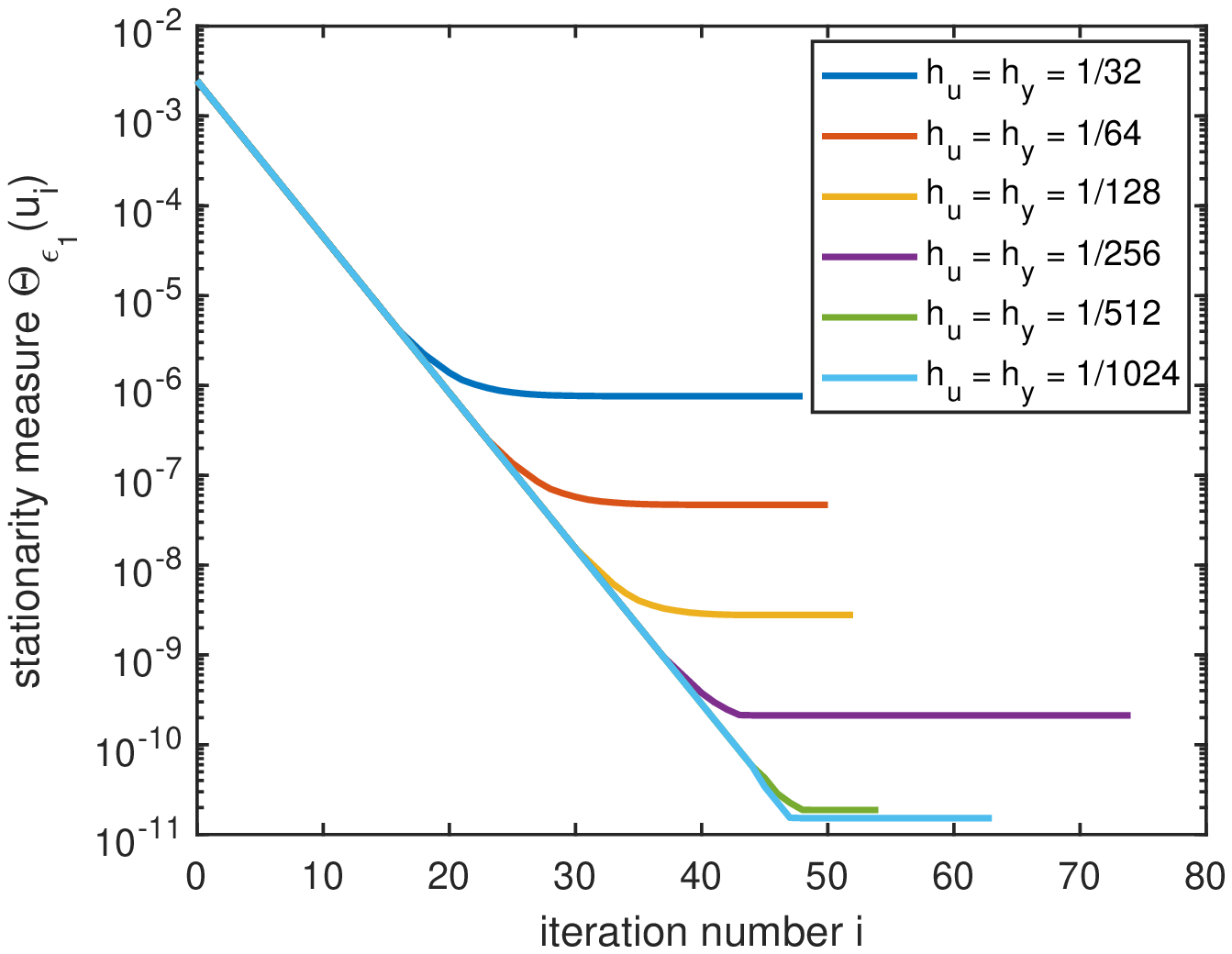}} 
\vspace{0.2cm}
\caption{Value of the objective function (left) and the stationarity measure 
$\Theta_{\epsilon_1}(\uu_i)$ (right) as 
functions of the iteration counter $i$ of \cref{alg:gradproj}
in the case of Example \ExA{} for different widths $h_\uu$ and $h_y$.
The legend refers to both figures.
It can be seen that 
the reduction of the objective value and stationarity measure stagnates at a 
threshold that depends on the discretization level. This reflects that, by discretizing the involved PDEs, 
one introduces an error and, thus, causes 
\cref{alg:gradproj} to run with inexact gradient and function value evaluations; see \Cref{subsec:ImplementationDetails}.}
\label{fig:Ex1-1}
\end{center}
\end{figure} 

\begin{figure}[H]
\begin{center}
\subfigure{\includegraphics[width=7.21cm,height=8.25cm,keepaspectratio]{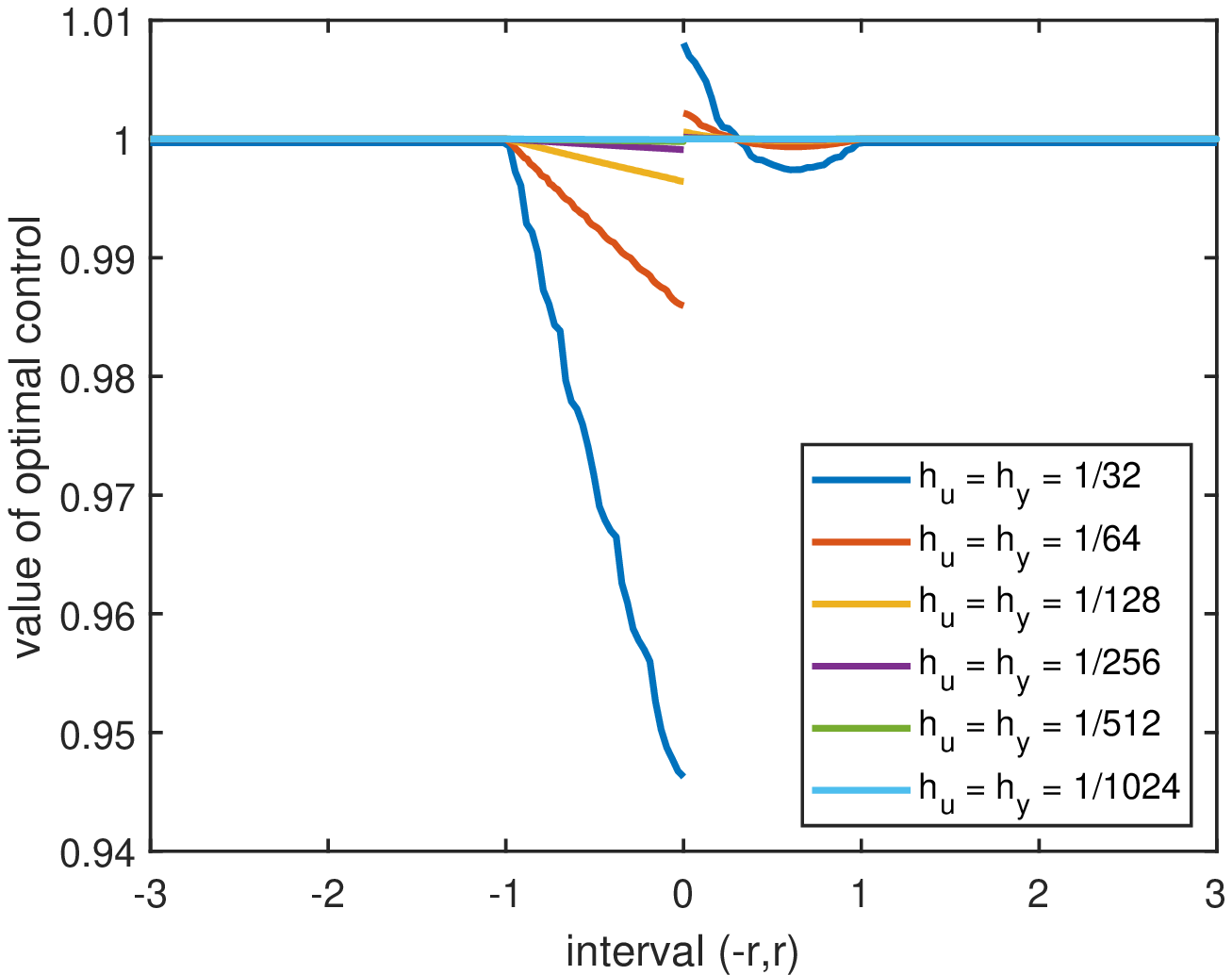}} 
\subfigure{\includegraphics[width=8cm,height=8.25cm,keepaspectratio]{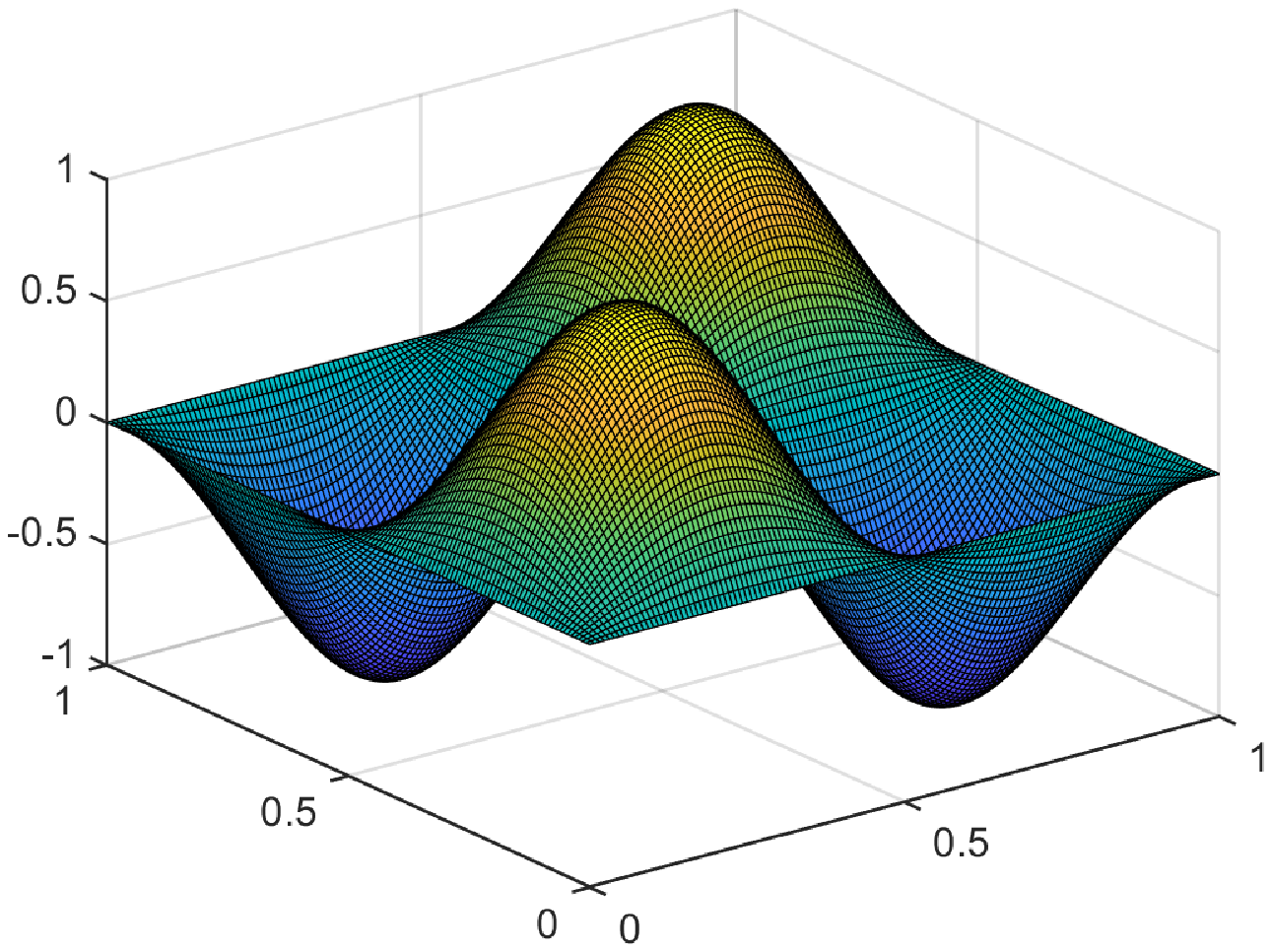}} 
\vspace{0.2cm}
\caption{Approximations of the optimal control $\bar \uu$
obtained from \cref{alg:gradproj} at the end of the calculations 
depicted in \Cref{fig:Ex1-1} for different mesh widths $h_\uu$ and $h_y$ (left)  and resulting optimal state $\bar y$
for $h_\uu = h_y = 1/128$ (right).}
\label{fig:Ex1-2}
\end{center}
\end{figure}
\vspace{-0.2cm}
\begin{table}[H]
\begin{tabular}{c | c c c c c c }
\hline\noalign{\smallskip}
$h_\uu = h_y$   & $1/32$ & $1/64$ & $1/128$ & $1/256$ & $1/512$  & 1/1024 \\
\noalign{\smallskip}\hline\noalign{\smallskip}
 \begin{tabular}{@{}c@{}}$F_r$ \end{tabular} 
   & $ 1.424\cdot 10^{-5}$   & $9.095\cdot 10^{-7}$   & $ 5.715 \cdot 10^{-8}$ & $3.577\cdot 10^{-9}$ & $2.236\cdot 10^{-10}$ & $1.397\cdot 10^{-11}$ \\
\noalign{\smallskip}\hline\noalign{\smallskip}
 \begin{tabular}{@{}c@{}} $\Theta_{\epsilon_1}$ \end{tabular} 
  & $7.580 \cdot 10^{-7}$   & $ 4.672\cdot 10^{-8}$   & $ 2.781\cdot 10^{-9}$ & $2.123\cdot 10^{-10}$ & $1.875\cdot10^{-11}$ & $1.522 \cdot 10^{-11}$\\
\noalign{\smallskip}\hline\noalign{\smallskip}
 \begin{tabular}{@{}c@{}} $L^\infty$-error \end{tabular} 
  & $ 5.373 \cdot 10^{-2}$   & $1.404 \cdot 10^{-2}$   & $3.580 \cdot 10^{-3} $ & $ 9.033 \cdot 10^{-4}$ & $2.268\cdot 10^{-4}$ & $5.684\cdot 10^{-5}$
\\
\noalign{\smallskip}\hline\noalign{\smallskip} 
\end{tabular}
~\\[0.4cm]
\caption{
Final objective value (second row), stationarity measure (third row),
and $L^\infty(-r,r)$-error between the approximate solution
obtained from \cref{alg:gradproj} and $\bar \uu \equiv 1$ (fourth row) 
at the end of the computations depicted in  \Cref{fig:Ex1-1} for different widths $h_\uu$ and $h_y$ (first row). 
The measured $L^\infty(-r,r)$-errors suggest that, for this test case, the convergence is quadratic.}
\label{tab:Ex1Table1}   
\end{table} 
\vspace{-0.2cm}
\begin{table}[H]
\begin{tabular}{c | c c c c c c }
\hline\noalign{\smallskip}
\diagbox{$h_y$}{$h_\uu$}  & $1/32$ & $1/64$ & $1/128$ & $1/256$ & $1/512$  & 1/1024 
\\
\noalign{\smallskip}\hline\noalign{\smallskip}
 \begin{tabular}{@{}c@{}}$1/80$ \end{tabular} 
& $\infty$    & $\infty$    &  $\infty$  & $\infty$  & $\infty$  & $\infty$ 
\\
\noalign{\smallskip}\hline\noalign{\smallskip}
 \begin{tabular}{@{}c@{}}$1/90$ \end{tabular} 
& $37$    & $38$    &  $\infty$  & $\infty$  & $\infty$  & $\infty$ 
\\
\noalign{\smallskip}\hline\noalign{\smallskip}
 \begin{tabular}{@{}c@{}}$1/100$ \end{tabular} 
& $32$    & $32$    &  $33$  & $34$  & $34$  & $34$ 
\\
\noalign{\smallskip}\hline\noalign{\smallskip}
 \begin{tabular}{@{}c@{}}$1/110$ \end{tabular} 
& $32$    & $32$    &  $32$  & $32$  & $32$  & $32$  
\\
\noalign{\smallskip}\hline\noalign{\smallskip}
 \begin{tabular}{@{}c@{}}$1/120$ \end{tabular} 
& $32$    & $32$    &  $32$  & $32$  & $32$  & $32$  
\\
\noalign{\smallskip}\hline\noalign{\smallskip}
 \begin{tabular}{@{}c@{}}$1/130$ \end{tabular} 
& $32$    & $32$    &  $32$  & $32$  & $32$  & $32$  
\\
\noalign{\smallskip}\hline\noalign{\smallskip}
\end{tabular}
~\\[0.4cm]
\caption{Number of gradient projection steps needed to drive the stationarity measure 
$\Theta_{\epsilon_1}$ below $\epsilon_2 = 10^{-8}$ in the situation 
of Example \ExA{} for different $h_y$ and $h_\uu$. It can be seen that 
the convergence is mesh-independent for all $h_y$ that are smaller than
$\approx 1/110$; cf.\ \Cref{fig:Ex1-1}.}
\label{tab:Ex1Table2}   
\end{table} 

Having demonstrated that \cref{alg:gradproj} works and 
is able to reproduce the analytical solution $\bar \uu = 1$ in the situation of Example \ExA{}, 
we now turn our attention to Example \ExB{} for which no analytical solution is available. 
The results that we have obtained for Example \ExB{} can be seen in \Cref{fig:Ex2-1,fig:Ex2-2,tab:Ex2Table1} below.
In all of the depicted experiments, the parameters 
in \cref{alg:gradproj}  were chosen as $\sigma = 2048$, 
$\omega = \theta = 0.8$, $\epsilon_1 = 10^{-16}$, and $\epsilon_2 = 10^{-8}$;
the tolerance of the semismooth Newton method  
was set to $10^{-12}$; the initial guess was $\uu_0 = 0$;
and the controls and states were discretized as described in \cref{subsec:ImplementationDetails}
with $Q = 10$.
\vspace{0.3cm}
\begin{figure}[H]
\begin{center}
\subfigure{\includegraphics[width=7.2cm,height=8.25cm,keepaspectratio]{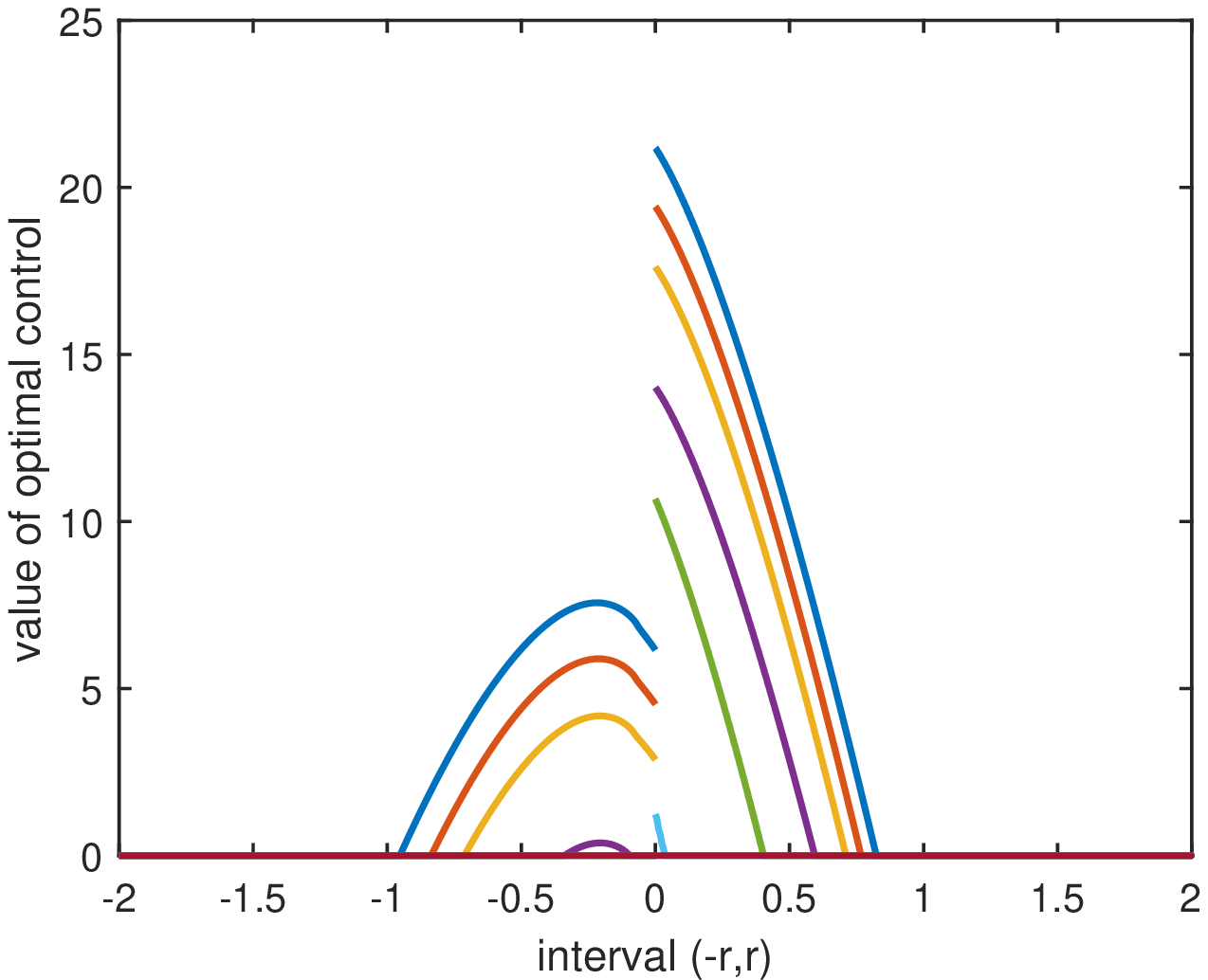}} 
\hspace{0.3cm}
\subfigure{\includegraphics[width=7.3cm,height=8.25cm,keepaspectratio]{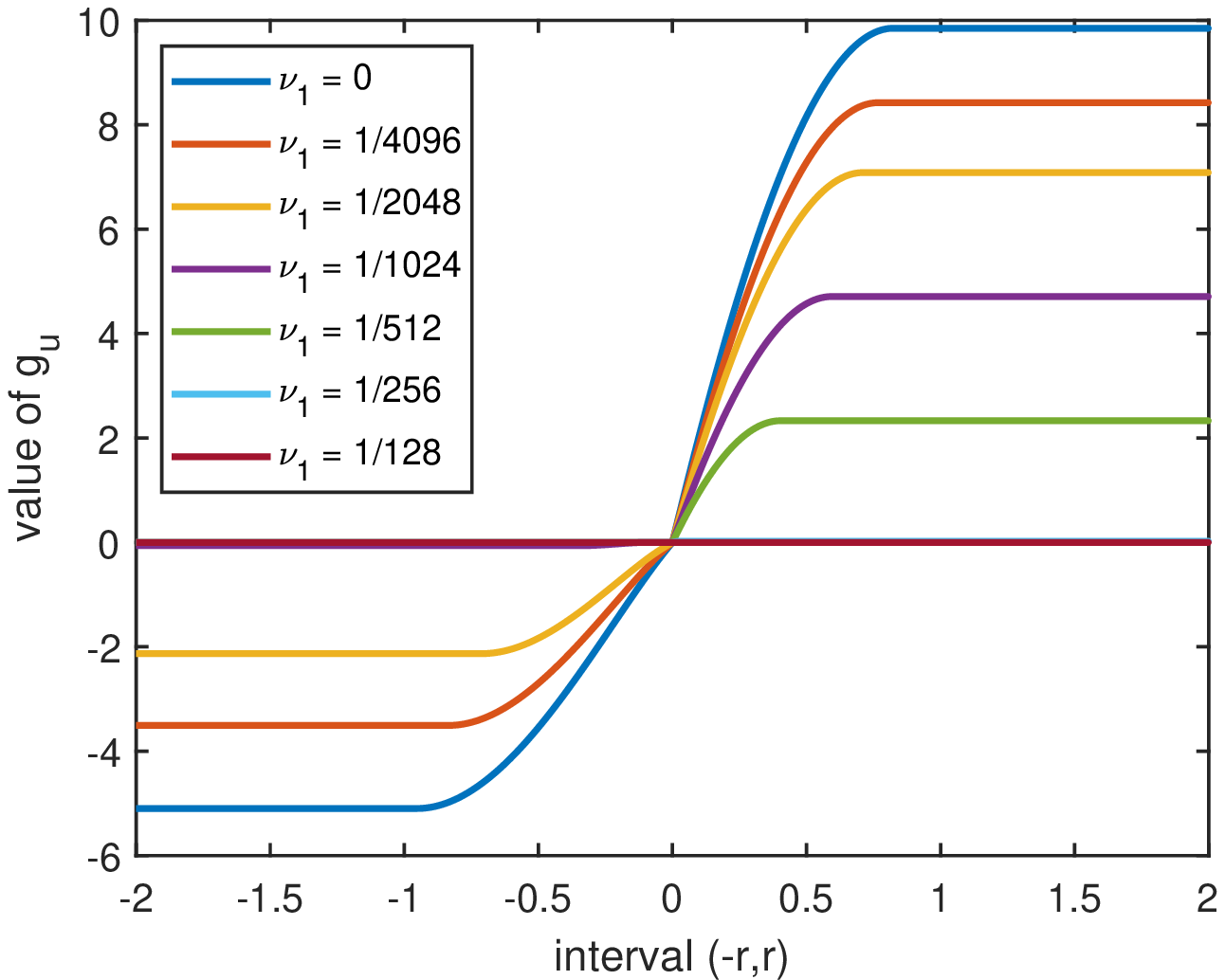}} 
\vspace{0.2cm}
\caption{Approximations of the optimal control $\bar \uu$ (left)
and the superposition function $g_{E_r(\bar \uu)}$ (right)
obtained from \cref{alg:gradproj} in the situation of Example \ExB{} for different values of $\nu_1$ 
and $h_y = h_\uu = 1/512$. The legend refers to both figures.
It can be seen that the $L^1$-regularization
promotes sparsity properties; cf.\ the stationarity system \eqref{eq:pdsysPr}.
For the highest value $\nu_1 = 1/128$, the optimal control is zero.}
\label{fig:Ex2-1}
\end{center}
\end{figure}
\vspace{-0.5cm}
\begin{figure}[H]
\begin{center}
\subfigure{\includegraphics[width=7.5cm,height=8.25cm,keepaspectratio]{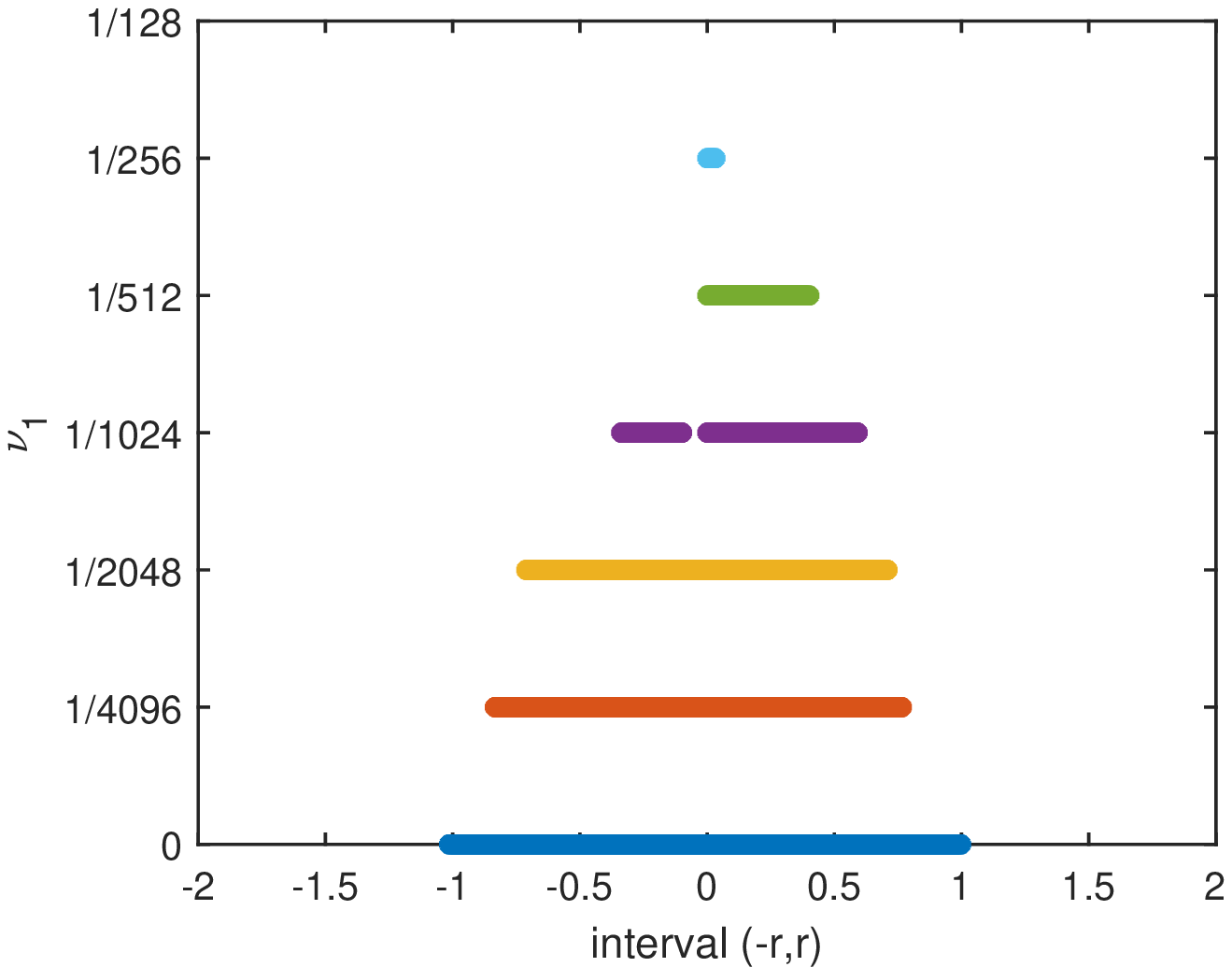}} 
\hspace{0.3cm}
\subfigure{\includegraphics[width=7.7cm,height=8.25cm,keepaspectratio]{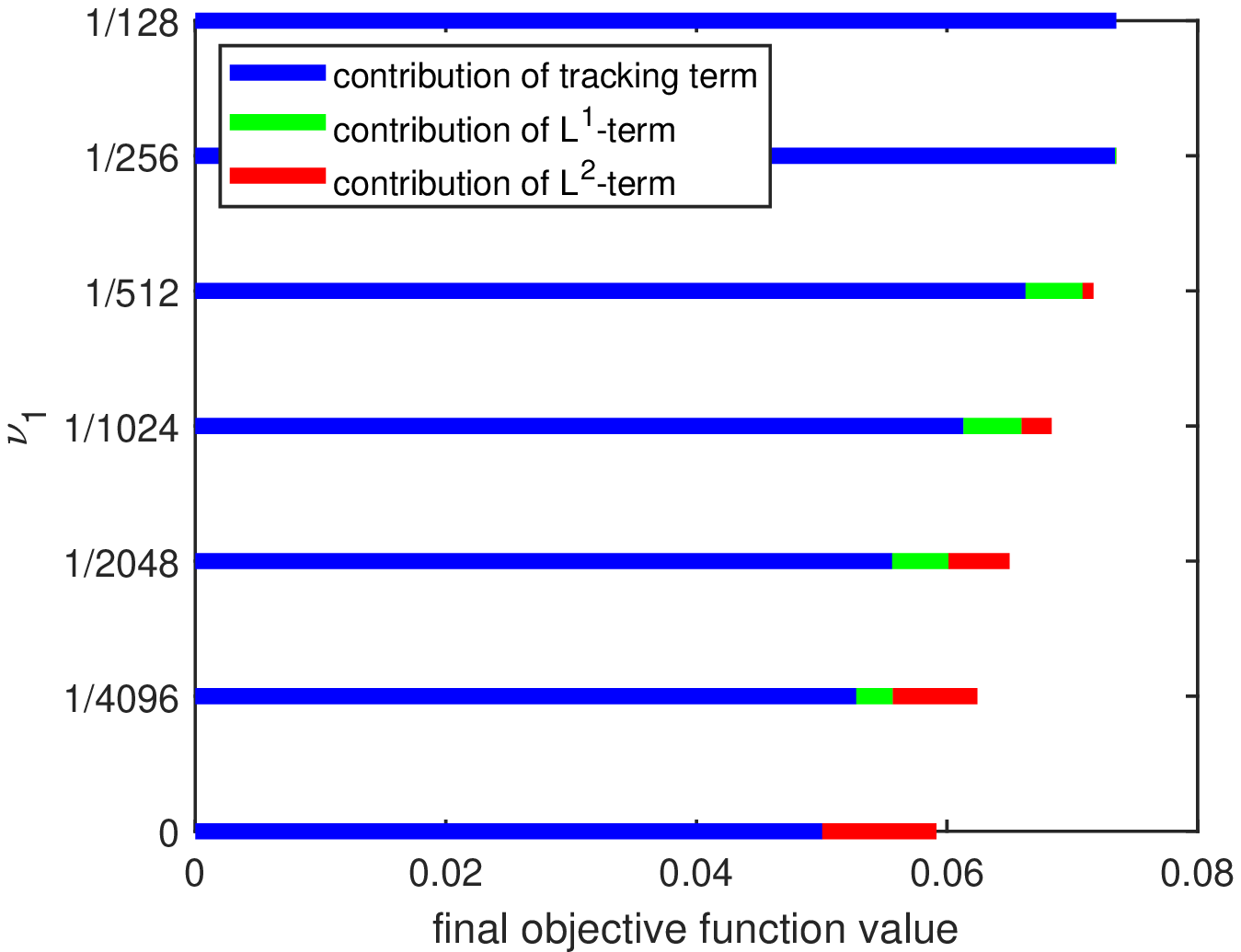}} 
\vspace{0.1cm}
\caption{Depiction of the (essential) support of the optimal control $\bar \uu$ (left)
and contributions of the tracking term and the $L^1(-r,r)$- and $L^2(-r,r)$-regularization terms
to the final objective function value (right) in the situation of \Cref{fig:Ex2-1}. 
It can be seen that the essential support of the optimal control $\bar \uu$ shrinks as 
$\nu_1$ increases. For the highest value $\nu_1 = 1/128$, it vanishes.}
\label{fig:Ex2-2}
\end{center}
\end{figure}

\begin{table}[H]
\begin{tabular}{c | c c c c c c }
\hline\noalign{\smallskip}
\diagbox{$h_y$}{$h_\uu$}  & $1/32$ & $1/64$ & $1/128$ & $1/256$ & $1/512$  & 1/1024 
\\
\noalign{\smallskip}\hline\noalign{\smallskip}
 \begin{tabular}{@{}c@{}}$1/64$ \end{tabular} 
& $45$    & $\infty$    &  $46$  & $47$  & $49$  & $48$ 
\\
\noalign{\smallskip}\hline\noalign{\smallskip}
 \begin{tabular}{@{}c@{}}$1/128$ \end{tabular} 
& $46$    & $46$    &  $46$  & $46$  & $46$  & $46$ 
\\
\noalign{\smallskip}\hline\noalign{\smallskip}
 \begin{tabular}{@{}c@{}}$1/256$ \end{tabular} 
& $46$    & $46$    &  $46$  & $46$  & $46$  & $46$ 
\\
\noalign{\smallskip}\hline\noalign{\smallskip}
 \begin{tabular}{@{}c@{}}$1/512$ \end{tabular} 
& $46$    & $46$    &  $46$  & $46$  & $46$  & $46$ 
\\
\noalign{\smallskip}\hline\noalign{\smallskip}
 \begin{tabular}{@{}c@{}}$1/1024$ \end{tabular} 
& $46$    & $46$    &  $46$  & $46$  & $46$  & $46$ 
\\
\noalign{\smallskip}\hline\noalign{\smallskip}
\end{tabular}
\vspace{0.4cm}
\caption{Number of gradient projection steps needed to drive the stationarity measure 
$\Theta_{\epsilon_1}$ below the tolerance $\epsilon_2 = 10^{-8}$ in the situation 
of Example \ExB{} for different mesh widths $h_y$ and $h_\uu$ and $\nu_1 = 1/1024$. It can again 
be seen that 
the convergence is mesh-independent for all $h_y$ that are 
sufficiently small. Compare with \Cref{fig:Ex1-1} and \Cref{tab:Ex1Table2} in this context.}
\label{tab:Ex2Table1}   
\end{table} 
~\\[-1cm]

As \cref{tab:Ex2Table1} shows,
when applied to Example \ExB{},
\cref{alg:gradproj} exhibits an asymptotically mesh-independent convergence behavior 
that is 
analogous to that observed for Example \ExA{} in \Cref{fig:Ex1-1,tab:Ex1Table2}. 
In \Cref{fig:Ex2-1}, one can further see that the optimal superposition functions $g_{E_r(\bar \uu)}$ calculated in Example \ExB{} 
are indeed ReLU-like and ``tight'' in the sense that they are sigmoidal, continuously differentiable away from the origin, 
possess a distinct kink at zero, and have a vanishing derivative outside of the range of the optimal state. 
This confirms the predictions made by \cref{cor:tightytight,cor:actisim} and shows that \eqref{eq:P} indeed promotes properties that are 
present in popular activation functions from machine learning. 
In \Cref{fig:Ex2-1,fig:Ex2-2}, one can also 
see that the $L^1$-regularization term in the objective function of \eqref{eq:modprobnum} 
causes the optimal controls to be sparse in the sense that they are supported on a set 
whose measure decreases as $\nu_1$ increases. This behavior is in good accordance with the 
influence that the parameter $\nu_1$ has on the system \eqref{eq:pdsysPr}; cf.\ 
\cite{Stadler2009}. We remark that
the derivation of a-priori error estimates for problems of the type \eqref{eq:modprobnum} 
is an interesting topic. We leave this field for future research. 


\bibliographystyle{plain}
\bibliography{references}

\begin{thebibliography}{10}

\bibitem{Aarset2023}
C.~Aarset, M.~Holler, and T.~T.~N. Nguyen.
\newblock Learning-informed parameter identification in nonlinear
  time-dependent {PDE}s.
\newblock {\em Appl.~Math.~Optim.}, 88(3):Paper No. 76, 53, 2023.

\bibitem{Ambrosio2000}
L.~Ambrosio, N.~Fusco, and D.~Pallara.
\newblock {\em Functions of Bounded Variation and Free Discontinuity Problems}.
\newblock Oxford University Press, Oxford \& New York, 2000.

\bibitem{Attouch2006}
H.~Attouch, G.~Buttazzo, and G.~Michaille.
\newblock {\em Variational Analysis in {S}obolev and {BV} Spaces}.
\newblock SIAM, Philadelphia, 2006.

\bibitem{Barbu1984}
V.~Barbu.
\newblock {\em Optimal Control of Variational Inequalities}.
\newblock Research Notes in Mathematics. Pitman, 1984.

\bibitem{Barbu1996}
V.~Barbu and K.~Kunisch.
\newblock Identification of nonlinear elliptic equations.
\newblock {\em Appl.~Math.~Optim.}, 33(2):139--167, 1996.

\bibitem{Barbu1994}
V.~Barbu, P.~Neittaanm{\"a}ki, and A.~Niemist{\"o}.
\newblock A penalty method for the identification of nonlinear elliptic
  differential operator.
\newblock {\em Numer.~Funct.~Anal.~Optim.}, 15(5-6):503--530, 1994.

\bibitem{Beck2014}
A.~Beck.
\newblock {\em Introduction to Nonlinear Optimization}.
\newblock MOS/SIAM Series on Optimization. SIAM, 2014.

\bibitem{Bocharev2007}
V.~I. Bogachev.
\newblock {\em Measure Theory}.
\newblock Springer, 2007.

\bibitem{BonnansShapiro2000}
J.~F. Bonnans and A.~Shapiro.
\newblock {\em Perturbation Analysis of Optimization Problems}.
\newblock Springer Series in Operations Research. Springer, New York, 2000.

\bibitem{Brunton2016}
S.~L. Brunton, J.~L. Proctor, and J.~N. Kutz.
\newblock Discovering governing equations from data by sparse identification of
  nonlinear dynamical systems.
\newblock {\em Proc.~Nat.~Acad.~Sci.}, 113(15):3932--3937, 2016.

\bibitem{ChristofPhd2018}
C.~Christof.
\newblock {\em Sensitivity Analysis of Elliptic Variational Inequalities of the
  First and the Second Kind}.
\newblock PhD thesis, Technische Universit{\"a}t Dortmund, 2018.

\bibitem{Christof2020-TV}
C.~Christof.
\newblock Gradient-based solution algorithms for a class of bilevel
  optimization and optimal control problems with a nonsmooth lower level.
\newblock {\em SIAM J.~Optim.}, 30(1):290--318, 2020.

\bibitem{Christof2022}
C.~Christof and J.~Kowalczyk.
\newblock On the omnipresence of spurious local minima in certain neural
  network training problems.
\newblock {\em Constr.~Approx}, 2023.
\newblock to appear.

\bibitem{Christof2018nonsmoothPDE}
C.~Christof, C.~Meyer, S.~Walther, and C.~Clason.
\newblock Optimal control of a non-smooth semilinear elliptic equation.
\newblock {\em Math.~Control Relat.~Fields}, 8(1):247--276, 2018.

\bibitem{Christof2021}
C.~Christof and G.~M\"uller.
\newblock Multiobjective optimal control of a non-smooth semilinear elliptic
  partial differential equation.
\newblock {\em ESAIM Control Optim.~Calc.~Var.}, 27, 2021.
\newblock Art.~S13.

\bibitem{Clason2021}
C.~Clason, V.~H. Nhu, and A.~R\"osch.
\newblock No-gap second-order optimality conditions for optimal control of a
  non-smooth quasilinear elliptic equation.
\newblock {\em ESAIM Control Optim.~Calc.~Var.}, 27, 2021.
\newblock Art.~62.

\bibitem{Cohen2018OnAW}
M.~B. Cohen, J.~Diakonikolas, and L.~Orecchia.
\newblock On acceleration with noise-corrupted gradients.
\newblock In {\em Proceedings of the 35th International Conference on Machine
  Learning (PMLR 80)}, pages 1019--1028, Stockholm, Sweden, 2018.

\bibitem{Court2022}
S.~Court and K.~Kunisch.
\newblock Design of the monodomain model by artificial neural networks.
\newblock {\em Discrete Contin.~Dyn.~Syst.}, 42(12):6031--6061, 2022.

\bibitem{Dong2022-3}
G.~Dong, M.~Hinterm{\"u}ller, and K.~Papafitsoros.
\newblock A descent algorithm for the optimal control of {ReLU} neural network
  informed {PDEs} based on approximate directional derivatives.
\newblock arXiv:2210.07900v1, 2022.

\bibitem{Dong2022-1}
G.~Dong, M.~Hinterm{\"u}ller, and K.~Papafitsoros.
\newblock Optimization with learning-informed differential equation constraints
  and its applications.
\newblock {\em ESAIM Control Optim.~Calc.~Var.}, 28:3, 2022.

\bibitem{Dong2022-2}
G.~Dong, M.~Hinterm{\"u}ller, K.~Papafitsoros, and K.~Völkner.
\newblock First-order conditions for the optimal control of learning-informed
  nonsmooth {PDEs}.
\newblock arXiv:2206.00297v2, 2022.

\bibitem{Drabek2007}
P.~Drábek and J.~Milota.
\newblock {\em Methods of Nonlinear Analysis: Applications to Differential
  Equations}.
\newblock Birkhäuser Verlag, 2007.

\bibitem{Gilbarg2001}
D.~Gilbarg and N.~S. Trudinger.
\newblock {\em Elliptic Partial Differential Equations of Second Order}.
\newblock Springer, 2001.

\bibitem{Goebel1998}
M.~Goebel.
\newblock Smooth and nonsmooth optimal {L}ipschitz control --- a model problem.
\newblock In W.~H. Schmidt, K.~Heier, L.~Bittner, and R.~Bulirsch, editors,
  {\em Variational Calculus, Optimal Control and Applications}, pages 53--60,
  Basel, 1998. Birkh{\"a}user Basel.

\bibitem{Grisvard1985}
P.~Grisvard.
\newblock {\em Elliptic Problems in Nonsmooth Domains}.
\newblock Pitman, 1985.

\bibitem{Grutzner2021}
S.~Gr{\"u}tzner and A.~Muntean.
\newblock Identifying processes governing damage evolution in quasi-static
  elasticity part 1 – analysis.
\newblock {\em Adv.~Math.~Sci.~Appl.}, 30(2):305--334, 2021.

\bibitem{Hao2015}
D.~N. H{\`a}o, B.~V. Huong, P.~X. Thanh, and D.~Lesnic.
\newblock Identification of nonlinear heat transfer laws from boundary
  observations.
\newblock {\em Appl.~Anal.}, 94(9):1784--1799, 2015.

\bibitem{Hertlein2019}
L.~Hertlein and M.~Ulbrich.
\newblock An inexact bundle algorithm for nonconvex nonsmooth minimization in
  {H}ilbert space.
\newblock {\em SIAM J.~Control Optim.}, 57(5):3137--3165, 2019.

\bibitem{Hinze2012}
M.~Hinze and A.~R{\"o}sch.
\newblock Discretization of optimal control problems.
\newblock In G.~Leugering, S.~Engell, A.~Griewank, M.~Hinze, R.~Rannacher,
  V.~Schulz, M.~Ulbrich, and S.~Ulbrich, editors, {\em Constrained Optimization
  and Optimal Control for Partial Differential Equations}, pages 391--430,
  Basel, 2012. Springer Basel.

\bibitem{Josephy1981}
M.~Josephy.
\newblock Composing functions of bounded variation.
\newblock {\em Proc.~Amer.~Math.~Soc}, 83(2):354--356, 1981.

\bibitem{Kaltenbacher2022}
B.~Kaltenbacher and T.~T.~N. Nguyen.
\newblock Discretization of parameter identification in {PDEs} using neural
  networks.
\newblock {\em Inverse Problems}, 38(12):124007, 2022.

\bibitem{Kaltenbacher2020}
B.~Kaltenbacher and W.~Rundell.
\newblock The inverse problem of reconstructing reaction-diffusion systems.
\newblock {\em Inverse Problems}, 36(6):065011, 2020.

\bibitem{Kian2023}
Y.~Kian.
\newblock Lipschitz and {H}{\"o}lder stable determination of nonlinear terms
  for elliptic equations.
\newblock {\em Nonlinearity}, 36(2):1302, 2023.

\bibitem{KinderlehrerStampacchia1980}
D.~Kinderlehrer and G.~Stampacchia.
\newblock {\em An Introduction to Variational Inequalities and Their
  Applications}, volume~31 of {\em Classics in Applied Mathematics}.
\newblock SIAM, 2000.

\bibitem{Mityagin2020}
B.~S. Mityagin.
\newblock The zero set of a real analytic function.
\newblock {\em Math.~Notes}, 107(3):529--530, 2020.

\bibitem{Monteiro2019}
G.~A. Monteiro, A.~Slav\'{\i}k, and M.~Tvrd\'{y}.
\newblock {\em {K}urzweil-{S}tieltjes Integral: Theory and Applications}.
\newblock Number~15 in Series in Real Analysis. World Scientific, Singapore,
  2019.

\bibitem{Qin2019}
T.~Qin, K.~Wu, and D.~Xiu.
\newblock Data driven governing equations approximation using deep neural
  networks.
\newblock {\em J.~Comput.~Phys.}, 395:620--635, 2019.

\bibitem{Roesch1994}
A.~R{\"o}sch.
\newblock Identification of nonlinear heat transfer laws by optimal control.
\newblock {\em Num.~Funct.~Anal.~Optim.}, 15(3-4):417--434, 1994.

\bibitem{Roesch1996}
A.~R{\"o}sch.
\newblock Fr\'{e}chet differentiability of the solution of the heat equation
  with respect to a nonlinear boundary condition.
\newblock {\em Z.~Anal.~Anwend.}, 15(3):603--618, 1996.

\bibitem{Roesch1998}
A.~R{\"o}sch.
\newblock Second order optimality conditions and stability estimates for the
  identification of nonlinear heat transfer laws.
\newblock In W.~Desch, F.~Kappel, and K.~Kunisch, editors, {\em Control and
  Estimation of Distributed Parameter Systems}, pages 237--246, Basel, 1998.
  Birkh{\"a}user Basel.

\bibitem{Roesch2002}
A.~R{\"o}sch.
\newblock A {G}auss-{N}ewton method for the identification of nonlinear heat
  transfer laws.
\newblock In K.-H. Hoffmann, I.~Lasiecka, G.~Leugering, J.~Sprekels, and
  F.~Tr{\"o}ltzsch, editors, {\em Optimal Control of Complex Structures}, pages
  217--230, Basel, 2002. Birkh{\"a}user Basel.

\bibitem{Roesch1992}
A.~R{\"o}sch and F.~Tröltzsch.
\newblock An optimal control problem arising from the identification of
  nonlinear heat transfer laws.
\newblock {\em Arch.~Control~Sci.}, 1:4--183, 1992.

\bibitem{Rudy2017}
S.~H. Rudy, S.~L. Brunton, J.~L. Proctor, and J.~N. Kutz.
\newblock Data-driven discovery of partial differential equations.
\newblock {\em Sci.~Adv.}, 3(4):e1602614, 2017.

\bibitem{Roesch1996-2}
A.~Rösch.
\newblock Stability estimates for the identification of nonlinear heat transfer
  laws.
\newblock {\em Inverse Problems}, 12(5):743--756, 1996.

\bibitem{Schweizer2013}
B.~Schweizer.
\newblock {\em Partielle Differentialgleichungen}.
\newblock Springer-Verlag, Berlin/Heidelberg, 2013.

\bibitem{Stadler2009}
G.~Stadler.
\newblock Elliptic optimal control problems with {L1}-control cost and
  applications for the placement of control devices.
\newblock {\em Comput.~Optim.~Appl.}, 44(159), 2009.

\bibitem{Troeltzsch2010}
F.~Tr{\"o}ltzsch.
\newblock {\em Optimal Control of Partial Differential Equations}.
\newblock AMS, 2010.

\bibitem{Ziemer1989}
W.~P. Ziemer.
\newblock {\em Weakly Differentiable Functions}.
\newblock Springer Verlag, New York, 1989.

\end{thebibliography}

\end{document}